%% file: main.tex
\documentclass{amsart}
\usepackage[utf8]{inputenc}
\usepackage{graphicx}
\usepackage[dvipsnames]{xcolor}
\usepackage{amsmath,latexsym}
\usepackage{amsfonts,amssymb}
\usepackage{amscd, amsthm}
\usepackage{stmaryrd}
\usepackage{fancyhdr}
\usepackage{commath,enumitem,chngcntr}
\usepackage{indentfirst}
\usepackage[titletoc,toc,title]{appendix}
\usepackage{bbm}
\usepackage[top=1in, bottom=1in, left=1.3in, right=1.3in]{geometry}
\usepackage{pst-node}
\usepackage{tikz}
\usepackage{tikz-cd}
\usetikzlibrary{
  arrows,
  arrows.meta,
  calc,
  positioning,
  decorations.pathreplacing
}

\usepackage{float}
\usepackage{mathrsfs}
\usepackage{overpic}
\usepackage{natbib}          %
\usepackage{faktor}
\usepackage{mathtools}
\usepackage{booktabs,tabularx,array}
\mathtoolsset{showonlyrefs}
\usepackage{hyperref}
\usepackage{subcaption}

\usepackage{lineno}

\usetikzlibrary{
  arrows,
  arrows.meta,
  calc,
  positioning,
  decorations.pathreplacing
}
\usepackage{float}
\usepackage{subcaption}
\newcommand{\CasePanelWidth}{.99}

\newtheorem{theorem}{Theorem}[section]
\newtheorem{lemma}[theorem]{Lemma}

\newtheorem{definition}[theorem]{Definition}

\newtheorem{proposition}[theorem]{Proposition}
\newtheorem{corollary}[theorem]{Corollary}
\newtheorem{remark}[theorem]{Remark}
\newtheorem{conjecture}[theorem]{Conjecture}
\DeclareMathOperator{\Mod}{Mod}

\def\C{\mathbb{C}}
\def\R{\mathbb{R}}
\def\H{\mathbb{H}}
\def\T{\mathcal{T}}
\def\D{\mathbb{D}}
\def\SS{\mathbb{S}}
\def\E{\mathbb{E}}
\def\Z{\mathbb{Z}}
\def\P{\mathcal{P}}
\def\K{\mathbb{K}}

\newcommand{\dmo}{\DeclareMathOperator}
\dmo{\Isom}{Isom}
\dmo{\Dev}{Dev}
\dmo{\arctanh}{arctanh}
\dmo{\Id}{Id}

\dmo{\inte}{int}
\dmo{\st}{st}

\definecolor{ground}{HTML}{F3F5F6}
\definecolor{groundedge}{HTML}{A7B0B7}
\definecolor{foot}{HTML}{DDEEF4}
\definecolor{surface}{HTML}{77B8C5}
\definecolor{surfaceedge}{HTML}{23697A}
\definecolor{vertex}{HTML}{164B5B}
\definecolor{rayred}{HTML}{C9342D}
\definecolor{construct}{HTML}{3B70AD}
\definecolor{muted}{HTML}{75818A}

\tikzset{
  ground/.style={fill=ground,draw=groundedge!82,line width=.40pt},
  footprint/.style={fill=foot,fill opacity=.72,draw=surfaceedge!42,line width=.45pt},
  surface cell/.style={fill=surface,fill opacity=.42,draw=none},
  surface cell alt/.style={fill=surface!72!white,fill opacity=.42,draw=none},
  face/.style={fill=surface!38!white,fill opacity=.14,draw=surfaceedge!56,line width=.42pt},
  mesh/.style={draw=surfaceedge!78,line width=.34pt,opacity=.42},
  surface rim/.style={draw=surfaceedge,line width=.82pt,opacity=.96},
  hull edge/.style={draw=surfaceedge,line width=.70pt},
  hidden hull edge/.style={draw=surfaceedge,line width=.43pt,densely dashed,opacity=.55},
  vpoint/.style={circle,fill=vertex,fill opacity=.82,draw=none,inner sep=.88pt},
  cpoint/.style={circle,fill=construct,draw=white,line width=.15pt,inner sep=1.09pt},
  rpoint/.style={circle,fill=rayred,draw=white,line width=.15pt,inner sep=1.23pt},
  red ray/.style={draw=rayred,line width=1.02pt,-{Latex[length=2.2mm,width=1.5mm]}},
  red connector/.style={draw=rayred,line width=.66pt,densely dashed},
  red curve/.style={draw=rayred,line width=.96pt},
  red curve end/.style={draw=rayred,line width=.96pt,-{Latex[length=2.1mm,width=1.4mm]}},
  red segment/.style={draw=rayred,line width=1.08pt,densely dashed,-{Latex[length=2.1mm]}},
  construction/.style={draw=construct,line width=.68pt,-{Latex[length=1.85mm,width=1.25mm]}},
  guide/.style={draw=muted,line width=.50pt,densely dashed,opacity=.68},
  continuation/.style={draw=surfaceedge!82,line width=.58pt,-{Latex[length=1.85mm,width=1.25mm]}},
  axis/.style={draw=muted,line width=.55pt,-{Latex[length=1.8mm,width=1.25mm]}},
  lab/.style={font=\footnotesize,inner sep=.25pt},
  slab/.style={font=\footnotesize,inner sep=.15pt},
  redlab/.style={font=\footnotesize,text=rayred,inner sep=.25pt}
}

\newcommand{\PanelBoxA}{\path[use as bounding box] (-5.25,-1.30) rectangle (5.25,2.15);}
\newcommand{\PanelBoxB}{\path[use as bounding box] (-5.25,-1.30) rectangle (5.25,4.60);}
\newcommand{\PanelBoxC}{\path[use as bounding box] (-5.25,-1.30) rectangle (5.25,4.18);}
\newcommand{\PanelBoxD}{\path[use as bounding box] (-5.25,-1.30) rectangle (5.25,4.42);}
\newcommand{\PanelBoxE}{\path[use as bounding box] (-5.25,-1.30) rectangle (5.25,4.10);}
\newcommand{\PanelBoxF}{\path[use as bounding box] (-5.25,-1.30) rectangle (5.25,4.12);}
\newcommand{\PanelBoxG}{\path[use as bounding box] (-5.25,-1.30) rectangle (5.25,4.16);}
\newcommand{\ProcessPanelBox}{\path[use as bounding box] (-7.02,-1.28) rectangle (5.18,3.15);}
\newcommand{\BasePlane}{%
  \path[ground] (-3.55,-2.20,0)--(3.55,-2.20,0)--(3.55,2.20,0)--(-3.55,2.20,0)--cycle;
  \draw[axis] (-3.35,-1.90,0)--(-3.35,-1.90,.85) node[left,slab] {$z$};
}
\newcommand{\VPoint}[2]{\node[vpoint] at (#1,#2,0) {};}
\newcommand{\RPoint}[3]{\node[rpoint] at (#1,#2,#3) {};}
\newcommand{\CPoint}[2]{\node[cpoint] at (#1,#2,0) {};}
\newcommand{\RedRayTop}{3.57}

\newcommand{\LatticeTile}[2]{%
  \foreach \ii in {0,1,2,3,4,5}{%
    \foreach \jj in {0,1,2,3,4,5}{%
      \pgfmathsetmacro{\xa}{#1+\ii/6}
      \pgfmathsetmacro{\xb}{#1+(\ii+1)/6}
      \pgfmathsetmacro{\ya}{#2+\jj/6}
      \pgfmathsetmacro{\yb}{#2+(\jj+1)/6}
      \pgfmathsetmacro{\za}{sqrt(max(0,.5-(\xa-(#1+.5))^2-(\ya-(#2+.5))^2))}
      \pgfmathsetmacro{\zb}{sqrt(max(0,.5-(\xb-(#1+.5))^2-(\ya-(#2+.5))^2))}
      \pgfmathsetmacro{\zc}{sqrt(max(0,.5-(\xb-(#1+.5))^2-(\yb-(#2+.5))^2))}
      \pgfmathsetmacro{\zd}{sqrt(max(0,.5-(\xa-(#1+.5))^2-(\yb-(#2+.5))^2))}
      \path[surface cell] (\xa,\ya,\za)--(\xb,\ya,\zb)--(\xb,\yb,\zc)--(\xa,\yb,\zd)--cycle;
    }%
  }%
}

\newcommand{\StripTile}[1]{%
  \foreach \ii in {0,1,2,3,4,5,6,7}{%
    \foreach \jj in {0,1,2,3,4,5,6,7}{%
      \pgfmathsetmacro{\xa}{#1+\ii/8}
      \pgfmathsetmacro{\xb}{#1+(\ii+1)/8}
      \pgfmathsetmacro{\ya}{-1+2*\jj/8}
      \pgfmathsetmacro{\yb}{-1+2*(\jj+1)/8}
      \pgfmathsetmacro{\za}{sqrt(max(0,1.25-(\xa-(#1+.5))^2-\ya^2))}
      \pgfmathsetmacro{\zb}{sqrt(max(0,1.25-(\xb-(#1+.5))^2-\ya^2))}
      \pgfmathsetmacro{\zc}{sqrt(max(0,1.25-(\xb-(#1+.5))^2-\yb^2))}
      \pgfmathsetmacro{\zd}{sqrt(max(0,1.25-(\xa-(#1+.5))^2-\yb^2))}
      \path[surface cell] (\xa,\ya,\za)--(\xb,\ya,\zb)--(\xb,\yb,\zc)--(\xa,\yb,\zd)--cycle;
    }%
  }%
}

\newcommand{\StripWall}[2]{%
  \path[face] (#1,#2,3.45)--({#1+1},#2,3.45)
    --plot[domain=0:180,samples=28] ({#1+.5+.5*cos(\x)},{#2},{.5*sin(\x)})--cycle;
  \draw[surface rim] plot[domain=0:180,samples=28]
    ({#1+.5+.5*cos(\x)},{#2},{.5*sin(\x)});
}

\newcommand{\WedgeWall}[2]{%
  \pgfmathsetmacro{\aa}{.62}
  \pgfmathsetmacro{\sgn}{#2}
  \path[face,fill opacity=.24] (#1,{\sgn*\aa*#1},3.45)--({#1+1},{\sgn*\aa*(#1+1)},3.45)
    --plot[domain=0:180,samples=28]
      ({#1+.5+.5*cos(\x)},{\sgn*\aa*(#1+.5+.5*cos(\x))},{.5*sqrt(1+\aa^2)*sin(\x)})--cycle;
  \draw[surface rim] plot[domain=0:180,samples=28]
    ({#1+.5+.5*cos(\x)},{\sgn*\aa*(#1+.5+.5*cos(\x))},{.5*sqrt(1+\aa^2)*sin(\x)});
}

\newcommand{\CurtainCell}[1]{%
  \path[face] (#1,0,3.45)--({#1+1},0,3.45)
    --plot[domain=0:180,samples=30] ({#1+.5+.5*cos(\x)},0,{.5*sin(\x)})--cycle;
  \draw[surface rim] plot[domain=0:180,samples=30]
    ({#1+.5+.5*cos(\x)},0,{.5*sin(\x)});
}

\begin{document}

\title{Discrete uniformization of polyhedral surfaces}

\author{Feng Luo$^{\dagger}$}
\thanks{$\dagger$Corresponding author's email: fluo@math.rutgers.edu.}
\thanks{Contributing authors' emails: yanwen.luo@ustb.edu.cn; zhenghao.rao@rutgers.edu; xrzhao24@m.fudan.edu.cn}
\address{Department of Mathematics, Rutgers University--New Brunswick, Piscataway, NJ 08854, USA}
\email{fluo@math.rutgers.edu  }

\author{Yanwen Luo}
\address{School of Mathematics and Physics, University of Science and Technology Beijing, Beijing, 100083, P.R. China}
\email{yanwen.luo@ustb.edu.cn}

\author{Zhenghao Rao}
\address{Department of Mathematics, Rutgers University--New Brunswick, Piscataway, NJ 08854, USA}
\email{zhenghao.rao@rutgers.edu}

\author{Xinrong Zhao}
\address{School of Mathematical Sciences, Fudan University, Shanghai, 200433, P.R. China}
\email{xrzhao24@m.fudan.edu.cn}



\begin{abstract}
The main result of the paper shows that each connected polyhedral surface with a hyperbolic background metric, or a Euclidean background metric with uniformly bounded circumdisk radii, is discrete conformal to a complete constant-curvature Riemannian surface equipped with a nonempty closed discrete subset. We also prove a discrete Riemann mapping theorem. The proofs are based on the recent work on the discrete Schwarz lemma, the discrete Liouville theorem for polyhedral surfaces, and a Weyl-type realization theorem for hyperbolic surfaces. 
\end{abstract}

\maketitle

\section{Introduction}

The classical uniformization theorem states that any connected Riemannian surface, whether it is compact or not,   is conformally diffeomorphic to a complete Riemannian surface of constant curvature $-1, 0$, or $1$.  Furthermore, two complete Riemannian surfaces of constant curvature are conformally diffeomorphic if and only if they are related by a Möbius transformation \footnote{A Möbius transformation is either a holomorphic or an anti-holomorphic linear fractional transformation. A surface with a constant-curvature metric admits special charts whose transition functions are Möbius transformations. A Möbius transformation between two Riemannian surfaces with constant-curvature metrics is a diffeomorphism whose coordinate representations in these special charts are Möbius.  The Schwarz lemma implies that a Möbius transformation between two complete hyperbolic surfaces is an isometry.}. 
Figure \ref{standardunif} illustrates the theorem. 
 
\begin{figure}[h!]
\begin{center}
\includegraphics[scale=0.3]{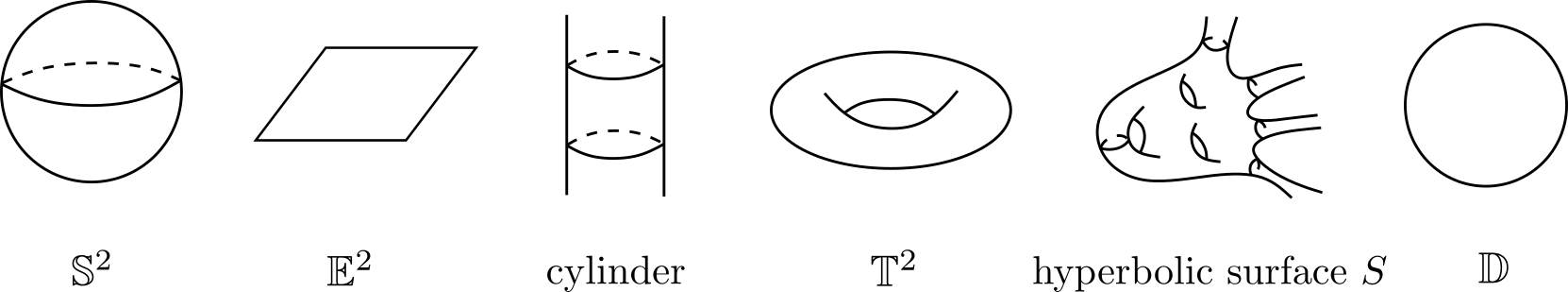}
\caption{The complete Riemannian surfaces of constant curvature in the classical uniformization theorem. }
\label{standardunif}
\end{center}
\end{figure}
 \vspace{0.2cm}
\begin{figure}[ht!]  
\begin{center}
\includegraphics[scale=0.27]{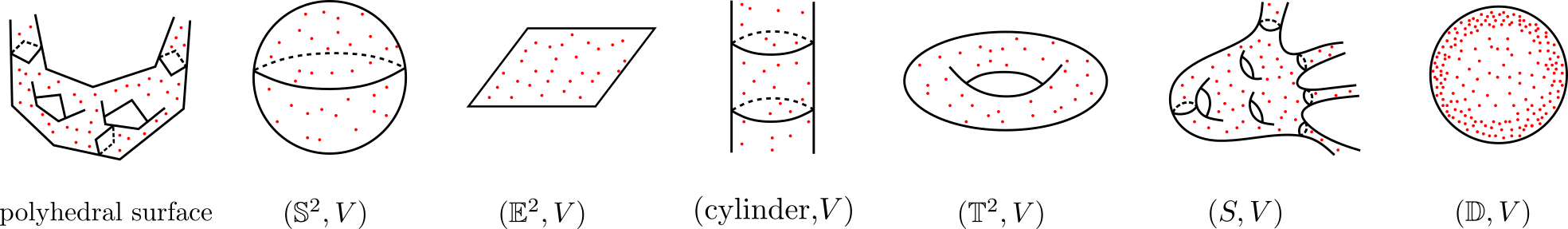}
\caption{The standard models of polyhedral surfaces in the discrete uniformization theorem.} 
\label{unifc}
\end{center}
\end{figure}

The main result of this paper establishes an analogous theorem for polyhedral surfaces in discrete conformal geometry, formulated in terms of standard model surfaces. See Figure \ref{unifc}.

\begin{theorem} \label{dut}
 (i) Every connected polyhedral surface with a hyperbolic background metric is discrete conformal to a standard model surface.

(ii) Every connected polyhedral surface with a Euclidean background metric whose empty-disks have uniformly bounded diameters is discrete conformal to a standard model surface.   Moreover, the uniform-bound hypothesis cannot be omitted in general. 

(iii) Every connected polyhedral surface with a spherical background metric whose empty-disks have diameters uniformly bounded away from $\pi$ is discrete conformal to a standard model surface. Moreover, the uniform-bounded away from $\pi$ hypothesis cannot be omitted in general.  

(iv) Two standard model surfaces are discrete conformal if and only if they are related by a Möbius transformation. 
\end{theorem}

Here, an \underline{empty-disk} in a polyhedral surface $(S, V,  d)$ is an isometric immersion of an open round disk into the punctured surface $(S-V, d|_{S-V})$.
When the polyhedral surface is homeomorphic to the 2-sphere, Theorem \ref{dut} is due to Rivin.
See Theorem \ref{rivin}.  Theorem \ref{dut} is known to be true for closed surfaces of high genus with Euclidean or hyperbolic background metrics, through work in \cite{filla2008,  gglsw, glsw}. 
The precise definitions of polyhedral surfaces, standard models, and discrete conformal equivalence are given in \S 2. 
We will briefly describe the first two notions.
A \underline{standard model surface} is a complete Riemannian surface of constant curvature endowed with a nonempty closed discrete subset. A \underline{polyhedral surface} $(S,V,d)$ consists of a surface equipped with a nonempty closed discrete subset $V$ and a cone metric $d$ whose cone singularities are contained in $V$, 
such that $(S, V, d)$
is obtained by isometrically gluing geometric triangles\footnote{A spherical triangle is assumed to be 
contained in an open hemisphere.} in the unit $2$-sphere
$\mathbb{S}^2$, or the Euclidean plane $\E^2$, or the hyperbolic plane $\H^2$, subject to the \underline{Delaunay condition}. 
The Delaunay condition means that the circumdisks of geometric triangles in $S$ are
compact and that their interiors are empty-disks.
The resulting surfaces are called polyhedral surfaces with a spherical, or Euclidean, or hyperbolic background metric, respectively. 

Note that if the surface is closed, then a fundamental theorem in computational geometry shows that for any hyperbolic or Euclidean cone metric $d$ on $(S, V)$, $(S, V, d)$ is polyhedral.

Theorem \ref{dut} is a consequence of recent work on a Weyl-type realization theorem for hyperbolic surfaces \cite{luow2024},  the discrete Liouville theorem \cite{luo-luo-rao}, and the discrete Schwarz lemma \cite{Zhao2026rigidity, luo-luo}. 

\medskip
\noindent{\bf Discrete conformality and convex geometry in the hyperbolic 3-space.} Discrete conformal equivalence of polyhedral surfaces is closely related to the convex hull construction in the hyperbolic 3-space $\H^3$.  Let $X$ be a compact set in the Riemann sphere $\hat \C =\C \cup \{\infty\}$, regarded as the conformal infinity of the upper-half-space model $\H^3_U$ of the hyperbolic 3-space. Assume that $X$ contains at least three points. We use $C(X)$ to denote the convex hull of $X$ in hyperbolic 3-space and $\partial C(X)$ to denote its boundary, equipped with the induced path metric. Notice that if $C(X)$ is two-dimensional, its boundary $\partial C(X)$ is defined as its metric double in the sense of Alexandrov \cite{alexandrov-collected-papers-vol-2}. One usually calls the surface $\partial C(X)$ the \underline{Thurston dome} over the open set $\hat \C-X$. See Figure \ref{Thudome}.  A theorem of W. Thurston \cite{thurston, epsteinm2006} says that $\partial C(X)$ is a complete hyperbolic surface.  Suppose  $V_1, V_2$ are two closed discrete subsets of the unit disk $\D$ equipped with the hyperbolic metric $d_{\H}$.  We will see that two standard model surfaces $(\D, V_1, d_{\H})$ and $(\D, V_2, d_{\H})$  are \underline{discrete conformal} if and only if the boundaries of the convex hulls $C(V_i \cup \D^c)$ are isometric. See Figure \ref{fig:standardmodel}.

Our strategy to prove Theorem \ref{dut} follows the classical approach of passing to the universal cover.  Recall that the proof of the uniformization theorem goes as follows. Let $(\tilde{S}, \tilde{g})$ be the universal cover of a connected orientable Riemannian surface $(S, g)$ with the pull-back metric $\tilde{g}$. The fundamental group $\Gamma$ of $S$ acts on $(\tilde{S}, \tilde{g})$ isometrically with quotient $(S, g)$.  For the simply connected Riemannian surface $(\tilde{S}, \tilde{g})$, one proves two results. The first is the existence theorem, which states that ($\tilde{S}, \tilde{g})$ is conformal to the 2-sphere $\SS ^2$, or the Euclidean plane $\E^2$ or the hyperbolic plane $\H^2$. The existence theorem implies that the fundamental group $\Gamma$ acts conformally, freely, and discontinuously on $\mathbb {S} ^2$, $\mathbb {E}  ^2$, or $\mathbb {H}  ^2$ with quotient surface conformal to $(S,g)$. The second result says that conformal diffeomorphisms of $\mathbb {S} ^2$, $\E^2$, and $\H^2$ are Möbius transformations, or complex linear maps (i.e., similarity maps), or hyperbolic isometries, respectively.  These statements follow from the Liouville theorem and the Schwarz lemma and may be
regarded as the corresponding uniqueness results. With the existence and uniqueness theorems for simply connected surfaces, the uniformization theorem for all orientable Riemannian surfaces follows by observing that the fundamental group $\Gamma$ acts isometrically on $\E^2$ or $\H^2$. 

To prove Theorem \ref{dut} in the discrete setting, 
we need similar existence and uniqueness theorems for simply connected polyhedral surfaces.  In the case of the 2-sphere, the existence and uniqueness were proved in \cite{rivin1994}. 

\begin{theorem}[Rivin \cite{rivin1994}]
\label{rivin}
    Any finite-area complete hyperbolic metric on a punctured sphere is isometric to the boundary of an ideal hyperbolic polyhedron in the hyperbolic 3-space $\H^3$. Moreover, the ideal polyhedron is unique up to hyperbolic isometries.
\end{theorem}

For non-compact simply connected surfaces, the existence 
theorem was proved in \cite{luow2024}. Recall that a \underline{circle-type closed set} $Y \subset \mathbb S^2$ is a compact set each of whose connected components is either a round disk or a point. Consider the 2-sphere $\mathbb S^2$ as the conformal infinity of the Poincar\'e model $\H^3_P$ of the hyperbolic 3-space. 

\begin{theorem}[Existence Theorem \cite{luow2024}]\label{lw}
     Every complete hyperbolic surface of genus zero with countably many topological ends is isometric to $\partial C(Y)$ for a circle-type closed set $Y\subset \SS^2 = \partial\H^3_P$.
\end{theorem}

The uniqueness result for polyhedral surfaces with Euclidean background metrics, called the discrete Liouville theorem, is a consequence of recent work \cite{luo-luo-rao}.
\begin{theorem}[Discrete  Liouville Theorem \cite{luo-luo-rao}] \label{llr}
    Suppose $V_1$ and $V_2$ are two closed discrete subsets of the plane $\C$ such that each $V_i$ contains at least two points. If the boundaries of the hyperbolic convex hulls $\partial C(V_1\cup\{\infty\})$ 
    and $\partial C(V_2\cup\{\infty\})$ in the upper-half-space model of the hyperbolic 3-space are isometric, then $ C(V_1\cup\{\infty\})$ is isometric to $ C(V_2\cup\{\infty\})$.
\end{theorem}

The uniqueness result for polyhedral surfaces with hyperbolic background metrics, called the discrete Schwarz lemma, is a consequence of recent independent work in \cite{Zhao2026rigidity} and \cite{luo-luo}.  

\begin{theorem}[Discrete Schwarz Lemma \cite{Zhao2026rigidity,luo-luo}] \label{llz}
    If $V_1$ and $V_2$ are two closed discrete subsets of the open unit disk $\D$ with $\overline{V_i}-V_i =\partial \D$ such that the boundaries of the hyperbolic convex hulls $\partial C(V_1  \cup \D^c)$ and $\partial C(V_2 \cup \D^c)$ in the upper-half-space model of the hyperbolic 3-space are isometric, then   $ C(V_1 \cup \D^c)$ is isometric to $ C(V_2 \cup \D^c)$.
\end{theorem}
Figure \ref{fig:standardmodel} illustrates Theorems \ref{llr} and \ref{llz}. The formulation of Theorems \ref{llr} and \ref{llz} is similar to Cauchy's rigidity theorem for polyhedra in the Euclidean space $\E^3$. It states\footnote{Cauchy's original theorem assumes that the isometry between the boundaries $\partial C_E(V_1)$ and $\partial C_E(V_2)$ preserves the cellular structure. This constraint was removed by A. D. Alexandrov.} that if $V_1$ and $V_2$ are two finite sets in the Euclidean space $\E^3$ such that their Euclidean convex hulls $C_E(V_1)$ and $C_E(V_2)$ are of dimension at least two and $\partial C_E(V_1)$ is isometric to $\partial C_E(V_2)$, then  $C_E(V_1)$ is isometric to $C_E(V_2)$.

\begin{figure}[htbp]
  \centering
    \hfill
  \begin{minipage}[t]{0.45\textwidth}
    \centering
    \includegraphics[width=\textwidth]{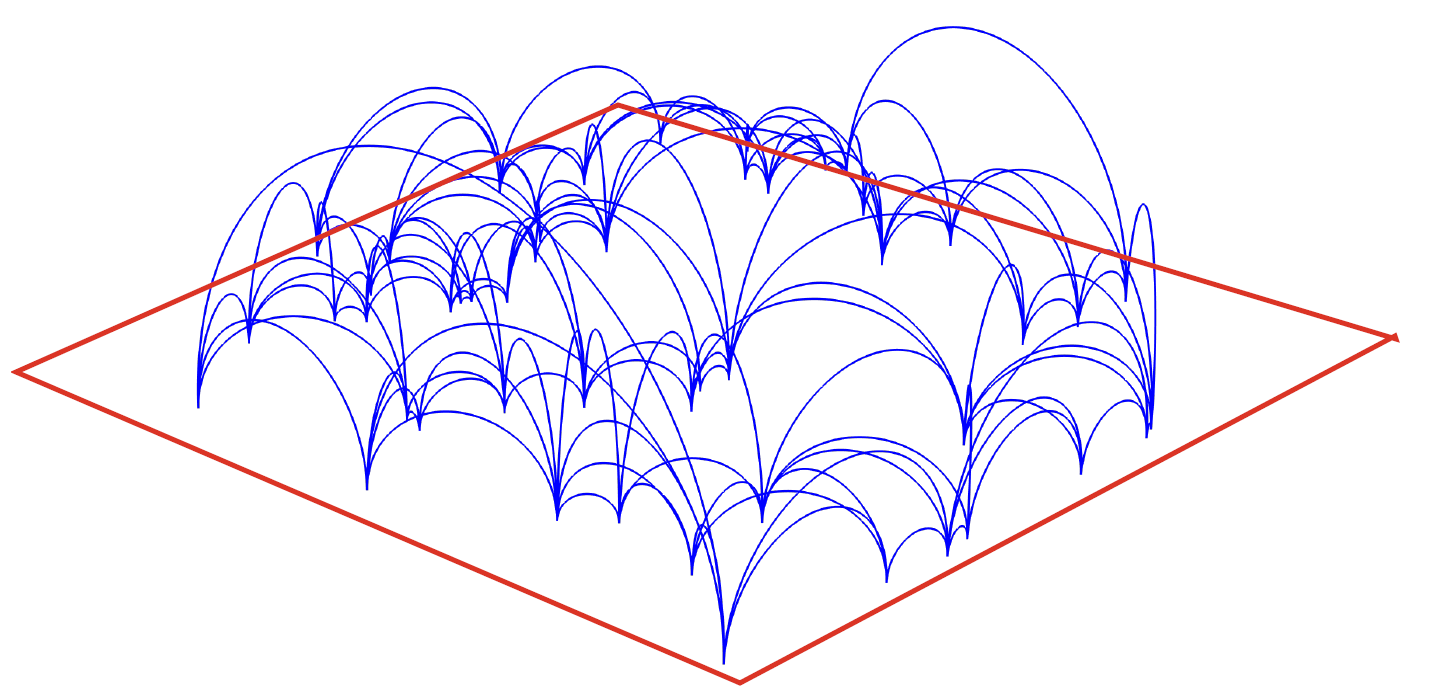}
  \end{minipage}
  \hfill
  \begin{minipage}[t]{0.4\textwidth}
    \centering
    \includegraphics[width=\textwidth]{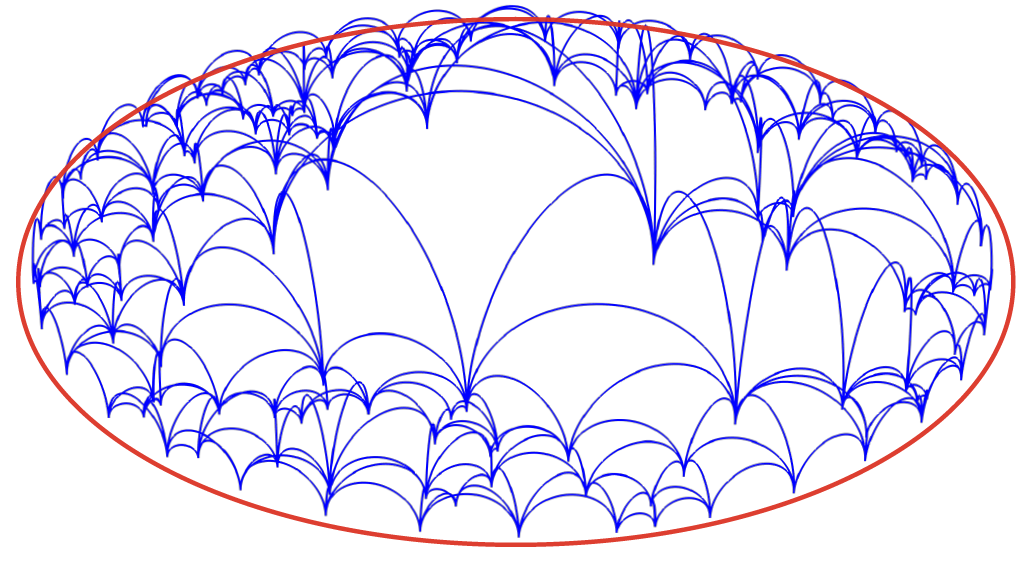}
  \end{minipage}
   \hfill
  \caption{The convex hulls $C(V\cup\{\infty\})$ of $V \subset \C$ and  $C(V \cup \D^c)$ of $V \subset \D$.}
  \label{fig:standardmodel}
\end{figure}



We remark that stronger results than Theorems \ref{llr} and \ref{llz} are proved in \cite{luo-luo-rao} and \cite{Zhao2026rigidity}, in which the requirement that the sets be closed and discrete is replaced by the assumption that they have vanishing one-dimensional Hausdorff measure.
These will be recalled in Theorem \ref{5.2}. These results imply a stronger rigidity theorem for complete constant-curvature Riemannian surfaces equipped with closed subsets of vanishing $1$-dimensional Hausdorff measure. See Theorem \ref{thm:new5.2} for details.





A consequence of Theorems \ref{lw} and \ref{llz} is the following discrete version of the Riemann mapping theorem. Recall that a 
\underline{crosscut} of a simply connected domain $\Omega$ in the complex plane $\C$ is a simple arc joining two distinct boundary points of $\Omega$ such that the arc, with its endpoints removed, is contained entirely in $\Omega$. By the Jordan curve theorem, $\Omega -\alpha$ has two connected components. 
See, for instance, \cite{milnor} for details. A closed discrete subset $V$ of $\Omega$ is called a \underline{discrete approximation} (or simply an \underline{approximation}) of $\Omega$ if for any crosscut $\alpha$ of $\Omega$, $V$ intersects both connected components of $\Omega -\alpha.$  In the case where $\Omega$ is a Jordan domain, this is equivalent to the condition that every point on the Jordan curve $\partial \Omega$ is a limit point of $V$ in $\mathbb{C}$, i.e., $\overline{V}-V=\partial \Omega$.

\begin{theorem}\label{disriemap} Suppose $\Omega$ is a simply connected domain in the Riemann sphere $\hat \C$ such that its boundary $\partial \Omega$ contains at least three points and $V \subset \Omega$ is a discrete approximation of $\Omega$. Then there exists a discrete approximation $W$ of $\D$  such that the Thurston dome
$\partial C(V \cup \Omega^c)$ is isometric to $\partial C(W \cup \D^c)$. Furthermore, the discrete set $W$ is unique up to Möbius transformations of $\D$.
\end{theorem}

\begin{figure}[htbp]
  \centering
  \begin{minipage}[t]{0.45\textwidth}
    \centering
\includegraphics[width=\textwidth]{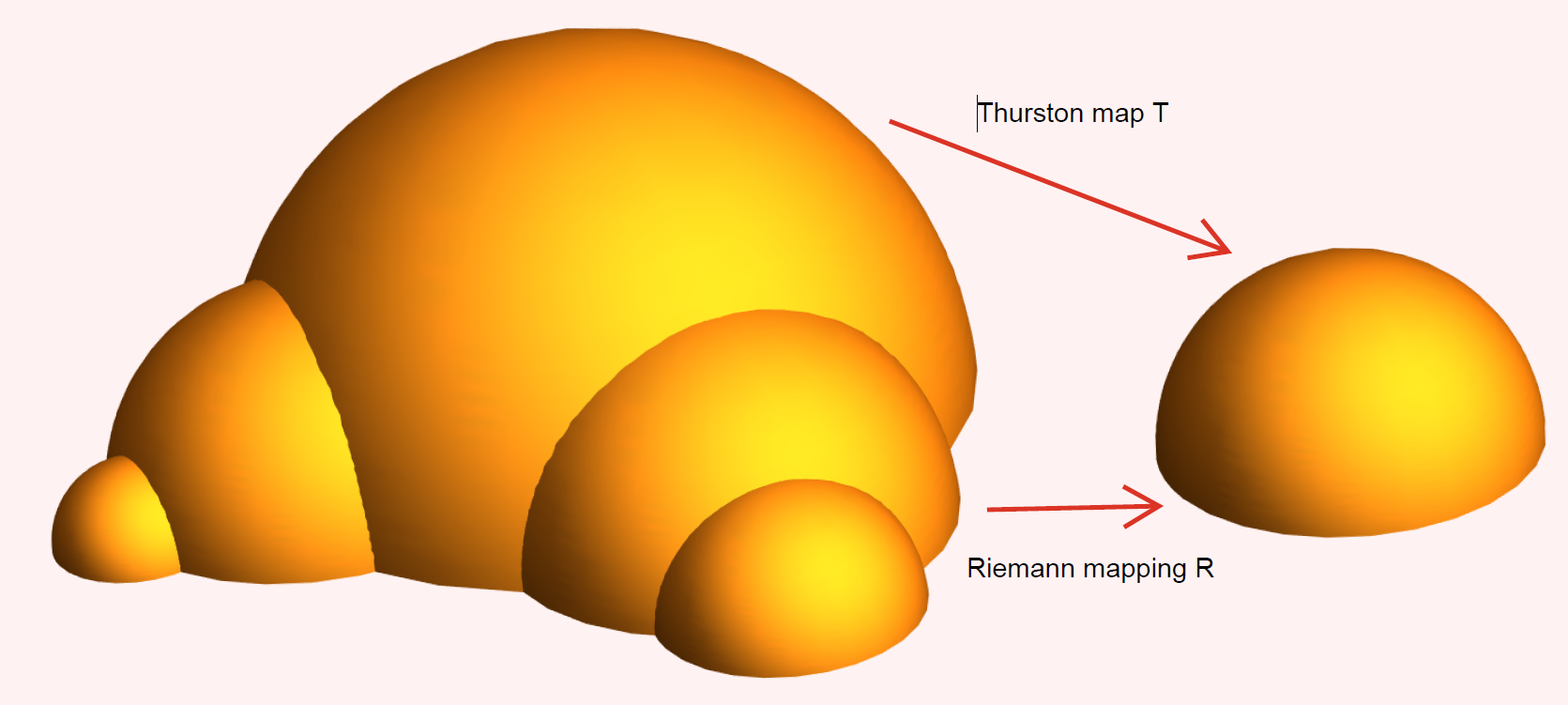}
    \caption{Thurston's isometry $T$ between the domes and the Riemann mapping $R$ between the domains.} 
     \label{Thudome}
  \end{minipage} 
  \hspace{0.5cm}
  \begin{minipage}[t]{0.5\textwidth}
    \centering
\includegraphics[width=\textwidth]{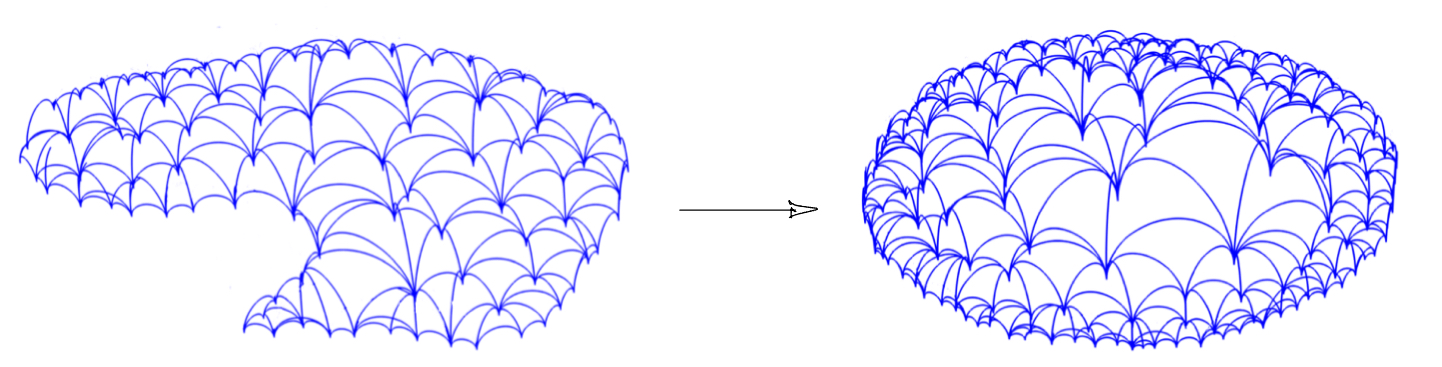}
    \caption{The discrete Riemann mapping $f:V \to W$ is induced by the isometric map between the domes over $\Omega-V$ and $\D-W$.}
    \label{drm}
  \end{minipage} 
 \end{figure}
We call the map from $V$ to $W$ induced by an orientation-preserving isometry from $\partial C(V\cup \Omega^c)$ to $\partial C(W\cup \D^c)$ the \underline{discrete Riemann mapping}. See Figure \ref{drm}. 
Note that Theorem \ref{disriemap} does not follow from Theorem \ref{dut} since $(\Omega, V, d_{\E})$ may not be polyhedral.  

We conjecture that the discrete Riemann mapping converges to the Riemann mapping. 
Recall that an $\epsilon$-net in a metric space $(X, d)$ is a subset $V$ such that each point in $X$ is contained in an $\epsilon$-ball centered at a point in $V$. 
\begin{conjecture}[Convergence of Discrete Riemann Mappings]\label{cov-drm}
    Suppose $\Omega$ is a simply connected proper domain in $\C$ and $f: \Omega \to \D$ is the normalized Riemann mapping sending a given point $p \in \Omega$ to $0$ with $f'(p)>0$. Let $V_{n}$ be a $\frac{1}{n}$-net in $\Omega$ such that $p \in V_n$ and $V_n$ is a discrete approximation of $\Omega$. 
    Suppose $f_n: V_n \to W_n$ is the discrete Riemann mapping with $f_n(p)=0$.
    Then there exists a sequence of unit complex numbers $e^{i \theta_n }$ such that for any compact subset $K \subset \Omega$, $e^{ i \theta_n } f_n |_{ V_n \cap K}$ converges uniformly\footnote{This means that for any $\epsilon >0$, there exists $N$ such that if $n>N$, $|e^{ i \theta_n }f_n(x) -f(x)| \leq \epsilon$, for all $x \in V_n \cap K$.} to $f|_K$.
\end{conjecture}
Preliminary results on the convergence of discrete conformal maps to their smooth counterparts have been established for closed surfaces \cite{guluowu, wuz2024, luowuzhu} and bounded Jordan domains approximated by standard regular hexagonal triangulations \cite{luosunwu}.

\medskip
\noindent
{\bf Two discretization schemes of Riemann surfaces: circle packing and vertex scaling.} 
Since the introduction of circle packing theory by Thurston as a discrete approach to conformal mapping \cite{thurston}, discrete conformal geometry has grown into a very active research area that provides a discrete counterpart to the classical theory of Riemann surfaces. The subject has deep connections with geometry, topology, analysis, computer graphics, and related fields. Among the various discretization schemes that have been developed \cite{glickenstein2017duality, gu-luo-yau, bowers2004uniformizing}, two have emerged as particularly influential: tangential circle packings and vertex scaling. Both arise from fundamental properties of conformal maps. On the one hand, conformal maps preserve circles in tangent spaces, leading to circle packing metrics and the associated circle packing theory. On the other hand, they preserve length cross-ratios in tangent spaces. Discretizing this property gives the notion of vertex scaling and the resulting discrete conformal equivalence for non-compact surfaces investigated in this paper. See Definition \ref{dcdef} and Appendix \ref{appen:vs} for details.

The discrete uniformization theory for tangential circle packings was established through the pioneering work of Koebe \cite{koebe1936kontakt}, Andreev \cite{andreev1970convex}, Thurston \cite{thurston}, Beardon and Stephenson \cite{beardon1990uniformization, beardons1991}, and He and Schramm \cite{he1999,hes1993,heschramm1995hyperbolic,schramm1991}. The classical Koebe--Andreev--Thurston theorem states that every simplicial triangulation of the 2-sphere is the nerve of a circle packing on the 2-sphere, unique up to Möbius transformations. Here, the nerve of a circle packing is the triangulation whose vertices correspond to the circles and whose edges record tangencies between pairs of circles. 

Beardon and Stephenson \cite{beardon1990uniformization} and He and Schramm \cite{hes1993,he1999,schramm1991} extended this theory to noncompact simply connected triangulated surfaces. They proved that every locally finite, infinite simplicial triangulation $\mathcal{T}$ of the plane admits a circle packing whose nerve is isomorphic to the one-skeleton of $\mathcal{T}$, realized either in the unit disk or in the Euclidean plane. The realization is unique up to Möbius transformations in the disk case and up to similarities in the Euclidean case. Moreover, these two possibilities are mutually exclusive, giving rise to the classification of circle packings into hyperbolic and parabolic types. They also established a strengthened discrete Riemann mapping theorem for circle packings: if $(\Omega,\mathcal{P})$ is a bounded simply connected domain equipped with a locally finite infinite circle packing $\mathcal{P}$, then there exists a circle packing on the unit disk, unique up to Möbius transformations, whose nerve is isomorphic to that of $\mathcal{P}$. Comprehensive accounts of circle packing theory can be found in \cite{bowers2020combinatorics, stephenson2005}.

The results of this paper and the work in \cite{rivin1994, filla2008, glsw, gglsw} provide a complete analog of this theory in the setting of vertex scaling. Rivin's
Theorem~\ref{rivin} 
plays the role of the Koebe--Andreev--Thurston theorem for the $2$-sphere. The existence and uniqueness of discrete uniformization via vertex scaling established in Theorems~\ref{lw}, \ref{llr}, and \ref{llz} are the counterparts of the work of Beardon and Stephenson, Schramm, and He. These results further yield a strengthened discrete Riemann mapping theorem stated as Theorem~\ref{disriemap} 
precisely paralleling the corresponding circle packing theorem.

There is a fundamental difference between these two discretization schemes. The circle packing scheme depends on the choice of triangulation, whereas the vertex scaling scheme is intrinsic and independent of any particular triangulation. Moreover, (tangential) circle packing theory applies only to the special class of cone metrics arising from circle packing metrics. In contrast, vertex scaling is defined for arbitrary cone metrics on closed surfaces, making the theory applicable to all polyhedral surfaces rather than only those that arise from circle packings. From this perspective, vertex scaling provides a more general and flexible notion of discrete conformality while preserving many of the geometric and variational features that originally motivated circle packing theory.

\medskip
\noindent
{\bf Organization of the paper.}
The paper is organized as follows. In \S2, we formulate the definitions of polyhedral surfaces and standard models, as well as discrete conformal
equivalence between such surfaces. Some basic properties of Delaunay triangulations and tessellations of polyhedral surfaces are proved in \S3.  In \S4, we prove that the associated hyperbolic metrics (i.e., the discrete conformal class) for polyhedral surfaces satisfying the assumptions in Theorem \ref{dut} are complete.  The main theorem, Theorem \ref{dut}, is proved in \S5. In \S6, we prove  Theorem \ref{disriemap}.   In the appendices, we first provide the detailed proof of Lemma \ref{lem:Q-S-X-d}, recall the equivalent definition of discrete conformality using vertex scaling, construct Tianqi Wu's polyhedral surface, and finally give two examples of polyhedral surfaces with Euclidean and spherical background metrics whose associated hyperbolic metrics are not complete. 

\medskip
\noindent{\bf List of notations} 
\begin{enumerate}
    \item $\R^n$ is considered as a vector space. 
    \item The three $n$-dimensional space forms are $\mathbb{S}^n$, $\E^n$, and $\H^n$ with metrics $d_\SS$, $d_\E$, and $d_\H$.
    \item The models used for $\H^n$ are the Poincar\'e ball model $\H^n_P$, the Klein ball model $\H^n_K$, and the upper half-space model $\H^n_U$.

    \item The convex hull of a set $X$ in $\mathbb{S}^2$, $\E^2$, or $\H^2$ is denoted by $C_{\mathbb S}(X)$, $C_{\E}(X)$, or $C_{\H}(X)$ respectively.  For a set $X \subset \partial \H^3$, its hyperbolic convex hull is denoted by $C(X)$, whose boundary is $\partial C(X)$.
    \item $\D$ is the unit disk in $\C$.
    \item The common identifications: $\mathbb{S}^2=\partial\mathbb{H}^3_P,
\E^2=\partial\mathbb{H}^3_U-\{\infty\}, 
\mathbb{H}^2_P=\D\subset\mathbb{C}\cup\{\infty\}=\partial\mathbb{H}^3_U.$

    \item $M^\circ$ denotes the interior of a manifold $M$ with boundary, or the interior of a set $M$ in $\mathbb{R}^n$.
    \item $\overline{X}$ is the closure of a subset $X$ of $\R^3$.

\item If $K$ is a proper 3-dimensional convex set in $\H^3$, we use $\partial K$ to denote the boundary of $K$ in $\H^3$.

\end{enumerate}

\medskip
\noindent
 {\bf Acknowledgments.}  
 We thank Tianqi Wu for discussions and for allowing us to use his examples in this paper. F. L. would like to thank Dennis Sullivan for the insightful discussions and encouragement that inspired and sustained the long-term research behind this project.  The work of F. L. is supported in part by NSF grants DMS 2220271 and DMS 2501286, and by Simons Foundation grant SFI-MPS-SFM-00011051. The work of Z. R. is supported in part by NSF grant DMS-2603932.

\medskip
\noindent
 {\bf Tool and computational resource disclosure.}  
 The two examples in Appendix~\ref{appen:incomplete}, which show that the conditions in Theorem \ref{dut} are optimal, are produced with the assistance of ChatGPT. 
All other mathematical results, proofs, and conclusions in the manuscript were independently developed by the authors without the assistance of AI. The authors take full responsibility for the content of the article.

\medskip

\section{Polyhedral surfaces, standard models, and discrete conformal equivalence}\label{sec:Polyhedral-standard-models}

 Polyhedral surfaces are a standard discretization of smooth Riemannian surfaces.  While the concept of a closed polyhedral surface is classical, its non-compact counterpart demands a more careful treatment. We therefore begin with marked surfaces carrying cone metrics and subsequently define polyhedral surfaces using Delaunay triangulations. The Delaunay condition is crucial for the definition of discrete conformal equivalence. 

 All surfaces in the rest of this paper are assumed to be connected and without boundary unless otherwise stated.
 
\noindent{\bf Marked surfaces with cone metrics.} 
A marked surface is a pair $(S, V)$ where $S$ is  connected and possibly non-compact, and $V$ is a nonempty closed discrete subset of $S$ called the set of vertices. 
A \underline{Euclidean (or hyperbolic, or spherical) cone metric}  $d$ on $(S, V)$ is a metric on $S$ such that it is a Euclidean (or hyperbolic, or spherical) metric away from $V$ and has a cone-type singularity, whose neighborhood is isometric to an open set in a cone. 
A standard example of a Euclidean cone metric is the boundary of the convex hull of a finite set in the Euclidean 3-space.
We call $(S, V, d)$ a marked surface with a cone metric,  and call $d$ the background metric. 

We usually rely on a triangulation $\T$ to describe a marked surface with a cone metric. A \underline{geometric triangulation}  $\T$ of a marked surface with a cone metric $(S, V, d)$ is a locally finite with identified vertices and multiple edges triangulation of $S$ such that $V$ is the set of vertices $V(\T)$ of $\T$, and each open triangle in $\T$ is isometric to an open triangle in  $\mathbb E^2$,  $\H^2$, or $\mathbb S^2$.  Recall that a \underline{spherical triangle} is the convex hull of three distinct points contained in an open hemisphere in $\mathbb {S} ^2$ such that the convex hull is 2-dimensional.
With the triangulation $\mathcal{T}$, the cone metric $d$ is reconstructed by isometrically gluing Euclidean (or hyperbolic, or spherical) triangles along pairs of their edges. The metric $d$ can be described by the \underline{edge length function} $l: E(\T)\to \mathbb{R}_{>0}$, which assigns each edge its length. 
Here, $E(\mathcal{T})$ is the set of edges of $\mathcal{T}$.


\medskip

\noindent{\bf Polyhedral surfaces and their associated hyperbolic metrics.}
We introduce a special class of marked surfaces with cone metrics, called \underline{polyhedral surfaces}. It is related to a well-known class of triangulations in computational geometry called  \underline{Delaunay triangulations}.

Let $(S, V, d)$ be a marked surface with a cone metric $d$ and $B$ be a round open ball in $\mathbb{K}^2$, where $\mathbb{K}$ is $\E,\H$ or $\mathbb S$. We assume that the radius of $B$ is strictly less than $\pi/2$ if $\K^2=\mathbb S^2$.  The closure of $B$  in $\mathbb{K}^2$ is denoted by $\overline{B}$. An \underline{empty-disk} in $(S,V, d)$ introduced in \cite{bobenkos2007} is an isometric immersion $\varphi: B\to (S-V, d|_{S-V})$. We call an  empty-disk $(B, \varphi)$ a \underline{circumdisk} for $(S, V, d)$ if $\varphi$ extends continuously to $\Phi:\overline{B} \to S$ such that $\Phi^{-1}(V)$ contains at least three points. From this definition, the image of each circumdisk is contained in the compact subset $\Phi(\overline{B})$ of the surface $S$. For simplicity, we also call $(\overline{B}, \Phi)$ or simply $\Phi(\overline{B})$ the circumdisk.  We say a $(B, \phi)$ is the circumdisk of a triangle $\tau$ if the interior $\tau^{\circ}$ of $\tau$ is the $\phi$ image of an open triangle whose vertices are in $\Phi^{-1}(V)$.


\begin{definition}[Polyhedral surfaces and Delaunay condition]\label{poly}
    A marked surface with a cone metric $(S, V, d)$ is a polyhedral surface if there exists a geometric triangulation $\mathcal{T}$ of $(S, V, d)$ satisfying the following \underline{Delaunay condition}: each triangle $\tau$ in $\mathcal{T}$ has  a \textbf{compact} circumdisk in $S$.
    The triangulation $\mathcal{T}$ is called a Delaunay triangulation, and the metric $d$ is called a polyhedral metric on $(S, V)$.
\end{definition}

Delaunay triangulation is a fundamental and widely used tool in computational geometry. It is dual to the Voronoi diagram associated with $(S, V, d)$. In general, the Delaunay triangulation $\T$ for $(S, V, d)$ in the above definition is not unique. 
A basic theorem in computational geometry states that every closed marked surface with a  Euclidean or hyperbolic cone metric is polyhedral. However, this statement is false for spherical cone metrics on closed surfaces. For instance, there are many spherical cone metrics $d$ on a closed marked surface $(S, V)$ such that $(S-V, d|_{S-V})$ contains an embedded open hemisphere.
These cone metrics can never be triangulated, and hence they are not polyhedral.  It is proved in \cite{Izmestiev-Prosanov-Wu} that if a closed surface $(S, V, d)$ with a spherical background metric admits a geometric triangulation, then it is polyhedral. 

 We remark that there exists a geometric triangulation of the hyperbolic plane $\H^2$ such that $(\mathbb{H}^2, V(\mathcal{T}), d_{\H})$  admits no Delaunay triangulation; in this example, some circumdisks are horoballs and hence are noncompact. 
On the other hand, it is not hard to prove that if $\mathcal{T}$ is a geometric triangulation of the Euclidean plane $\mathbb{E}^2$, there exists a Delaunay triangulation $\mathcal{T}'$ of $\mathbb{E}^2$ such that $V(\mathcal{T}) = V(\mathcal{T}')$. Therefore, $(\mathbb{E}^2, V(\mathcal{T}), d_{\E})$ is a polyhedral surface.   

Here is an example of a non-polyhedral marked surface with a Euclidean cone metric. See Figure \ref{non-poly}. The underlying surface is an open triangulated triangle in $\E^2$, in which the
triangles of the triangulation are obtained by adding one diagonal to each trapezoid. The circumdisk of each triangle contains no vertices in its interior. However, many circumdisks do not lie within the open triangle. 
\begin{figure}[h!]
    \centering
    \includegraphics[width=0.25\textwidth]{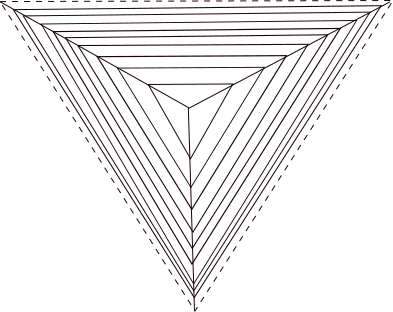}
    \caption{A non-polyhedral open triangulated surface.}
    \label{non-poly}
\end{figure}

Given a polyhedral surface $(S, V, d)$ with a Euclidean, hyperbolic, or spherical background metric, we will construct its \underline{associated hyperbolic metric} $d^*$ on the punctured surface $S-V$. This metric $d^*$ is used to define the 
discrete conformal class of $(S, V, d)$.  There are several other equivalent definitions of the discrete conformal class. See Appendix~\ref{appen:vs} for more information.

The basic construction of $d^*$ is based on a natural correspondence that associates to each geometric triangle $\tau$ in $\mathbb K^2 =\E^2$, or $\mathbb S^2$,  or $\H^2$ an ideal hyperbolic triangle $\tau^*$ with the same set of vertices, together with a naturally defined homeomorphism from $\tau^*$ to $\tau$ with vertices removed. See Figure \ref{geo}. We identify $\E^2$, $\mathbb S^2$, and $\H^2$ with the complex plane in $\partial \H^3_U$, the conformal boundary $\partial \H^3_P$, and the open unit disk $\D$ in $\partial \H^3_U$ respectively. Under these identifications, each isometry $f$ of $\mathbb K^2$ extends uniquely to an isometry $f^*$ of $\H^3$.  Furthermore, there exist canonical projections $\phi_{\mathbb K}$ from $\H^3_U$ to $\mathbb K^2$ for $\mathbb K=\E, \H$, as well as $\phi_{\mathbb S}: \H^3_P-\{0\} \to \mathbb S^2$. These maps are characterized by the property that the geodesic through $x$ and $\phi_{\mathbb K}(x)$  contains $\infty$ if $\mathbb K^2=\E^2$, contains the origin $(0,0,0)$ if $\mathbb K^2 =\mathbb S^2$, and is perpendicular to the hemisphere $C(\partial \D) =\mathbb S^2_+$ if $\mathbb K^2=\H^2$. In particular, $\phi_{\E}$ and $\phi_{\mathbb S}$ are vertical and radial projections, respectively. 
More precisely, for $(u, t) \in \R^2 \times \R$, one easily checks that $\phi_{\mathbb K}$ is given by
\begin{equation}\label{eq:projection}
\phi_{\mathbb K}(u,t)=
\begin{cases}
(u,0), & \text{if } \mathbb K^2=\mathbb{E}^2, t>0\\
 \frac{2u}
{1+|u|^2+t^2+
\sqrt{\left(1+|u|^2+t^2\right)^2-4|u|^2}},
          & \text{if } \mathbb K^2=\mathbb{H}^2, t>0 \\
\frac{(u,t)}{\sqrt{ ||u||^2+t^2}}, & \text{if } \mathbb K^2=\mathbb S^2,  (u,t) \neq (0,0).
\end{cases}
\end{equation}

If $\tau$ is a geometric triangle with vertices $v_1, v_2, v_3$ in $\mathbb K^2$, i.e., $\tau=C_{\mathbb K}(\{v_1, v_2, v_3\})$, then its associated ideal hyperbolic triangle is $\tau^* =C(\{v_1, v_2, v_3\})$, the hyperbolic convex hull of its vertices.
Similarly, if $e=C_{\mathbb K}(\{v_i, v_j\})$ is an edge of $\tau$ in $\mathbb K^2$, we let $e^*=C(\{v_i, v_j\})$ be the geodesic in $\H^3$ having the same endpoints as $e$. 
The natural homeomorphism 
$\phi_{\tau}: \tau^* \to \tau -\{v_1, v_2, v_3\}$ is defined to be $$\phi_{\tau}=\phi_{\mathbb K}|_{\tau^*},$$ where $\mathbb K$ depends on the constant-curvature geometry in which $\tau$ is embedded.
By the construction, $\phi_{\tau}$ sends the edge $e^*$ of $\tau^*$ to the corresponding edge $e$ of $\tau$.

\begin{figure}[htbp]  
\begin{center}
\begin{overpic}[width=\linewidth]{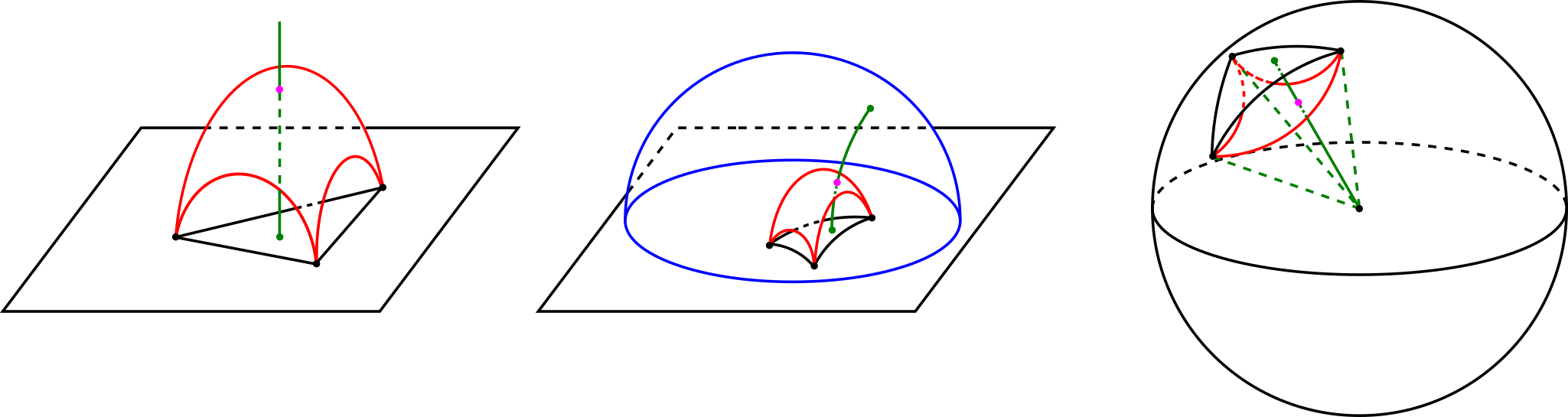}
    \put(21,22){\textcolor{red}{$\tau^*$}}
    \put(3,8){$\E^2$}
    \put(16,20){\textcolor[HTML]{FF00FF}{$x$}}
    \put(22.5,10.5){$\tau$}
    \put(14,8){\textcolor[HTML]{008000}{$\phi_\tau(x)$}}

    \put(47.5,14){\textcolor{red}{$\tau^*$}}
    \put(47.5,9){$\tau$}
    \put(41,12){$\D$} 
    \put(55,14.4){\textcolor[HTML]{FF00FF}{$x$}}
    \put(54,10.4){\textcolor[HTML]{008000}{$\phi_\tau(x)$}}

    \put(77,14){\textcolor{red}{$\tau^*$}}
    \put(75,18){$\tau$}
    \put(93,20){$\H^3_P$} 
    \put(87,11){$O$} 
    \put(84.3,19){\textcolor[HTML]{FF00FF}{$x$}}
    \put(82,24.4){\textcolor[HTML]{008000}{$\phi_\tau(x)$}}
    \put(100,22){$\SS^2$}

\end{overpic}
\caption{Ideal triangles associated with geometric triangles and natural maps.}  
\label{geo}
\end{center}
\end{figure}
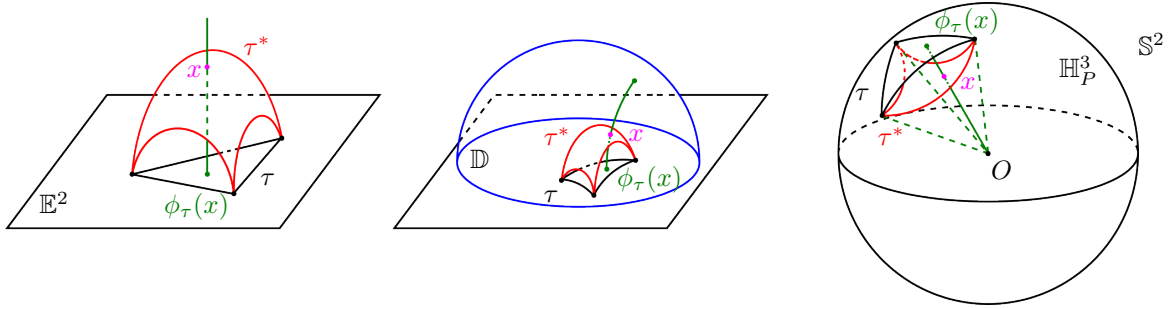

The following lemma follows from the definition. We omit the proof. 
\begin{lemma} \label{natural-h} 
Suppose $f$ is an isometry of $\mathbb K^2$ sending a triangle $\tau$ to another triangle $f(\tau)$ and $f^*$ is the isometry of $\H^3$ which extends $f$. Then,

(i) $\phi_{\mathbb K} \circ f^* = f \circ \phi_{\mathbb K},$

(ii) $f^*(\tau^*) = (f(\tau))^*,$ and

(iii) $ \phi_{f(\tau)} \circ f^* = f \circ \phi_{\tau}.$
\end{lemma}

Given an isometry $\psi$ from $\tau_1^*$ to $\tau_2^*$, the \underline{discrete conformal map} between two geometric triangles $\tau_1$ and $\tau_2$ is defined to be $\phi_{\tau_2} \circ \psi \circ \phi_{\tau_1}^{-1}$.  Note that $\tau_1$ and $\tau_2$ may come from different constant-curvature geometries. As first observed in \cite{bps}, discrete conformal maps from a Euclidean triangle $\tau_1$ to another Euclidean triangle $\tau_2$ are the same as real projective maps that send $\tau_1$ to $\tau_2$ and preserve the corresponding
circumdisks.

Now, we proceed to the construction of $d^*$, using geometric triangulations. Let $(S, V, d)$ be a polyhedral surface with a $\mathbb K^2$ background metric, where $\mathbb K^2=\E^2, \mathbb S^2$, or $\H^2$. Find a geometric triangulation $\T$ of $(S, V, d)$ and let $\Delta$ be the set of all triangles in $\T$. The triangulation $\T$ provides a way to reconstruct the metric surface $(S, V, d)$ by taking the disjoint union $\bigsqcup _{\tau \in \Delta} \tau \times \{\tau\}$ and then taking the quotient under the isometric identifications of pairs of edges, i.e.,
$(S, V, d)$ is isometric to $\bigsqcup _{\tau \in \Delta} \tau \times \{\tau\}/\sim$. Here we consider $\tau \times \{\tau\} \subset \mathbb K^2$ as an isometric copy of $\tau$. The isometric gluing is defined as follows.
If $\sigma, \tau \in \T$ are two geometric triangles sharing a common edge $e$, then there exists a unique isometry $f$ of $\mathbb K^2$ sending $e \times \{\tau\}$ to $e \times \{\sigma\}$ such that $f(\tau \times \{\tau\})$ and $\sigma \times \{\sigma\}$ are on different sides of $e \times \{\sigma\}$. Let $f^*$ be the isometric extension of $f$ to $\H^3$. We now glue  $(\tau \times \{\tau\})^*$ to  $(\sigma \times \{\sigma\})^*$ along their corresponding edges $(e \times \{\tau\})^*$ and $(e \times \{\sigma\})^*$ by $f^*$.  The quotient surface  $\hat{S}:=\bigsqcup _{\tau \in \Delta} (\tau
\times \{\tau\})^* /\sim$ has the same vertex set $V$. Due to the isometric gluing of ideal hyperbolic triangles along edges, the surface $\hat S -V$ carries a (possibly incomplete) hyperbolic metric $\hat d_{\T}$. By Lemma \ref{natural-h}, the natural homeomorphisms $\phi_{\tau}$ induce a homeomorphism $\phi_{S,V}: (\hat S, V) \to (S, V)$.  Define the hyperbolic metric $d^*_{\T}$ on $S-V$ to be the pushforward metric $(\phi_{S,V})_*(\hat d_{\T})$.
It is known that $d^*_{\T}$ has a cusp end at each vertex $v \in V$ (see for instance \cite{glsw}). The hyperbolic metric $d^*_{\T}$ associated with $d$ was first introduced in the important work~\cite{bps}. 
In general, the metric $d^*_{\T}$ depends on the choice of $\T$ and cannot be
considered as a metric associated with $(S, V, d)$. It was observed in \cite{glsw} and will be proved in \S3 that if $\T$ and $\T'$ are two Delaunay triangulations of $(S, V, d)$, then $d_{\T}^*=d_{\T'}^*$.  Due to this, we will write $d^*$ for $d^*_{\T}$ and call it the \underline{associated hyperbolic metric} for the polyhedral surface $(S, V, d)$.
See Figure \ref{777}. 

It is well known that Delaunay triangulations are related to convexity and convex hulls in hyperbolic 3-space $\H^3$. Given a closed subset $V\subset\K^2$, we define a closed subset $Y\subset
\partial\mathbb{H}^3$ by
\begin{equation}\label{eq:def-Y}
Y=
\begin{cases}
V, & \K^2=\mathbb{S}^2,\\
V\cup\{\infty\}, & \K^2=\E^2,\\
V\cup \D^c, & \K^2=\mathbb{H}^2.
\end{cases}
\end{equation}
Let $C(Y)\subset \mathbb{H}^3$ be the hyperbolic convex hull of $Y$. Note that if $C(Y)$ is 2-dimensional, the surface $\partial C(Y)$ is understood to be the metric double of $C(Y)$ in the sense of Alexandrov, i.e., $\partial C(Y)$ is obtained by gluing two copies $C(Y)_+$ and $C(Y)_{-}$ of $C(Y)$ via the identity map along their $1$-dimensional boundaries. By Thurston's theorem, the induced path metric $d_{\partial C(Y)}$ on $\partial C(Y)$ is a complete hyperbolic metric.

\begin{lemma} \label{convexhulld*} Let $\T$ be a Delaunay triangulation of $\mathbb K^2$ with vertex set $V$, where $\mathbb K^2 =\SS^2, \E^2$ or $\H^2$. Then there exists a natural homeomorphism $h$ from $\mathbb {K}^2-V$ to $\partial C(Y)$ such that $\phi_{\tau} \circ h|_{\tau} =id$ for all triangles $\tau \in \T$.
\end{lemma}

 \begin{proof}      We begin with the observation that circumdisks in $\K^2-V(\T)$ are closely related to the 2-dimensional faces of $C(V(\T))$. 
 Take the circumdisk $D$ for a triangle $\tau \in \T$. By definition, $\tau^* \subset \partial C( D)$. By the Delaunay condition on $\T$, $D$ contains no vertices of $V$ in its interior. Therefore, the interior of the hyperbolic half-space $C(D)$ in $\H^3$ is disjoint from $C(Y)$ and hence $\partial C(D)$ is a supporting plane for $C(Y)$. It follows that $\tau^* \subset \partial C(Y)$. Conversely, take a point $p \in \partial C(Y)$ and a supporting plane $\partial  C(D')$ of $C(Y)$ such that $D' \subset \partial \H^3$ is a round disk whose interior is disjoint from $V$ and  $p \in C(\partial D')$. By construction, $\partial D'$ contains at least two points of $V$.  If  $\partial D' \cap V$ contains at least three points, then $D'$ is the circumdisk of some triangle in $\T$ since $\T$ is Delaunay. Therefore $p \in \tau^*$ for some triangle $\tau \in \T$.  If $\partial D' \cap V$ contains only two points $u, v$, then $p \in C(\{u,v\})$. In this case, we can deform $D'$ to a new round disk $D$, keeping $\{u,v\} \subset \partial D \cap V$ and ensuring that $\partial D \cap V$ contains at least three points. This shows $D$ is a circumdisk, and by the case just proved, we see that $p \in \tau^*$ for some triangle $\tau \in \T$.  It follows that $(\K^2- V, \hat d_{\T}) =\partial C(Y)$. The restriction $\phi_{\mathbb K}|_{\partial C(Y)}: \partial C(Y) \to \mathbb K^2-V$ produces the natural homeomorphism $h^{-1}$. Therefore, the lemma follows. 
\end{proof}

\begin{figure}[ht!]
\centering
\begin{overpic}[scale=0.8]{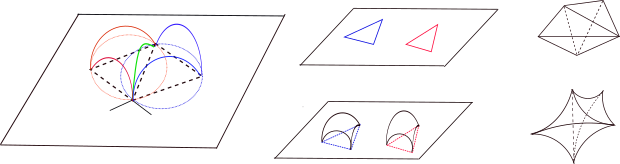}
    \put(60,23.5){$\tau$}
    \put(63.5,21){$f$}
    \put(63,23){$\longrightarrow$}
    \put(68,23.5){$\sigma$}
    \put(60,6){$f^*$}
    \put(59.5,8){$\longrightarrow$}
    \put(66.5,8){$\sigma^*$}
    \put(56.5,8){$\tau^*$}
   \put(15,21){$\sigma^*$}
   \put(32,21){$\tau^*$}
     \put(86,16){$d$}
     \put(86,1.5){$d^*$}
        \put(92,14){$\downarrow$}
\end{overpic}

\caption{The construction of $d^*$ associated with $d$ and the convexity arising from the Delaunay condition.}
\label{777}
\end{figure}
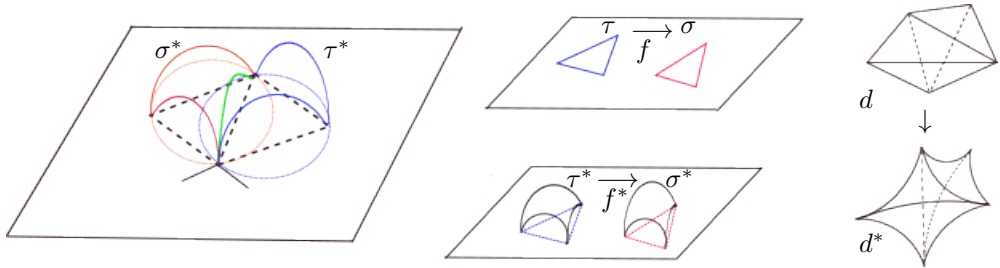

\medskip
\noindent{\bf Standard model surfaces and convex-hull hyperbolic metrics.} Standard model surfaces are discrete counterparts of complete Riemannian surfaces of constant curvature appearing in the classical uniformization theorem.  If $X$ is a closed subset of a Riemannian surface $(S, g)$, we call the triple $(S, X, g)$ \underline{exceptional} if $(S, X, g)$ is conformal to one of  $(\mathbb S^2, \{p\}), (\mathbb S^2, \{p_1, p_2\})$, $(\R P^2, \{p\})$ and $(\C, \{p\})$.

\begin{definition} A \underline{standard model surface} is a triple $(S,V,d)$, where $(S,d)$ is a complete Riemannian surface of constant curvature $-1, 0$ or $1$,  $V \subset S$ is a nonempty closed discrete subset, and $(S, V, d)$ is not exceptional.
    \label{standardmodel}
\end{definition}

Not all standard models are polyhedral. For example, a non-compact hyperbolic surface with a finite set of vertices is a standard model that is not polyhedral. 

For a standard model surface $(S, V, d)$, we associate with $(S,V,d)$ a complete hyperbolic metric $d^*_c$ on $S- V$, called a \underline{convex-hull hyperbolic metric}.  The construction is motivated by the use of hyperbolic convex hulls. More precisely, given a nonempty closed subset $X$ of the hyperbolic plane $\H^2_P =\D \subset \partial \H^3_U$, 
the canonical homeomorphism $Q:=(\phi_{\H}|_{\partial C(X \cup \D^c)})^{-1}$  satisfies 
$Q \circ \phi=\phi \circ Q$ for  any isometry $\phi$ of $\H^3_U$ with $\phi(\D)=\D$.
The pullback under $Q$ of the hyperbolic metric on $\partial C(X\cup \D^c)$ is the convex-hull hyperbolic metric $d_c^*$ on $\D-X$. The metric should be compared with the Poincar\'e metric, i.e., the complete conformal hyperbolic metric, on $\D-X$. The relationship has been studied by many authors.
The goal of this subsection is to show that convex-hull hyperbolic metrics exist for a large class of surfaces, including standard models. 

\begin{proposition}\label{chhm}
Suppose that $(S,V,d)$ is a complete Riemannian surface of constant curvature $0$, $-1$, or $1$, where $V\subset S$ is a nonempty closed subset. Assume that $(S,V)$ is non-exceptional and, in the case where the curvature of $d$ is $1$, that $V$ is finite.
Then there exists a complete hyperbolic metric $d^*_c$ on $S-V$ such that

(i) if $(S, d)$ is $\H^2_P$, $\mathbb S^2$, or $\E^2$, then $(S-V, d_c^*)$ is isometric to $\partial C(Y)$, with $Y$ defined by \eqref{eq:def-Y};

(ii) if $\phi: (S_1, V_1, d_1) \to (S_2, V_2, d_2)$ is a Riemannian covering map, then the pullback metric $\phi^* (d_2)^*_c$ is equal to $(d_1)^*_c$ on $S_1-V_1$, i.e.,
\begin{equation} \label{eq-chh}
\phi^* (d_2)^*_c =(d_1)^*_c.
\end{equation}
\end{proposition}
Note that if $\phi$ is a self-isometry of $(S_1, V_1, d_1)$, then \eqref{eq-chh} implies that $\phi|_{S_1-V_1}$ is a self-isometry of $(S_1-V_1, (d_1)^*_c)$.

In view of this proposition, we introduce the following definition.

\begin{definition}[convex-hull hyperbolic metric] For a surface $(S,V,d)$ satisfying the conditions in Proposition \ref{chhm}, the metric $d^*_c$ is called a \underline{convex-hull hyperbolic metric} on $S-V$ associated with $(S, V, d)$.   
    \end{definition}

\begin{proof}[Proof of Proposition \ref{chhm}]

The construction is first carried out for the simply connected model surfaces and then for general surfaces by passing to their universal covers.

Let  $S$ be simply connected. Depending on its curvature, we identify $(S,d)$ with one of the following standard models as before:
$$
\mathbb{S}^2=\partial\mathbb{H}^3_P,
\qquad
\E^2=\partial\mathbb{H}^3_U-\{\infty\},
\qquad
\mathbb{H}^2_P= \D \subset \partial \H^3_U.
$$
Under this identification, $A|_S$ is defined for any hyperbolic isometry $A\in\Isom(\H^3)$.
We continue to denote by $Y$ the subset of $\partial\mathbb{H}^3$ defined from $V$ by \eqref{eq:def-Y}.
Our objective is to construct a homeomorphism $Q_{S,V,d}:  S-V \to \partial C(Y)$.

\begin{lemma}\label{lem:Q-S-X-d}
    For each simply connected surface $(S, V, d)$ satisfying the same assumptions,  there exists a homeomorphism $Q_{S,V,d}: S - V \to \partial C(Y)$ such that if  $(S_1, V_1, d_1)$ and $(S_2, V_2, d_2)$ are simply connected and related by an isometry $A \in \operatorname{Isom}(\mathbb{H}^3)$ with $A(S_1) = S_2$ and $A(V_1) = V_2$,  then
    \begin{equation} \label{equivq}
        A \circ Q_{S_1, V_1, d_1}(\xi) = Q_{S_2, V_2, d_2}\circ A (\xi),\quad\forall \xi\in S_1 - V_1.
    \end{equation}
\end{lemma}

The proof of the above lemma for the spherical and hyperbolic background metrics is relatively simple. If $d$ is spherical, the map $Q_{S,V, d}^{-1}$ is the restriction to $\partial C(V)$ $ (= \partial C(Y)$)
of the hyperbolic radial projection from the conformal center of $V$. If $d$ is hyperbolic, or $d$ is flat such that $C_{\E}(V) =\E^2$, $Q_{S, V, d}^{-1}$ is the restriction $\phi_{\K}|_{\partial C(Y)}$ for $\K=\H$ or $\E$ respectively. However, for the remaining cases of the flat metric $d$, the proof is considerably longer, involves several cases, and is irrelevant for the rest of the paper. We therefore relegate it to Appendix~\ref{appen:lemma2.7} to keep the main argument focused and uninterrupted.

For each simply connected surface $(S, V, d)$, we define $d^*_c$ to be the pullback metric of $d_{\partial C(Y)}$ by the map $Q_{S, V, d}$ in Lemma \ref{lem:Q-S-X-d}. By the equivariance of $Q_{S, V, d}$, the metric $d^*_c$ is well-defined and has the following property. 

\begin{corollary} \label{2.8}If $\phi$ is a self-isometry of a simply connected surface $(S, V, d)$, then $\phi|_{S-V}$ is an isometry with respect to $d^*_c$.
 \end{corollary}

Now consider the case where $(S,V, d)$ is not simply connected.
Let $F:(\tilde{S}, \tilde d) \to (S, d)$ be the Riemannian universal cover of $S$, set $\tilde{V}=F^{-1}(V)$, and let $\Gamma:=\pi_1(S)$ act isometrically on $(\tilde S, \tilde V, \tilde d)$ as the deck transformation group. 
By Corollary \ref{2.8}, for each $\gamma \in \Gamma$, $\gamma|: (\tilde S-\tilde V,  (\tilde d)^*_c) \to (\tilde S-\tilde V,  (\tilde d)^*_c) $ is an isometry.  Since $F|: \tilde S-\tilde V \to S-V $ is a regular covering map with deck transformation group $\Gamma$, it follows that the pushforward metric $(F|)_*(\tilde d)^*_c$ is a complete hyperbolic metric on $S-V$, which is defined to be the associated convex hull metric $d^*_c$ on $S-V$. 
 The property \eqref{eq-chh} follows from the definition, \eqref{equivq}, and the following well-known fact about covering spaces: if $\alpha: X \to Z$ is a Riemannian covering from a connected Riemannian manifold $X$ and $\beta: \tilde X \to X$ is the universal Riemannian covering space of $X$, then $\alpha \circ \beta: \tilde X \to Z$ is the universal Riemannian covering space of $Z$.
 \end{proof}

If a standard model surface is also a polyhedral surface, we have the following.

\begin{corollary} \label{st-poly} Suppose $(S, V, d)$ is a standard model surface which is also a polyhedral surface. Then,

(i) if $d$ is hyperbolic or Euclidean, then the associated hyperbolic metric $d^*$ and the convex-hull hyperbolic metric $d^*_c$ are equal on $S-V$;

(ii) if $d$ is spherical, then $d^*$ is isometric to $d^*_c$ by an isometry isotopic to the identity.
\end{corollary}

\begin{proof}
    We first consider the case where $S$ is simply connected.
    If $d$ is hyperbolic, then the result follows from Lemma \ref{convexhulld*}. If $d$ is Euclidean, then, since $(S, V, d)$ is polyhedral, we know that $V$ is the set of vertices of a Delaunay triangulation of $\E^2$ and $C_{\E}(V) = \E^2$. Thus (i) also follows from Lemma \ref{convexhulld*}.
   If $d$ is spherical, the conformal barycenter $b(V)$ of the vertex set $V$ may be different from the origin $0$ in the Poincar\'e model. Therefore, the radial projections from $0$ and $b(V)$ may produce different metrics via the pushforward construction.  However, these two hyperbolic metrics are still isometric by an isometry isotopic to the identity. Indeed, moving $b(V)$ to $0$ along a hyperbolic geodesic segment and using the radial projections at points in the segment produces the required isotopy.

If $S$ is not simply connected, let $F: (\tilde S, \tilde V, \tilde d) \to (S, V, d)$ be the Riemannian universal cover of $(S, V, d)$. By definition, the pullback metrics $F^*(d^*)$ and $F^*(d^*_c)$ are the associated hyperbolic metric $(\tilde d)^*$ and the convex-hull hyperbolic metric $(\tilde d)^*_c$ on $\tilde S-\tilde V$, respectively.  By the simply connected case just proved, we see $(\tilde d)^* =(\tilde d)^*_c$, i.e., $F^*(d^*) =F^*(d^*_c)$ if $d$ is hyperbolic or Euclidean. It follows that $d^*=d^*_c$ in these two cases. If $d$ is spherical, the same argument shows that $d^*=d^*_c$. This is because the lift $\tilde V$ of $V$ on $\R P^2$ is symmetric about the origin in $\SS^2$ so that the conformal barycenter of $\tilde V$ is always the origin. Then $d^*=d^*_c$.
\end{proof}

For the rest of the paper, we will use these two metrics interchangeably if a standard model surface is polyhedral with a hyperbolic or Euclidean background metric.  

\bigskip
\noindent{\bf Discrete conformal equivalence of polyhedral surfaces.} 

\begin{definition}[Discrete conformal equivalence and discrete conformal maps]\label{dcdef} 
Two polyhedral surfaces or standard model surfaces $(S_1, V_1, d_1)$ and $(S_2, V_2, d_2)$ with their associated hyperbolic metrics or convex-hull metrics $\delta_1$ and $\delta_2$ are \underline{discrete conformal} if there exists an isometry $\phi$ from $(S_1 - V_1, \delta_1)$ to $(S_2 -V_2,\delta_2)$ such that $\phi$ can be extended to a homeomorphism $\bar \phi$ from $(S_1, V_1)$ to $(S_2, V_2).$ 
We call $\bar \phi$ a \underline{discrete conformal map} from $(S_1, V_1,d_1)$ to $(S_2, V_2, d_2).$ 

If $S_1=S_2=S$ and $V_1=V_2=V$, we say $(S, V, d_1)$ is \underline{marked discrete conformal} to $(S, V, d_2)$ if there exists a discrete conformal map $\phi: (S, V, d_1) \to (S, V, d_2)$ such that $\phi$ is isotopic to the identity map on $(S, V)$.
\end{definition}

From the definition, we see that for any finite sets $V_1 \subset \mathbb S^2$ and $V_2 \subset \E^2$, the two standard models $(\mathbb S^2, V_1, d_{\mathbb S})$ and $(\E^2, V_2, d_{\E})$ are not discrete conformal.

Note that in the case of the spherical background metric, unless we specify the choices of the associated hyperbolic metric $d^*$ or the convex-hull hyperbolic metric, the discrete conformal map is only well defined up to isotopy of the surface $(S, V, d)$.  

In Appendix~\ref{appendix:example}, we present Tianqi Wu's construction of  
a non-polyhedral standard model surface that is discrete conformal to a polyhedral surface.

 




\section{Delaunay tessellations of polyhedral surfaces}


In this section, we define the \underline{Delaunay tessellations} of polyhedral surfaces with spherical, hyperbolic, and Euclidean background metrics. The main observation is that the associated hyperbolic metric $d^*$ of a polyhedral surface $(S, V, d)$ depends only on the Delaunay tessellation and is independent of the choice of Delaunay triangulations. The proofs follow the approach developed in \cite{bobenkos2007}, where the Delaunay tessellations of closed polyhedral surfaces with Euclidean background metrics are studied.  

Recall that given a marked cone surface $(S,V,d)$, a \underline{circumdisk} in $(S, V, d)$ is a pair $(B, \phi)$, where $B$ is an open round disk with the condition that the radius of $B$ is less than $\pi/2$ if $d$ is spherical,  and $\phi: B \to (S-V, d)$ is an isometric immersion such that (i) $\phi$ can be extended continuously to a map, still denoted by $\phi$, from the closure $\overline{B}$ to $(S, d)$, and (ii) $\phi^{-1}(V)$ contains at least three points. Let $\{(B_\alpha,\varphi_\alpha)\}$ be the collection of all circumdisks in $(S, V, d)$,
and define  $\mathcal D$ to be the following collection of vertices, edges, and faces. 
\begin{enumerate}
    \item the set of vertices $V(\mathcal{D}) = V$, 
    
    \item the set of faces (2-cells) of $\mathcal{D}$ consists of the images under $\varphi_\alpha$ of the convex hull of $\varphi_\alpha^{-1}(V)$ in $\overline{B_\alpha}$  for any $(B_\alpha, \varphi_\alpha)$, and

    \item the set of edges of $\mathcal{D}$ consists of images of boundary edges of the convex hull of $\varphi_\alpha^{-1}(V)$ in $\overline{B_\alpha}$ under $\varphi_\alpha$ for any $(B_\alpha, \varphi_\alpha)$.  
\end{enumerate}

Notice that the faces, edges, and vertices are compact subsets of $S$, since closed circumdisks are compact. Furthermore, the convex hull $C_{\mathbb K}(\varphi_\alpha^{-1}(V))$ is a finite-sided polygon in $\overline{B_{\alpha}} \subset \mathbb{K}^2$ for $\mathbb{K}=\H, \E$, and $\mathbb{S}$, since $\varphi_\alpha^{-1}(V)$ is a finite set. Indeed, if $\varphi_\alpha^{-1}(V)$ were infinite, some vertex $v\in V$ would have infinitely many preimages under $\varphi_{\alpha}$, giving rise to a sequence of geodesic loops in $S$ based at $v$ with lengths converging to zero. This contradicts the fact that $v$ admits a cone neighborhood.
 

In general, there may not exist any circumdisk in the surface $(S, V, d)$. However, if 
$(S, V, d)$ is polyhedral, then $\mathcal{D}$ is nonempty since circumdisks of triangles in a Delaunay triangulation are in $\mathcal D$.  The goal  of this section is to prove the following. 

\begin{proposition}
\label{tessellation} Suppose  $(S, V, d)$ is a marked surface with a hyperbolic, Euclidean, or spherical background cone metric $d$. Then $(S, V, d)$ is a polyhedral surface if and only if the collection $\mathcal D$ forms a CW-decomposition of the surface $S$.  Furthermore, $\T$ is a Delaunay triangulation of $(S, V, d)$ if and only if $\T$ is a geometric subdivision of $\mathcal D$ with vertex set $V(\T) =V.$
\end{proposition}

We remark that Proposition \ref{tessellation} was first proved in 
\cite{bobenkos2007} for closed Euclidean cone metric surfaces. Our proof follows their proof.

In view of this proposition and following \cite{bobenkos2007}, when $(S, V, d)$ is polyhedral, we call $\mathcal D$ the \underline{Delaunay tessellation} of the surface.

We begin by introducing some notation and establishing basic properties of the edges and faces of $\mathcal D$ for general cone metric surfaces $(S,V,d)$. Throughout, for distinct points $p \neq q$ in $\mathbb K^2$, where $\K =\E, \H$ or $\mathbb S$,  we use $[p,q]$ and $(p,q)$ to denote a geodesic segment from $p$ to $q$ and its interior. 

\begin{lemma}
\label{edgenointer} For a marked surface $(S,V, d)$ with a cone metric $d$, the interior of every edge of $\mathcal{D}$ is embedded, and the interiors of distinct edges are disjoint. 
\end{lemma}
\begin{proof}
    For polyhedral surfaces with Euclidean background metrics, this has been proved in \cite{bobenkos2007}.  We only need to deal with polyhedral surfaces with spherical and hyperbolic background metrics. We start with the hyperbolic case. 

Suppose $e_i=\phi_i([v_i, w_i])$ are edges in $\mathcal D$ contained in the circumdisks $\phi_i: B_i \to (S, V, d)$ for $i=1,2$ with $v_i, w_i \in \partial B_i \cap \phi^{-1}_i(V)$. There are two cases to consider:
\begin{enumerate}
    \item $e_1 \neq e_2$, in which case we must show that $\phi_1((v_1, w_1)) \cap \phi_2((v_2, w_2)) =\varnothing$, and
    \item $e_1=e_2$, in which case we must show that $\phi_1|_{(v_1, w_1)}$ is injective. 
\end{enumerate}

For case (1), suppose, to the contrary, that there exist $q_i \in (v_i, w_i)$, for $i=1,2$, such that $\phi_1(q_1)=\phi_2(q_2)$. Since $e_1 \neq e_2$, $\phi_i(v_i, w_i) \cap V =\varnothing$, and $\phi_i(v_i), \phi_i(w_i) \in V$, we see that $\phi_1([v_1, w_1])$ intersects  $\phi_2([v_2, w_2])$ transversely at $\phi_1(q_1)$. 
We claim that
\begin{equation} \label{chord1}
\tanh \frac{d(q_1, v_1)}{2}\tanh \frac{d(q_1, w_1)}{2} = \tanh\frac{d(q_2, v_2)}{2}\tanh\frac{d(q_2, w_2)}{2}.
\end{equation}
If not, we may assume
\begin{equation} \label{chord11}
\tanh \frac{d(q_1, v_1)}{2}\tanh \frac{d(q_1, w_1)}{2} > \tanh\frac{d(q_2, v_2)}{2}\tanh\frac{d(q_2, w_2)}{2}.
\end{equation}
Since both $\phi_1$ and $\phi_2$ are local isometries with $\phi_1(q_1) =\phi_2(q_2)$, there exists an isometry $f \in \Isom(\H^2)$ such that $f(q_1)=q_2$ and $\phi_1=\phi_2 \circ f$ 
in a neighborhood of $q_1$. By the connectedness of $\overline{B_1} \cap f^{-1}(\overline{B_2})$, it follows that
\begin{equation}\label{chord2}
\phi_1(x)=\phi_2 (f(x)), 
\quad \text{for all $x \in \overline{B_1} \cap f^{-1}(\overline{B_2})$}.
\end{equation}

Consider the geodesic segment $f^{-1}([v_2, w_2]):=[v_2', w_2']$ which intersects $[v_1, w_1]$ transversely at $q_1$. Note that  $v_2'$ and $w_2'$ are outside of $B_1$. If otherwise, say $v_2'\in B_1$, then $\phi_1(v_2')=\phi_2(v_2) \in V$. This contradicts $\phi_1^{-1}(V) \cap B_1 =\varnothing$.
In particular, $[v_2', w_2']$ intersects $\partial B_1$ at two distinct points $v_2''$ and $w_2''$ such that  $v_2', v_2'', q_1, w_2'', w_2'$ occur in the order along $[v_2', w_2']$. See Figure \ref{chord6}. This implies
\begin{equation}\label{chord 3}
d(q_1, v_2') \geq d(q_1, v_2''), \quad \text{and} \quad  d(q_1, w_2') \geq d(q_1, w_2'')
\end{equation}

\begin{figure}[htbp]  
\begin{center}
\begin{overpic}[width=0.7\linewidth]{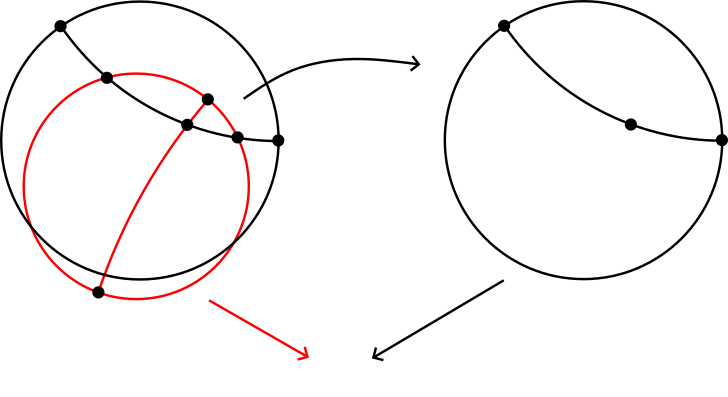}
  \put(13,39){$v''_2$}
  \put(3.5,52){$v'_2$}
  \put(40,35){$w'_2$}
  \put(28,31){$w''_2$}
  \put(14,49){$f^{-1}(B_2)$}
  
  \put(5,29){\textcolor{red}{$B_1$}}
  \put(12,11){$v_1$}
  \put(28,43){$w_1$}
  \put(21,36){$q_1$}
  
  \put(43,49){$f$}
  \put(41,42){$f(q_1)=q_2$}
  
  \put(64,52){$v_2$}
  \put(83,35){$q_2$}
  \put(101,35){$w_2$}
  \put(78,19){$B_2$}
  
  \put(35,12){\textcolor{red}{$\phi_1$}}
  \put(57,13){$\phi_2$}
  \put(43,1){$(S, V, d)$}
\end{overpic}

\caption{Two intersecting edges.}  
\label{chord6}
\end{center}
\end{figure}
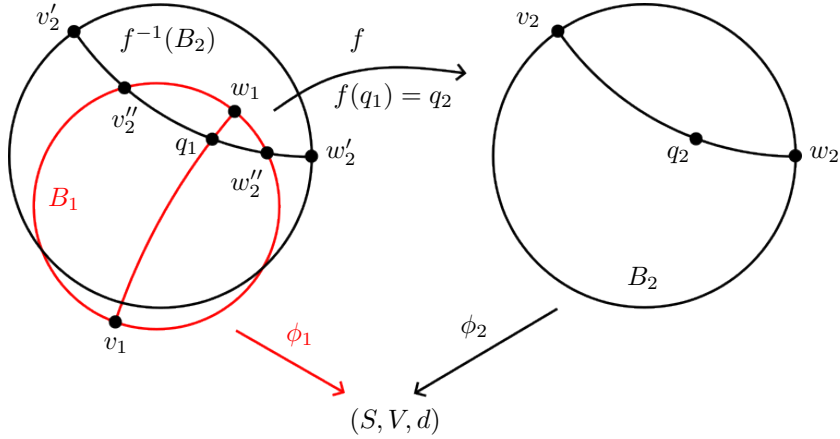

Since $v_1, v_2'', w_1, w_2''$ are on the circle $\partial B_1$, by the intersecting chords theorem in hyperbolic geometry, we have
\begin{equation} \label{chord4}
\tanh \frac{d(q_1, v_1)}{2}\tanh \frac{d(q_1, w_1)}{2} = \tanh\frac{d(q_1, v_2'')}{2}\tanh\frac{d(q_1, w_2'')}{2}.
\end{equation}

Together with \eqref{chord 3}, we obtain
\begin{equation}
    \begin{aligned}
        & \tanh \frac{d(q_1, v_1)}{2}\tanh \frac{d(q_1, w_1)}{2} = \tanh\frac{d(q_1, v_2'')}{2}\tanh\frac{d(q_1, w_2'')}{2}\\
        \leq & \tanh\frac{d(q_1, v_2')}{2}\tanh\frac{d(q_1, w_2')}{2}
=\tanh\frac{d(q_2, v_2)}{2}\tanh\frac{d(q_2, w_2)}{2}.
    \end{aligned}
\end{equation}
But this contradicts the assumption \eqref{chord11}. 
It follows that \eqref{chord1} holds and that both equalities in \eqref{chord 3} also hold. In particular, we see that $v_2'=v_2''$, $w_2'=w_2''$, and both of them are in $\partial B_1 \cap \phi_1^{-1}(V)$.  Since $[v_2', w_2']$ intersects $[v_1, w_1]$ transversely at an interior point $q_1$, it follows that the open geodesic segment $(v_1, w_1)$ is in the interior of the convex hull $C(\phi_1^{-1}(V))$ in $B_1$. This contradicts the assumption that $e_1$ is an edge.

For case (2), we take $\phi_2=\phi_1$, $B_2=B_1$, $v_2=v_1$ and $w_2=w_1$. Suppose the conclusion does not hold, i.e., $\phi_1(q_1)=\phi_1(q_2)$ for two distinct points $q_1, q_2 \in (v_1, w_1)$. By the same argument as in case (1), we find an isometry $f \in \Isom(\H^2)$ such that $f(q_1)=q_2$ and $\phi_1 \circ f =\phi_1=\phi_2$ in a neighborhood of $q_1$. Let $L$ be the geodesic through $v_1$ and $w_1$. Then $\phi_1 =\phi_1 \circ f$ implies $f(L)=L$. It follows that $f|_L$ must be a reflection, or a non-trivial translation along $L$. A reflection has a fixed midpoint inside the relevant interval and contradicts the local injectivity of $\phi_1$. A non-trivial translation has displacement smaller than the edge length; after replacing $f$ by $f^{-1}$ if needed, it moves an endpoint into the interior. Then equation $\phi_1 =\phi_1 \circ f^{\pm1}$ shows that $\phi_1$ maps that interior point to a vertex, contradicting $\phi_1^{-1}(V) \cap B_1 =\varnothing.$

  Notice that in the case of surfaces $(S, V, d)$ with spherical background metrics, all circumdisks are contained in some open hemisphere. 
   By the intersecting chords theorem in spherical geometry, two chords $[p,q]$, $[x,y]$ in a spherical ball that intersect at $z$ satisfy
   $$
   \tan \frac{d(p, z)}{2}\tan \frac{d(z, q)}{2} = \tan\frac{d(x, z)}{2}\tan\frac{d(y, z)}{2}.
   $$
    Therefore, the above proof still works in this case since $\tan x$ is a strictly increasing function on $[0, \pi/2)$ and the lengths of edges of geometric triangulations of a polyhedral surface with a spherical background metric are in the open interval $(0, \pi)$. \end{proof}

An open face in $\mathcal D$ is defined to be $\varphi_{\alpha}(C_{\K}(\varphi_{\alpha}^{-1}(V))^{\circ})$. 
\begin{corollary}
\label{edgefaceinter}
    Each edge is disjoint from every open face in $\mathcal D$.
\end{corollary}

\begin{proof}
    Note that such an intersection implies either that two edges intersect in their interiors, which is impossible by Lemma \ref{edgenointer}, or that an edge is contained in an open face. 

    In the latter case, suppose $\phi: D \to S$ is the circumdisk of the open face $F$, and that $e$ is an edge in $F$ covered by another circumdisk $\phi_1: D_1 \to S$. Since $e$ lies inside $F$, it is not a boundary edge of $F$. Furthermore, since its vertices cannot lie in the circumdisk $D$, it follows that $e$ must be a diagonal of $F$. Let $E$ be the chord of $D$ connecting two points in $\phi^{-1}(V)$ such that $\phi(E)=e$.
    Then $E$ divides $\partial D$ into two open arcs, denoted by $A_1$ and $A_2$. Then $A_1$ and $A_2$ each contain at least one vertex of the face $F$.

Since $\phi(D)$ and $\phi_1(D_1)$ contain the same edge $e$, we may use an isometry $g$ of $\K^2$ to replace the circumdisk $(D_1, \phi_1)$ by $(g(D_1), \phi_1\circ g^{-1})$ such that (i) $E$ is a chord of $g(D_1)$, (ii) $\phi_1 \circ g^{-1}|_E =\phi|_E$, and (iii) $\phi_1 \circ g^{-1} =\phi$ in a neighborhood of $E$ using an additional reflection about $E$ if necessary.  For simplicity, we assume that $g=id$.  Therefore, there exists an isometric immersion $\hat{\phi}: D\cup D_1\to S$ such that $\hat{\phi}|_D=\phi$ and $\hat{\phi}|_{D_1}=\phi_1$.
    If $D\neq D_1$, we know that $D_1$ contains either $A_1$ or $A_2$. Thus $\phi_1(D_1)=\hat{\phi}(D_1)$ contains some vertex, which is impossible for an empty-disk. Therefore, $(D,\phi)$ and $(D_1,\phi_1)$ are the same circumdisk.
\end{proof}

The following lemma shows that the family $\{(\varphi_\alpha, C_{\K}(\varphi^{-1}_{\alpha}(V)))\}$ serves as the attaching maps for the faces. 

\begin{lemma}
\label{facehomeo} For a surface $(S, V, d)$ with a cone metric, let $(B, \varphi)$ be a circumdisk with $\varphi^{-1}(V) = \{p_1, p_2, \cdots, p_n\}$ and $n \geq 3$, and $P = C_{\mathbb K}(p_1, p_2, \cdots, p_n)$ be the convex hull in $\overline{B}$ with interior $P^\circ$.  Then $\varphi|_{P^\circ}$ is a homeomorphism.
Furthermore, if two faces intersect in their interiors, they are the same face. 
\end{lemma}
\begin{proof}
    It suffices to show $\varphi|_{P^\circ}$ is injective.
    If not, then there exist $p \neq p'\in {P}^\circ$ such that $\varphi(p) = \varphi(p')$ in $S$. Since $\varphi$ is an isometric immersion,  there exists a neighborhood $U$ of $p$ in $B$ and an isometry $f$ of $\mathbb{H}^2$, or $\E^2$, or $\mathbb{S}^2$ such that $f(p)=p'$, $f(U)\subset B$, and for any $z\in U$, $\varphi(z)=\varphi(f(z))$. 
    Then $\varphi\circ f^{-1}: f(B)\to S$ is also a circumdisk. By construction of $f$,  $\varphi(z)=\varphi \circ f^{-1} (z)$ for $z\in f(U)$.
    Since $f(U)\subset B\cap f(B)$ is an open subset and $B \cap f(B)$ is connected, we see that 
    \begin{equation} \label{phiequ1}
    \varphi =\varphi \circ  f^{-1}, \quad \text{ on $\overline{B}\cap f(\overline{B})$}. \end{equation}

Note that $p_i \notin f(B)$. For otherwise, $p_i =f(q_i)$ with $q_i \in B$. Furthermore, $p_i \in \overline{B} \cap f(\overline{B})$. By \eqref{phiequ1}, $\varphi(q_i) =\varphi(f^{-1}(p_i)) =\varphi(p_i) \in V$.  It follows that $q_i \in B \cap \varphi^{-1}(V) =\varnothing,$ contradicting the empty-disk condition. Thus $\{p_1, ..., p_n\} \subset \partial B -f(B)$. 
Similarly, since $\{f(p_1), ..., f(p_n)\}$ is the preimage of $V$ under $\varphi \circ f^{-1}$ in $f(\overline{B})$, we have $f(p_i) \notin B$. In particular, $\{f(p_1), ..., f(p_n)\} \subset \partial f(B) - B$.

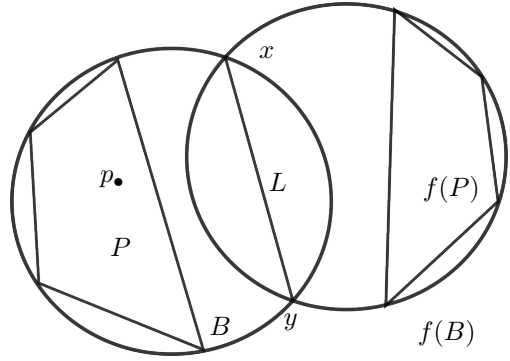
\begin{figure}[ht!]  
\begin{center}
\begin{overpic}[scale=1]{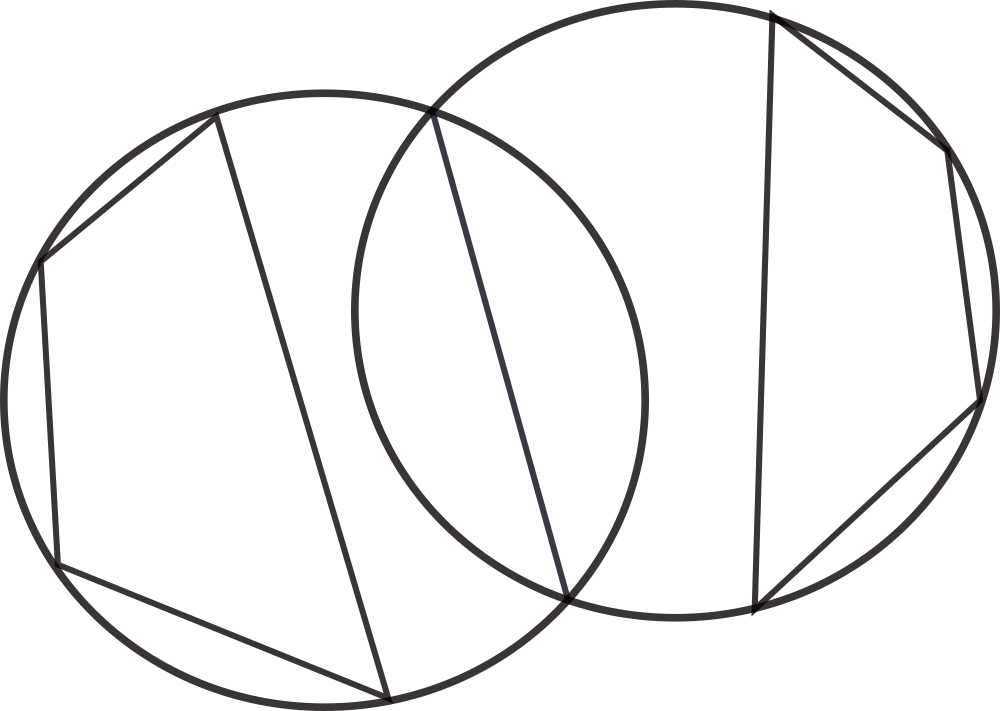}
   \put(20,20){$P$}
   \put(18,35){$p$ \makebox(0,0){\color{black}\circle*{1.5}}}
   \put(40,4){$B$}
   \put(83,32){$f(P)$}
    \put(82,3){$f(B)$}
    \put(50,60){$x$}
    \put(55,6){$y$}
    \put(52,33){$L$}
\end{overpic}
\caption{Two intersecting faces.}  
\label{nointersection}
\end{center}
\end{figure}

If $B=f(B)$, then $f$ is either a rotation or a reflection. Hence $f$ fixes the center $O$ of $B$. Then for any neighborhood $V$ of $O$, there exists $x\in V-\{O\}$ such that $f(x)\in V-\{O, x\}$. Thus $\varphi(x)=\varphi(f(x))$, which contradicts the fact that $\varphi$ is an immersion. Thus, $B$ and $f(B)$ are not identical.

    Since $f$ is an isometry, $B\cap f(B)\supset f(U)\neq \varnothing$, and $B$ and $f(B)$ do not contain each other,  we know that $\partial\overline{B}\cap f(\partial\overline{B}) = \{x, y\}$, where $x, y$ are not antipodal points in the case of $\mathbb{K}^2=\mathbb S^2$ since $B$ is an open ball of radius less than $\pi/2$,  as illustrated in Figure \ref{nointersection}. 
Consider the geodesic $L$ through $x, y$ which cuts $\partial (B \cap f(B))$ into two parts $f(B) \cap \partial B$ and $B \cap \partial f(B)$. By the condition $\{p_1, ..., p_n\} \cap (f(B)\cap \partial B ) =\varnothing$ and $\{f(p_1), ..., f(p_n)\} \cap (B \cap \partial f(B)) =\varnothing$, we see that $\{p_1, ..., p_n\}$ lies in a closed half-space bounded by $L$,  and $\{f(p_1), ..., f(p_n)\}$ lies in the other closed half-space bounded by $L$. 
It follows that the interiors of their convex hulls $P^\circ$ and $f(P)^\circ$ lie on different sides of $L$. In particular, $P^\circ \cap f(P)^\circ =\varnothing$. But that contradicts $p' \in P^\circ \cap f(P)^\circ.$

Finally, we verify that the interiors of distinct faces are disjoint. Assume the opposite, and let $(B_1, \varphi_1)$ and $(B_2, \varphi_2)$ be two distinct circumdisks corresponding to the faces $F_1$ and $F_2$ of $\mathcal{D}$ such that  there are two points $p\in F_1^\circ\cap F_2^\circ$ and $q\in {F_1}^\circ - {F_2}^\circ$, without loss of generality. Consider the unique geodesic $\gamma$ in $\varphi_1^{-1}(F_1)$ connecting $\varphi_1^{-1}(p)$ and $\varphi_1^{-1}(q)$.
    Then $\varphi_1(\gamma)$ is an arc in $S$ connecting $p$ and $q$, and hence intersects $\partial F_2$. Since both $p$ and $q$ are interior points of $F_1$, we have $\varphi_1(\gamma)\subset {F}_1^\circ$.
    Thus $\partial F_2$ intersects ${F}_1^\circ$, and thus an edge of $F_2$ intersects ${F}_1^\circ$, which contradicts Corollary \ref{edgefaceinter}.
\end{proof}

With the above preparations, we now prove Proposition \ref{tessellation}.

\begin{proof}[Proof of Proposition \ref{tessellation}]
Suppose $(S, V, d)$ is polyhedral with a Delaunay triangulation $\T$. By definition, the circumdisk $D_{\tau}$ of a triangle $\tau$ in $\T$ is a circumdisk in $\mathcal D$ and $\tau \subset D_{\tau}$. Since $S$ is a union of triangles $\tau \in \T$, it follows that each point $p \in S$ is contained in a face associated with the circumdisk $D_{\tau}$ for $\tau \in \T$. Therefore, by Lemma \ref{facehomeo}, we conclude that each circumdisk $D$ is of the form $D_{\tau}$ for some $\tau \in \T$. By Lemmas \ref{edgenointer}, \ref{facehomeo}, and Corollary \ref{edgefaceinter}, $\mathcal D$ forms a 
CW-decomposition of the surface $S$. Furthermore, each face $D$ in $\mathcal D$ is triangulated by triangles $\tau$ such that $D_{\tau} =D$, i.e., $\T$ is a subdivision of $\mathcal D$ with vertex set $V(\T)=V$.

Conversely, suppose $\mathcal D$ is a CW-decomposition of the surface $S$. Since each face $F$ of $\mathcal D$ is the isometric image of a finite-sided convex cyclic polygon, we can triangulate $\mathcal D$ by adding geodesic diagonals to $F$. This produces a geodesic triangulation $\T$ of the surface $(S, V, d)$. By construction, for each triangle $\tau \in \T$, its circumdisk $D_{\tau}$ is a circumdisk in $D$, and hence is an empty-disk whose closure is compact. Therefore, $\T$ is a Delaunay triangulation, and $(S, V, d)$ is polyhedral.
\end{proof}


\begin{corollary}
    \label{triagulationrefine}
    Suppose $(S, V, d)$ is a polyhedral surface and $\mathcal{T}$ is a Delaunay triangulation of $(S, V, d)$. Then, 
    \begin{enumerate}[label=(\arabic*)]
        \item the hyperbolic metric $d^*$ on $S-V$ associated with $\mathcal{T}$ depends only on the Delaunay tessellation of $(S, V, d)$; and
        \item if $\phi: (S, V, d) \to (S, V, d)$ is an isometry of the polyhedral surface, then $\phi$ preserves the Delaunay tessellation and induces a canonical isometry $\phi^*=\phi|_{S-V}: (S-V, d^*) \to (S-V, d^*).$ In particular, $(\Id)^*=\Id$ and $(\phi \circ \psi)^* =\phi^* \circ \psi^*.$
    \end{enumerate}
\end{corollary}

\begin{proof}
Let $\mathcal{D}$ be the Delaunay tessellation with 2-cells $\{D_1, \cdots, D_n, \cdots\}$. By construction, the interior of $D_i$ is isometric to the interior of a convex polygon $F_i$ in $\mathbb{K}^2$ which is inscribed in a circle, and the surface $(S, V, d)$ is isometric to the surface obtained by gluing the convex polygons $F_i$ isometrically along pairs of edges. 
For simplicity, let us assume $(S, V, d)$ is a Euclidean cone metric surface. One constructs a hyperbolic metric $d^*_{\mathcal{D}}$ in a similar way by replacing each cyclic convex polygon $F_i$ with vertices $\{u_1, ..., u_k\}$ by the hyperbolic convex hull $F_i^*=C(u_1, ..., u_k) \subset \H^3_U$ and isometrically gluing $F_i^*$ to $F_j^*$ along their common edge by the lifted hyperbolic isometry $f^*$ of the Euclidean isometry $f$ associated with the gluing of $F_i$ and $F_j$. See Figure \ref{777}.

We claim that $d^*_{\mathcal{D}}=d^*_{\T}$ for any Delaunay triangulation $\T$. Indeed, since $\T$ is a refinement of $\mathcal{D}$ without any extra vertices, each 2-cell $F_i$ is subdivided into a collection of triangles $\tau_1, ..., \tau_{k-2}$ in $\T$. Since $F_i$ is inscribed in a circle, we see that $\tau_1^* \cup ... \cup \tau_{k-2}^* =F_i^*$. It follows from the definition that $d^*_{\mathcal{D}}=d^*_{\T}$.

Finally, for part (2), since $\mathcal{D}$ is unique and invariant under isometries, we see that $\phi(\mathcal{D}) =\mathcal{D}$. By Lemma \ref{natural-h}, since $d^*$ is obtained by pulling back to $S-V$, via $\phi_\K$, the metric on the ideal hyperbolic polygons $\sigma^*$ associated with the 2-cells $\sigma$ of $\mathcal{D}$ ,
we see that the restriction $\phi^*=\phi|_{S-V}$ is an isometry of $d^*$.
It then follows that $(\Id)^*=\Id$ and $(\phi \circ \psi)^* =\phi^* \circ \psi^*.$
\end{proof}

A basic property of a polyhedral surface is that each circumdisk contains only finitely
many edges. 
\begin{proposition} \label{finiteness}
   Suppose $(S, V, d)$ is a polyhedral surface and $\mathcal{T}$ is a Delaunay triangulation of $(S, V, d)$ with $E = E(\mathcal{T})$ being the set of all edges. If $\Phi: B \to (S-V, d_{S-V})$ is an immersed empty circumdisk, then the set $\Phi^{-1}(\cup_{e\in E}{e}^\circ)$ has only finitely many connected components. 
\end{proposition}
\begin{proof}
   We prove it by contradiction. Suppose $\Phi^{-1}(\cup_{e\in E} {e}^\circ)$ has infinitely many components $\{{I_n}^\circ, n\in\Z_+\}$, where $I_n$ are geodesic segments $[x_n, y_n]$ with endpoints $x_n, y_n$ in $\partial B$.
    
   Given a vertex $v$ in $\T$, let the open star $\st(v)$ of $v$ be the union of all open simplices $\tau$ having $v$ as a vertex.  Then $\st(v)$ is an open set in $S$.  Since $\Phi(\overline{B})$ is compact, it is covered by finitely many open stars $\st(v)$. Moreover, since the open cover $\{ \st(v) : v \in V(\T)\}$ is locally finite and any two open stars $\st(v)$ and $\st(v')$ intersect if and only if $v$ and $v'$ are adjacent, we know that
   $\Phi(\overline{B})$ intersects only finitely many open stars. Hence, $\Phi(\overline{B})$ intersects only finitely many edges $e_1, ..., e_m$ in $\T$. By passing to a subsequence, we may assume that all these  disjoint intervals $I_n$ are mapped by $\Phi$ into one edge, say, $e_1$, i.e., $\Phi(I_n) \subset e_1$ for all $n$. 
    
   By compactness of $\overline{B}$, after passing to a subsequence again, we may assume that $\lim_{n\to\infty} x_n = x$ and $\lim_{n\to\infty} y_n = y$ in $\overline{B}$.  
   Let $z_n$ be the midpoint of $x_n$ and $y_n$, and $z$ the midpoint of $x$ and $y$.       
   Then $\lim_{n\to\infty} z_n = z$. Let $p = \Phi(z) \in e_1$ in $S$. 
   By the construction, we have
   \begin{equation} \label{disjoint}
   [z_n, z_{n+1}] \not\subset  \Phi^{-1}(e_1). \end{equation}
    
   Since $(S, V, \mathcal{T}, d)$ is a geometrically triangulated cone metric surface, each point $q \in S$ has a small ball neighborhood $N(q)$ with the following properties, as illustrated in Figure \ref{fig:nbhd of q}:
   \begin{enumerate}[label=(\arabic*)]
       \item if $q$ is not a vertex of $\T$, then $N(q) \subset S-V(\T)$, and for any pair of distinct points $u, v \in N(q)$, there is only one geodesic segment, namely, $[u,v]$, joining $u$ to $v$ in $N(q)$, and 

       \item if $q \in V(\T)$, then for any $u \in N(q) -\{q\}$, there exists only one geodesic segment, namely, $[u, q]$, joining $u$ to $q$ in $N(q).$
   \end{enumerate}

\begin{figure}[htbp]  
\begin{center}
\begin{overpic}[width=0.5\linewidth]{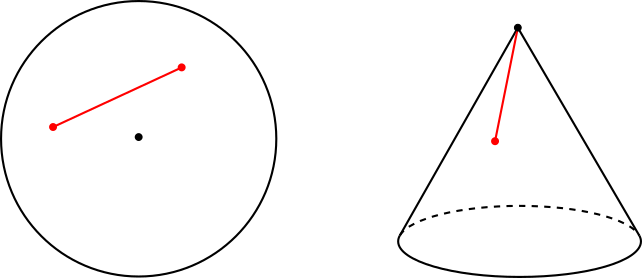}
  \put(23,19){$q$}
  \put(3.5,24){$u$}
  \put(30,30){$v$}
  \put(74,18){$u$}
  \put(82,39){$q$}
\end{overpic}
\caption{Small ball neighborhood of an arbitrary point $q\in S$.}  
\label{fig:nbhd of q}
\end{center}
\end{figure}
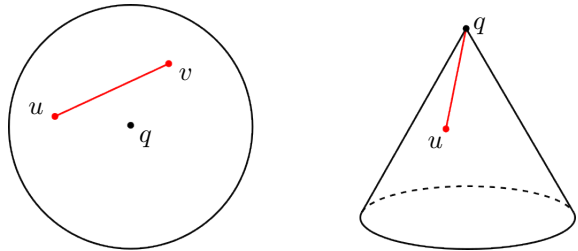

Let $N$ be the small ball neighborhood of $p=\Phi(z) \in e_1$ in $S$ as above.  If $p$ is not a vertex of $\T$, then, since $\Phi(I_n) \subset e_1$, we see that $p$ is in the interior of $e_1$. We may choose $N$ to be sufficiently small so that $N \cap e_1$ consists of one geodesic segment. For sufficiently large $n$, convergence implies that $\Phi(z_n)$ and $[\Phi(z_n), \Phi(z_{n+1})]$ are contained in $N$. Since $\Phi$ is a local isometry, $ \Phi([z_n, z_{n+1}])$ is a geodesic segment in $N$ joining $\Phi(z_n) \in e_1$ to $\Phi(z_{n+1}) \in e_1$. Since $e_1 \cap N$ is the only geodesic segment in $N$ that can join two points in $e_1$,  it follows that $\Phi([z_n, z_{n+1}]) =[\Phi(z_n), \Phi(z_{n+1})] \subset e_1$. But this contradicts \eqref{disjoint}. See the left-hand pictures in Figure \ref{fig:nbhd of p}.
    
If $p$ is an endpoint of $e_1$, consider the geodesic segment $[p, \Phi(z_n)]$ in $N$. Since $N$ is the small ball neighborhood of $p$, we see that $\Phi([z, z_n]) = [p, \Phi(z_n)]$. Therefore, $\Phi([z, z_n]) \subset e_1$, i.e.,
$[z, z_n] \subset \Phi^{-1}(e_1)$. This shows that  $[z, z_n] \subset I_n$. Since the ${I_n}^\circ$ are disjoint, it follows that all intervals $I_n$ share the same endpoint $z$ in $\overline{B}$ for sufficiently large $n$. 
But this contradicts the fact that $\Phi: \overline{B} \to S$ is the continuous extension of an isometric immersion of $B$ sending $I_n$ into $e_1$. More precisely, if such a $\Phi$ exists, take a small circle tangent to $z$ in $\overline{B}$ and let $C$ be the circle with point $z$ removed.  Then $C$ intersects each $I_n^{\circ}$ at a point.  It follows that the image $\Phi(C)$ is a curve of finite length and constant geodesic curvature in the neighborhood $N(\Phi(z))-\{\Phi(z)\}$, which intersects $e_1$ at infinitely many points.  This is impossible. 
See the right-hand pictures in Figure \ref{fig:nbhd of p}. \qedhere

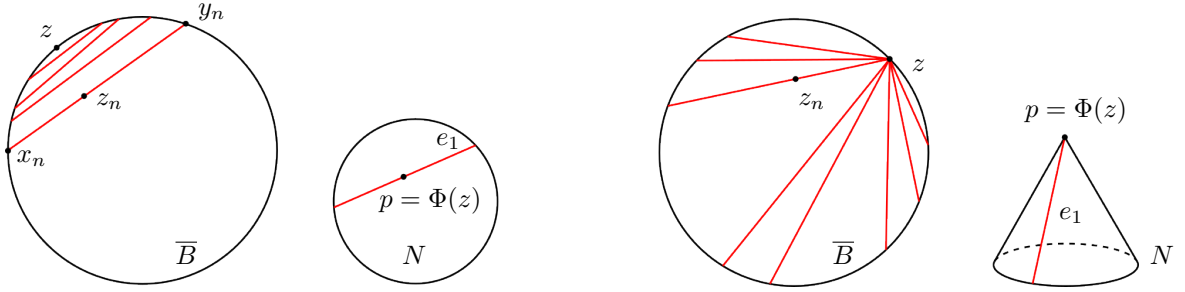
\begin{figure}[htbp]  
\begin{center}
\begin{overpic}[width=\linewidth]{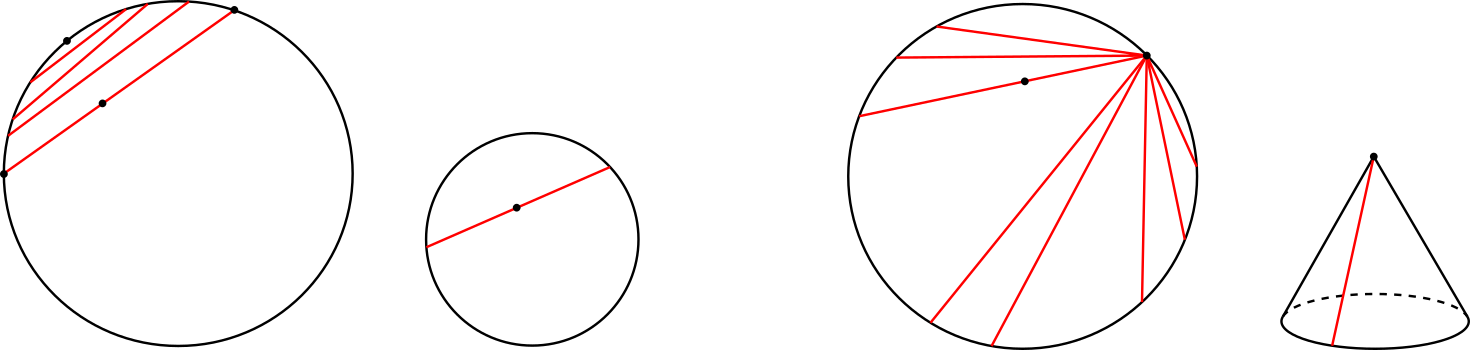}
  \put(17,24){$y_n$}
  \put(3,22){$z$}
  \put(8,16){$z_n$}
  \put(1,11){$x_n$}
  \put(15,2){$\overline{B}$}
  \put(38,12.5){$e_1$}
  \put(33,7){$p=\Phi(z)$}
  \put(35,2){$N$}
  \put(70,16){$z_n$}
  \put(80,19){$z$}
  \put(73,2){$\overline{B}$}
  \put(101,2){$N$}
  \put(90,15){$p=\Phi(z)$}
  \put(93,6){$e_1$}
\end{overpic}

\caption{Small ball neighborhood and the local isometric map $\Phi$ near the limit point $p=\Phi(z)$.}  
\label{fig:nbhd of p}
\end{center}
\end{figure}
\end{proof}

The following proposition illustrates that passing to a covering space does not change the fundamental properties of polyhedral surfaces.

\begin{proposition} \label{covering}
Suppose $(\tilde S, \tilde V, \tilde d)$ is a covering space of a marked cone metric surface $(S, V, d)$ with the pull-back vertex set $\tilde V$ and pull-back cone metric $\tilde d$.
\begin{enumerate}[label=(\roman*)]
    \item The surface $(\tilde S, \tilde V, \tilde d)$ is polyhedral if and only if $(S, V, d)$ is.
    \item Suppose $(S, V, d)$ is polyhedral with the associated hyperbolic metric $d^*$, and $\tilde{d}^*$ is the associated hyperbolic metric for $(\tilde S, \tilde V, \tilde d)$. Then $\tilde d^*$ is the pullback of $d^*$ under the covering map,  and $\tilde d^*$ is a complete hyperbolic metric if and only if $d^*$ is. 
\end{enumerate}
\end{proposition}

\begin{proof} To see (i), if $(S, V, d)$ is polyhedral with a Delaunay triangulation $\T$, then lifting triangulation of $\T$ and lifting compact circumdisks show that $(\tilde S, \tilde V, \tilde d)$ is polyhedral. To see the converse, we may assume that $\tilde S$ is the universal cover of $S$ by passing to the universal cover. By 
Proposition \ref{tessellation}, the Delaunay tessellation 
on $(\tilde S, \tilde V, \tilde d)$ is invariant under every deck transformation, so its tessellation descends to $(S, V, d)$. A non-trivial deck transformation cannot stabilize any of its finite convex polygons or open edges: its induced finite polygonal symmetry would have a fixed interior point, whereas the deck transformation action is free. The descended cells consequently give the required CW-decomposition of $(S, V, d)$. By Proposition \ref{tessellation}, $(S, V, d)$ is polyhedral. 

To see part (ii), recall a well-known theorem in Riemannian geometry (see, for instance, P. Petersen \cite[Chapter 5, Proposition 23]{peterson2006}) that a Riemannian manifold $(M, g)$ is complete if and only if its Riemannian covering space $(\tilde{M}, \tilde{g})$ is complete.  We claim that $(\tilde{S} -\tilde V, \tilde{d}^*)$ is a Riemannian covering space of $(S-V, d^*)$. The result then follows immediately.

Let $\pi:(\tilde S, \tilde V, \tilde{d})\to (S, V, d)$ be a covering map. It induces a covering map $\pi^*:(\tilde S-\tilde V, \tilde{d}^*) \to (S-V, d^*)$ as follows. Assume that $\mathcal T$ is a Delaunay triangulation of $(S, V, d)$, which lifts to a Delaunay triangulation $\tilde{\mathcal{T}}$ by Part (i). Recall from Lemma \ref{natural-h} that the constructions of $(S-V, d^*)$ and $(\tilde S-\tilde V, \tilde{ d}^*)$ use the homeomorphisms $\phi_\tau$ for triangles $\tau$ in $\mathcal{T}$ and $\tilde{\mathcal{T}}$, respectively. Let $\tilde{\tau}$ be a lift of $\tau$ in $\tilde{\mathcal{T}}$. Then the restriction of $\pi$ to $\tilde{\tau}$ is an isometry onto $\tau$. By Lemma \ref{natural-h}, it induces an isometry restricted to $\tilde \tau$ of $(\tilde S - \tilde V, \tilde d^*)$ to $(S-V, d^*)$. Similarly, let $\sigma$ be a triangle of $\mathcal{T}$ lifting to $\tilde \sigma$ sharing an edge with $\tilde{\tau}$. Then $\pi$ again restricts to an isometry from $\tilde{\tau}\cup\tilde{\sigma}$ to $\tau\cup\sigma$, inducing an isometry restricted to $\tilde{\tau}\cup\tilde{\sigma}$ from $(\tilde S - \tilde V, \tilde d^*)$ to $(S-V, d^*)$. Therefore, since every $x\in (S-V, d^*)$ is either in the interior of some triangle $\tau$ or in some edge shared by $\tau$ and $\sigma$, $\pi$ induces a local isometry $\pi^*$ from $\tilde{\tau}\cup\tilde{\sigma}$ to $\tau\cup\sigma$. It is also straightforward to see that $\pi^*$ is a covering map since $\pi$ is a covering map and we can choose a sufficiently small, evenly covered neighborhood $U$ of $x$ such that $U\cap V = \varnothing$. This completes the proof of the claim that  $(\tilde{S} -\tilde V, \tilde{d}^*)$ is a Riemannian covering space of $(S-V, d^*)$.
\end{proof}

\section{The completeness of the associated hyperbolic metric}

In this section, we prove:
\begin{theorem} \label{thm:completeness}
Suppose $(S, V, d)$ is a polyhedral surface whose background metric satisfies one of the following:
\begin{enumerate}
    \item  $d$ is a hyperbolic cone metric;
    \item  $d$ is a Euclidean cone metric, and the diameters of its  circumdisks are uniformly bounded above; or
    \item  $d$ is a spherical cone metric, and the diameters of its  circumdisks are uniformly bounded away from $\pi$.
\end{enumerate}
Then the associated metric $d^*$ on $S-V$ is a complete hyperbolic metric with a cusp end at each $v \in V$. 
\end{theorem}

\begin{remark} 
Note that for a polyhedral surface $(S, V, \T, d)$ with a Euclidean background metric and a Delaunay triangulation $\T$, the uniform boundedness of the diameters of circumdisks is equivalent to the uniform boundedness of the edge lengths. 
\end{remark}

To verify the above remark, first note that the diameter of a circumdisk is an
upper bound for its edge lengths; hence, one direction is clear. On the other hand, suppose the edge lengths are uniformly bounded by $C$. If there is a circumdisk $f: D \to S$ with diameter at least $4C$ and center $O$, then we claim that there exists $z \in D$ with $d(z,O)\leq C$ such that $f(z)\in S$ lies on some edge. If no such $z$ exists, then $f(B(O,C))$ is contained in the interior of some triangle $\tau_1$; here $B(O,C)$ is the open disk centered at $O$ of radius $C$, which is contained in $D$, and $\tau_1$ is not necessarily the same as the triangle $\tau$ corresponding to $D$. Since $f$ is a local isometry and the interior of $\tau_1$ is an embedded triangle in $S$, $f|_{B(O,C)}$ is actually an embedding. Thus $\tau_1$ contains an embedded disk of radius $C$, which contradicts the assumption that the edge lengths of $\tau_1$ are bounded by $C$.
Let $L$ be a chord passing through $z$ such that $f(L)$ is contained in an edge $e$. By the same argument as in the proof of Lemma \ref{edgenointer}, we note that $f|_L$ is injective. Hence the length of $e$ is at least the length of $L$, which implies that the length of $L$ is bounded by $C$. Since $d(z,O)\leq C$, the diameter of $D$ is then bounded by $3C$, which contradicts the assumption.

\begin{remark}
    The assumptions on the circumdisk diameters for Euclidean and spherical cone metrics cannot be omitted in general.
    See Appendix \ref{appen:incomplete} for examples in which the associated hyperbolic metrics are incomplete if the corresponding assumption is dropped. On the other hand, there is a Delaunay triangulation of the Euclidean plane without a uniform bound on diameters of its circumdisks such that the associated hyperbolic metrics are complete.  It is a very interesting question to determine the necessary and sufficient conditions on Delaunay triangulations of polyhedral surfaces in  Euclidean background metrics such that their associated hyperbolic metrics are complete.
\end{remark}

\begin{proof}[Proof of Theorem \ref{thm:completeness}] The statement that the metric $d^*$ has a cusp end at each $v \in V$ when $d$ is a Euclidean or hyperbolic cone metric was proved in \cite[Section 2]{glsw} and \cite[Section 3]{gglsw}. The proof in \cite{gglsw} also works for polyhedral surfaces with spherical background metrics.

We begin with the case of hyperbolic background metrics.    
Let $\mathcal{T}$ be a Delaunay triangulation of a polyhedral surface $(S, V, d)$ with a hyperbolic background metric, and let $d^*$ be the metric constructed using $\mathcal{T}$. Instead of using the upper-half-space model $\H^3_U$, we will use the Poincar\'e model $\H^3_P$ to construct the ideal triangles and the metric $d^*$, replacing $\D\subset \partial \H^3_U$ with the upper hemisphere $\SS^2_+\subset \partial\H^3_P$.
To see this, we identify $\H^2_P$ with the upper hemisphere $\mathbb{S}^2_+$ by the inversion $\rho$ about the sphere centered at $(0,0,-1)\in\R^3$ of radius $\sqrt{2}$, as shown in the leftmost figure in Figure \ref{xyz2}, and then $\rho$ induces the hyperbolic metric $d_\H$ on $\SS^2_+$. Thus we regard each triangle $\tau\in\T$ as a hyperbolic triangle in $\SS^2_+$, while $\rho(\tau)$ lies in $\H^2_P$. The ideal triangle $\tau^*$ associated with $\tau$ is the ideal triangle in $\H^3_P$ with the same vertex set as $\tau$ in $\mathbb S^2_+$. See the two middle figures in Figure \ref{xyz2}.
If two triangles $\tau$ and $\sigma$ in $\T$ share a common edge $e$ in $\mathcal{T}$, $\tau^*$ and $\sigma^*$ are glued isometrically along the corresponding edges by the lift $f^* \in \Isom(\H^3_P)$ of the isometry $f$ that glues $\tau$ to $\sigma$ along $e$, as illustrated in the right figure in Figure \ref{xyz2}. Let $\pi\colon {\H^3_P} \to \H_P^2$ be the shortest-distance projection. It extends continuously to a map $\overline{\H^3_P}-\partial \H^2_P \to \H^2_P$, and we still denote it by $\pi$. Let $\phi=\rho\circ\pi$. Then its restriction to each $\tau^*$ is a homeomorphism to $\tau$ with vertices removed. Since $\pi$ is distance-decreasing and the hyperbolic metric on $\SS^2_+$ is induced by the inversion $\rho$, we know that $\phi$ is also distance-decreasing; that is, it is $1$-Lipschitz. 

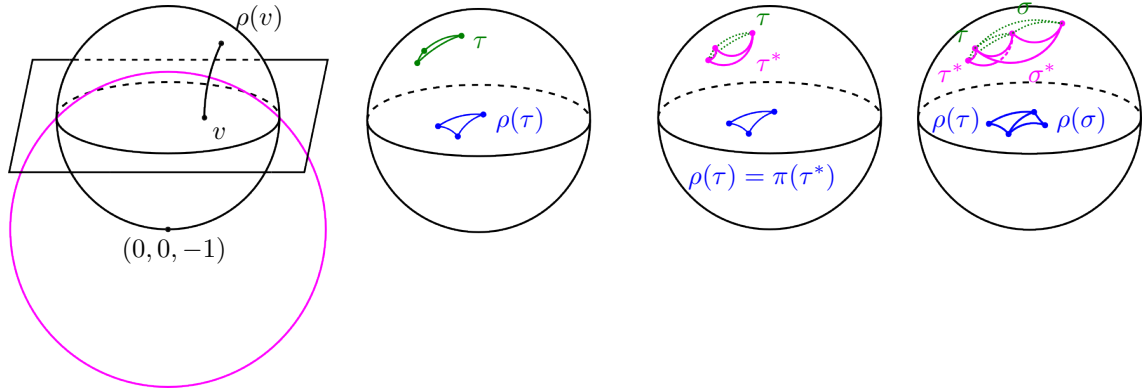
\begin{figure}[ht!]
    \centering
    \begin{overpic}[width=\linewidth]{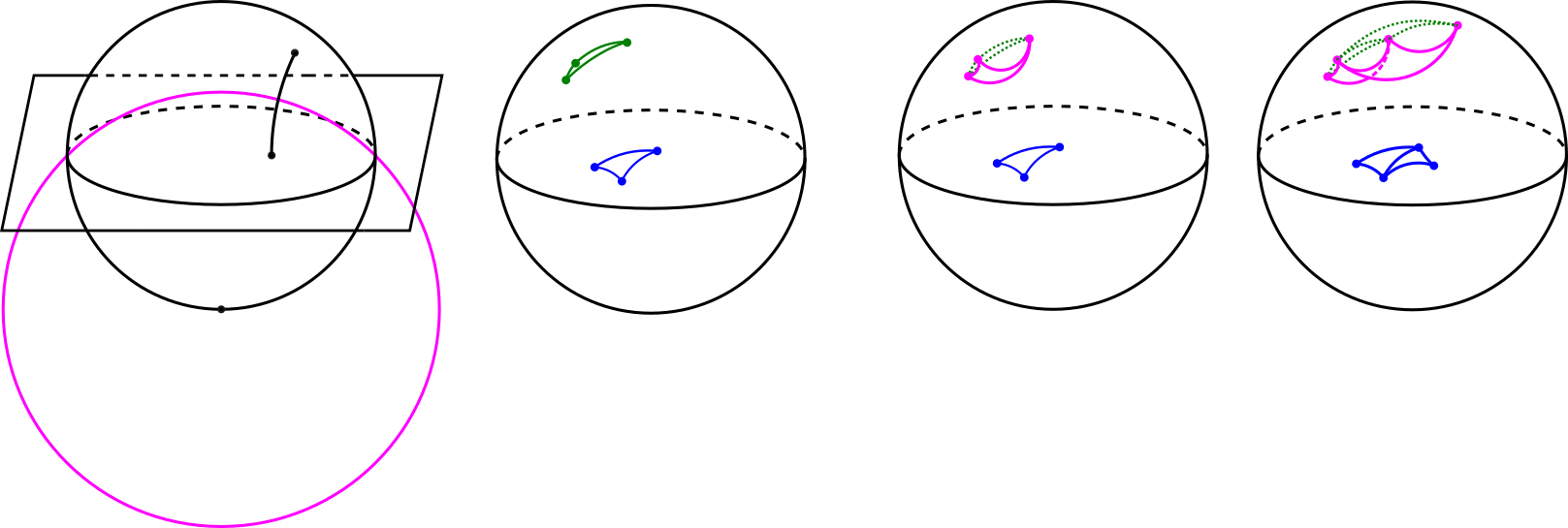}
    \put(10,11.5){$(0,0,-1)$}
    \put(18,22){$v$}
    \put(20,32){$\rho(v)$}
    
    \put(43,23){\textcolor{blue}{$\rho(\tau)$}}
    \put(41,30){\textcolor[HTML]{008000}{$\tau$}}

    \put(66,32){\textcolor[HTML]{008000}{$\tau$}}
    \put(66,28){\textcolor[HTML]{FF00FF}{${\tau}^*$}}
    \put(60,18){\textcolor{blue}{$\rho(\tau)=\pi(\tau^*)$}}

    \put(83.5,30.5){\textcolor[HTML]{008000}{$\tau$}}
    \put(82,27){\textcolor[HTML]{FF00FF}{${\tau}^*$}}
    \put(89,33){\textcolor[HTML]{008000}{$\sigma$}}
    \put(90,27){\textcolor[HTML]{FF00FF}{$\sigma^*$}}
    \put(81.5,23){\textcolor{blue}{$\rho(\tau)$}}
    \put(92.5,23){\textcolor{blue}{$\rho(\sigma)$}}
\end{overpic}
    \caption{Inversion $\rho$, ideal triangle $\tau^*$ for hyperbolic triangle $\tau$, and the gluing of ideal triangles.}
    \label{xyz2}
\end{figure}

By Proposition \ref{covering}, we may assume that $S$ is a non-compact simply connected surface.
Suppose the result is false, i.e., that $d^*$ is not complete. Then there exists a finite length geodesic ray $\gamma^*:[0, L)\to (S-V, d^*)$ with $L\in\R_+$ such that $\gamma^*([0, L))$ is not contained in any compact subset of $S-V$. Let $\mathcal{T}^*$ be the ideal triangulation of $(S-V, d^*)$ corresponding to $\T$ of $(S, V, d)$. Then there is a sequence $\{ e_n^* \}_{n=1}^\infty$ of edges in $\mathcal{T}^*$ that intersect $\gamma^*$ and are ordered by the direction of the ray $\gamma^*$. Hence we have a sequence of points $\gamma^*_i = \gamma^*(t_i)\in e_i^*$ such that $t_i<t_{i+1}$ and $\gamma^*(t_i, t_{i+1})\cap (\cup_{e^*\in E(\T^*)}e^* )= \varnothing$. Denote by $\tau^*_i$ the ideal triangle adjacent to $e^*_i$ and $e^*_{i+1}$. We note that the edges $e_i^*$ may not be distinct, and the same holds for the triangles $\tau_i^*$. Let $\tau_i, e_i$ be the corresponding triangles and edges in $(S,V,d)$.

Let $\tau_i^*\times\{i\}$ and $\tau_i\times\{i\}$ be copies of $\tau_i^*$ and $\tau_i$, respectively.
Denote by $e_i^*\times\{i\}$ and $e_{i+1}^*\times\{i\}$ the corresponding edges of $\tau_i^*\times\{i\}$, and by $e_i\times\{i\}$ and $e_{i+1}\times\{i\}$ the corresponding edges of $\tau_i\times\{i\}$.
We glue $\tau_i^*\times\{i\}$ and $\tau_{i+1}^*\times\{i+1\}$ by identifying $e_{i+1}^*\times\{i\}$ and $e_{i+1}^*\times\{i+1\}$, as $\tau^*_i$ and $\tau^*_{i+1}$ are glued along $e^*_{i+1}$.
Denote the quotient surface by
\begin{equation}
    A^*=\faktor{\bigsqcup_{ i \geq 1} \tau_i^*\times\{i\}}{e_{i+1}^*\times\{i\}\sim e_{i+1}^*\times\{i+1\}}.
\end{equation} 
Similarly, we define
\begin{equation}
    A=\faktor{\bigsqcup \tau_i\times\{i\}}{e_{i+1}\times\{i\}\sim e_{i+1}\times\{i+1\}}.
\end{equation}
By using the Seifert-van Kampen theorem countably many times, we know $A^*$ and $A$ are simply connected. Furthermore, since the gluing maps are isometric, $A$ and $A^*$ are (incomplete) hyperbolic manifolds.

By the monodromy theorem, there is an isometric immersion $\Psi\colon A\to \SS^2_+$, which has been identified with $\H^2_P$ by $\rho$. Then $\Psi$ is a developing map for the surface $A$. 
Since each triangle in $A$ shares the same vertices as the corresponding ideal triangle in $A^*$, and each ideal triangle is determined by its vertices, $\Psi$ then induces an isometric immersion of $A^*$ into $\H^3_P$; we continue to
denote this immersion by $\Psi$.
Moreover, with a slight abuse of notation, we denote $\tau_i^*\times\{i\}$ and $e_i^*\times\{i\}$ in $A^*$ by $\tau_i^*$ and $e_i^*$, respectively, and similarly for $\tau_i$ and $e_i$.
Then $\Psi (e_i^*)$ is a bi-infinite geodesic in $\H^3_P$ with endpoints $x_i^*$ and $y_i^*$. Note that $x_i^*$ and $y_i^*$ are also the endpoints of the geodesic $\Psi (e_i)\in\SS^2_+$. By $\Psi(A)\subset \SS^2_+$, we know that $\Psi(A^*)$ is also above the $xy$-plane in $\R^3$. See Figure \ref{fig:completeness1}.

\vspace{0.4cm}
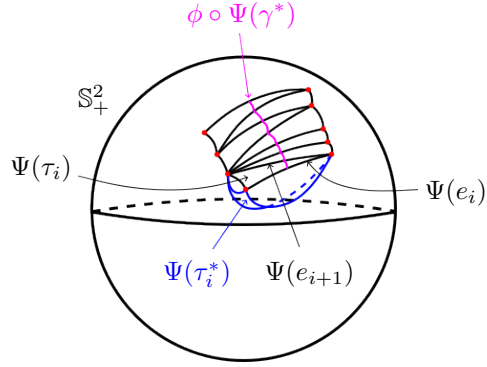
\begin{figure}[htbp!]
    \centering
    \begin{overpic}[width=0.3\linewidth]{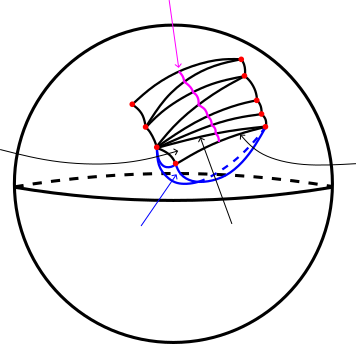}
    \put(0,75){$\SS^2_+$}
    \put(102,48){$\Psi(e_i)$}
    \put(55,24){$\Psi(e_{i+1})$}
    
    \put(25,24){\textcolor{blue}{$\Psi(\tau^*_i)$}}
    \put(-20,55){{$\Psi(\tau_i)$}}

    \put(32,100){\textcolor[HTML]{FF00FF}{$\phi\circ\Psi(\gamma^*)$}}
\end{overpic}
    \caption{A local illustration of the developing map $\Psi$ on $A$.}
    \label{fig:completeness1}
\end{figure}

Since $\Psi (\gamma^*)$ is a finite-length piecewise geodesic intersecting $\Psi (e_i^*)$, which has endpoints $x_i^*, y_i^*$, there exists $\delta >0$ such that $d_\E(x_n^*, y_n^*) = |x_n^* - y_n^*|\geq \delta >0$ for all $n \geq 1$. 
Let $D_i$ be the open circumdisk of $\Psi(\tau_i)$ in $\SS^2_+$, and $D^*_i$ the hyperbolic plane containing $\Psi(\tau^*_i)$ in $\H^3_P$. Then $\partial D_i=\partial D^*_i$.
Since $\phi$ is Lipschitz, we see that the hyperbolic length of $\phi \circ \Psi(\gamma^*)$ is finite. In particular, we have $\lim_{i\to\infty} l_{\H}(\phi \circ \Psi(\gamma^*[t_i, L))) = 0$, where $l_{\H}$ is the hyperbolic length. 

Now we show that there exists a positive constant $\epsilon$ such that
\begin{equation}\label{eq3} d_\H(\phi \circ \Psi(\gamma_i^*), \partial D_i)>\epsilon, \end{equation}
for all $i$, where $d_\H$ is the hyperbolic distance in $\SS^2_+$. Suppose, to the contrary, that there exists a sequence of $\phi \circ \Psi(\gamma_{n_i}^*)$ whose hyperbolic distance to $\partial D_{n_i}$ tends to zero. Since $d_\H(x, y) \geq d_\E(x, y)$ for $x, y \in \mathbb S^2_+$, we have 
\begin{equation}\label{eq31} 
\lim_i d_\E(\phi \circ \Psi(\gamma_{n_i}^*), \partial D_{n_i})=0. \end{equation}

Using the Hausdorff distance with respect to $d_\E$ on $\overline{\H_P^3}$, we may assume, by passing to a subsequence, that $\overline{D_{n_i}^*}$ converges to $D^*_\infty\subset \overline{\H^3_P}$.   Furthermore, we may assume that $x_{n_i}^*$ and $y_{n_i}^*$ converge to $x^*$ and $y^*$, respectively, where $x^*\ne y^*$ in $\overline{S_+^2}$, and that $\Psi (\gamma_{n_i}^*)$ converges to a point $q$ contained in the geodesic connecting $x^*$ and $y^*$ in $\H^3_P$. Then $d_\E(x^*,y^*)\geq \delta$ and $q \in \H^3_P$. This is because $\Psi(\gamma^*)$ has finite length, so the points $\Psi(\gamma^*_{i}) \in \Psi(\gamma^*)$ lie in a compact subset of $\mathbb{H}^3_P$. 
We extend $\phi$ to the conformal boundary of $\H^3_P$, and we have $\phi: \overline{\H^3_P}\to \overline{\SS_+^2}$.
Note that the restriction of $\phi$ to $D_\infty^*$ is a homeomorphism, and we let $D_\infty=\phi(D^*_\infty)$ be a closed disk in $\overline{\SS^2_+}$.
Since $q\in\H^3_P$ is in the interior of $D_\infty^*$, we see that $\phi(q)$ lies in the interior of the disk $D_\infty$. In particular, $d_\E(\phi(q),\partial D_\infty)>0$.
On the other hand, 
\begin{equation} \label{equ-3456}
  d_\E(\phi(q), \partial D_\infty) =  \lim_{i\to\infty} d_\E(\phi \circ \Psi(\gamma_{n_i}^*), \partial D_{n_i}) = 0,
\end{equation} 
by the Hausdorff convergence of $\partial D_{n_i}$ to $\partial D_{\infty}$ with respect to the metric $d_\E$ and $ \phi \circ \Psi(\gamma_{n_i}^*) \to \phi(q)$. This contradicts \eqref{eq31}, and hence \eqref{eq3} holds. 

We claim that \eqref{eq3} contradicts Proposition \ref{finiteness}. 
By the simple connectivity of $S$, there exists an orientation-preserving isometric immersion $\varphi_i: D_i\to S$ such that $(D_i,\varphi_i)$ is a circumdisk of $\Psi(\tau_i)$ and $\varphi_i|_{\Psi(\tau_i)}$ is the inverse of $\Psi|_{\tau_i}$. In particular,   $\varphi_{i+1}|_{\Psi(e_{i+1})}=\varphi_{i}|_{\Psi(e_{i+1})}$. Choose $n\in\Z_+$ such that, for every $m\geq n$ and every $i$, $l_\H(\phi \circ \Psi(\gamma^*[t_m, L)))<\epsilon<d_\H(\phi \circ \Psi(\gamma_i^*), \partial D_i)$. Hence when $m\geq n$, $\phi \circ \Psi(\gamma^*[t_m, L))\subset D_m$. In particular,  $D_n$ contains $\phi \circ \Psi(\gamma_m^*)$ and intersects $\Psi(e_m)$. Let $c_i=\phi \circ \Psi(\gamma^*[t_{i}, t_{i+1}])$ be the geodesic segment connecting $\phi \circ \Psi(\gamma^*(t_{i}))$ and $\phi \circ \Psi(\gamma^*(t_{i+1}))$ in the triangle $\tau_i$; see Figure \ref{fig:completeness2}. Then $c_m\subset  D_n$ for $m \geq n$. In particular, $D_n\cap\cdots\cap D_m \neq \varnothing$.

\begin{figure}[htbp!]
    \centering
    \vspace{0.5cm}
    \begin{overpic}[width=0.35\linewidth]{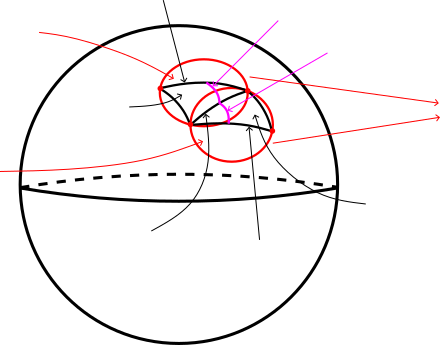}
    \put(11,51){$\Psi(\tau_i)$}
    \put(27,83){$\Psi(e_i)$}
    \put(9,24){$\Psi(e_{i+1})$}
    \put(44,19){$\Psi(e_{i+2})$}
    \put(85,29){$\Psi(\tau_{i+1})$}
    
    \put(62,78){\textcolor[HTML]{FF00FF}{$c_i$}}
    \put(75,70){\textcolor[HTML]{FF00FF}{$c_{i+1}$}}

    \put(103,51){{$S$}}

    \put(83,61){\textcolor{red}{$\varphi_i$}}
    \put(82,44){\textcolor{red}{$\varphi_{i+1}$}}

    \put(-3,70){\textcolor{red}{$D_i$}}
    \put(-18,37){\textcolor{red}{$D_{i+1}$}}
\end{overpic}
    \caption{Two consecutive developed circumdisks along $\phi\circ\Psi(\gamma^*)$.}
    \label{fig:completeness2}
\end{figure}

Now we show $\varphi_m|_{D_n\cap\cdots\cap D_m}=\varphi_n|_{D_n\cap\cdots\cap D_m}$, for all $m\geq n$. We prove it by induction on $m$.  
For $m=n+1$, we know $\varphi_{n+1}|_{\Psi(e_{n+1})}=\varphi_{n}|_{\Psi(e_{n+1})}$ since $\Psi(\tau_n)$ and $\Psi(\tau_{n+1})$ are glued together along their common edge $\Psi(e_{n+1})$.
Since both $\varphi_n$ and $\varphi_{n+1}$ are orientation-preserving, they coincide on a neighborhood of $\Psi(e_{n+1})$, and therefore on $D_n\cap D_{n+1}$.
Now suppose this statement holds for $m$. Then, for the same reason, we have $\varphi_{m+1}|_{D_{m+1}\cap D_m}=\varphi_{m}|_{D_{m+1}\cap D_m}$. Thus
\begin{equation} \label{overlap}
    \varphi_{m+1}|_{D_{m+1}\cap D_m\cap\cdots \cap D_n}=\varphi_{m}|_{D_{m+1}\cap D_m\cap\cdots \cap D_n}=\varphi_{n}|_{D_{m+1}\cap D_m\cap\cdots \cap D_n},
\end{equation}
which completes the induction.

We claim that the vertices of $\Psi(e_m)$ are not contained in $D_n$ for $m\geq n$. Indeed, since $c_{m}\subset D_m\cap\cdots \cap D_n$, for every $m\geq n$ there is a neighborhood $N(c_{m})\subset D_m\cap\cdots \cap D_n$ of $c_m$ such that $\varphi_{m}|_{N(c_{m})}=\varphi_n|_{N(c_{m})}$.
Since $\Psi(e_m)\cap D_n\neq \varnothing$, we know that $\varphi_n^{-1} (\cup_{e\in E} {e}^\circ)$ contains a subinterval of the edge $\Psi(e_m)$. 
Moreover, since $\varphi_n$ and $\varphi_m$ are isometric immersions, we know that $\varphi_n|_{\Psi(e_m)\cap D_n}$ and $\varphi_m|_{\Psi(e_m)\cap D_n}$ are both determined by $\varphi_n|_{N(c_{m})}=\varphi_m|_{N(c_{m})}$. Thus $\varphi_n|_{\Psi(e_m)\cap D_n} = \varphi_m|_{\Psi(e_m)\cap D_n}$.
Therefore, the vertices of $\Psi(e_m)$ are not contained in $D_n$ since $(D_n,\varphi_n)$ is an empty-disk.

Since the developing map $\Psi$ is an isometric immersion, the two triangles $\Psi(\tau_m)$ and $\Psi(\tau_{m+1})$ are on the opposite sides of the edge $\Psi(e_{m+1})$, for any $m$. Thus the edges $\Psi(e_m)$ and $\Psi(e_{m+2})$ are on opposite sides of the edge $\Psi(e_{m+1})$. In particular, when $m\geq n$, the chords $\Psi(e_m)\cap D_n$ and $\Psi(e_{m+2})\cap D_n$ are on the opposite sides of the chord $\Psi(e_{m+1})\cap D_n$. Therefore, we have a sequence of disjoint chords $\Psi(e_m)\cap D_n$ for all $m\geq n$.
By $(\Psi(e_n))^\circ\subset D_n$, we have for all $m>n$,
\begin{equation}
    (\Psi(e_n))^\circ\cap (\Psi(e_m))^\circ=
    (\Psi(e_n)\cap D_n)\cap (\Psi(e_m))^\circ=
    (\Psi(e_n)\cap D_n)\cap ((\Psi(e_m))^\circ\cap D_n)=\varnothing.
\end{equation}
Thus, the $(\Psi(e_m))^\circ$ are pairwise disjoint when $m\geq n$.
It then follows that $\varphi_n^{-1} (\cup_{e\in E} {e}^\circ)$ contains infinitely many components $\Psi(e_m)\cap D_n$, which contradicts Proposition \ref{finiteness}. This completes the case of polyhedral surfaces with hyperbolic background metrics.

Next, we consider polyhedral surfaces with Euclidean background metrics. For a Euclidean triangle $\tau\in\T$, we embed $\tau$ isometrically into $\C$ and construct the ideal hyperbolic triangle $\tau^*$ in the upper half-space model $\H^3_U$. See Figure \ref{777}.
We prove the result using the same argument as in the hyperbolic case.
Let $\gamma^*, \tau^*_n, e^*_n, \Psi, x^*_n, y^*_n, D_n^*$ be as defined for the hyperbolic case, and let $\phi: \H^3_U\to\C$ be the vertical projection that sends each $\Psi(\tau^*_n)$ to $\Psi(\tau_n)$. We still denote by $D_n\subset \C$ the circumdisk of $\Psi(\tau_n)$, and, by assumption, the diameters of $D_n$ are uniformly bounded by some constant $C_1$. Therefore, the restriction of $\phi$ to each $D_n^*$ is $L_1$-Lipschitz for some constant $L_1$ depending only on $C_1$. 

We claim that each $\partial D_n$ intersects a fixed compact set $K$. Indeed, since $\gamma^*$ has finite hyperbolic length, the Euclidean distance from the limit point of $\Psi(\gamma^*(t))\in\H^3_U$ as $t\to L$ to $\C$ has a lower bound. Since $\gamma^*$ intersects all $e^*_n$, it follows that there exists a positive number $\delta >0$ such that $|x_n^* -y^*_n| \geq \delta$. Furthermore, there exists a compact set $K \subset \C$ 
such that $\{x^*_n,y^*_n\} \cap K \neq \varnothing$ for all $n$. Thus $\partial D_n \cap K \neq \varnothing$.

As a consequence of this claim and the boundedness of the diameters of $D_n$, we see that there exists a compact set $K' \subset \C$ that contains all $D_n$. By the uniform Lipschitz continuity of $\phi$ on all $D_n^*$, our argument for the hyperbolic case works for the Euclidean case as well.

Finally, for polyhedral surfaces with spherical background metrics, we again apply the argument used in the hyperbolic case. Now each triangle $\tau\in\T$ is a spherical triangle in $\SS^2$, and its associated $\tau^*$ is the ideal triangle in $\H^3_P$ that has the same vertices as $\tau$.
Let $\gamma^*, \tau^*_n, e^*_n, \Psi, x^*_n, y^*_n$, and $D_n^*$ be as defined in the hyperbolic case, and $\phi: \H^3_P-\{O\}\to\SS^2$ the radial projection which sends each $\Psi(\tau^*_n)$ to $\Psi(\tau_n)$, where $O$ is the origin of $\R^3$. Here, each $\Psi(\tau_n)$ is contained in some open hemisphere, so $\phi$ is well-defined on each $\Psi(\tau^*_n)$. We still denote by $D_n$ the circumdisk of $\Psi(\tau_n)$, and, by assumption, the diameters of $D_n$ are uniformly bounded by some constant $C_2$ less than $\pi$. In this case, we also have that the restriction of $\phi$ to each $D_n^*$ is $L_2$-Lipschitz for some constant $L_2$ depending only on $C_2$. Thus our argument for the hyperbolic case works as well.
\end{proof}

The following proposition shows that the standard model surfaces of constant curvature $-1$ that are discrete conformal to polyhedral surfaces satisfy the condition in Theorem \ref{llz} concerning limit points of the vertex set.

\begin{proposition}\label{polylimit}
 Suppose a simply connected polyhedral surface $(S, V, d)$ is discrete conformal
 to $(\D,W,d_\H)$ for a closed discrete subset $W \subset \D$.
 Then every point in $\partial \D$ is a limit point of points of $W$, i.e., $\partial \D \subset \overline{W}$.
\end{proposition}
\begin{proof}

    Suppose, to the contrary, that there exists an open interval $I$ on $\partial \D$ whose closure is disjoint from the closure of $W$. We claim that there exists a half-plane embedded in $\partial C(W \cup \D^c)$. Assuming the claim, since $(S-V, d^*)$ is isometric to $\partial C(W \cup \D^c)$, there exists an embedding $\phi: P \to  (S-V, d^*)$, where $P \subset \H^2_P$ is an open half-plane. We claim this is impossible. Indeed, suppose that such an embedding exists, and let $\T$ be an ideal triangulation of $(S-V, d^*)$ and let $e$ be an edge in $\T$ such that $e$ intersects the image $\phi(P)$. Then $\gamma =\phi^{-1}(e)$ is a closed subset of $P$ which is contained in a geodesic. It follows that $\gamma$ is a geodesic ray ending at a point $X$ in the conformal infinity $\partial P \cap \partial \H^2_P$. Then $P$ contains a horocycle $Q$ centered at $X$, where $Q$ has infinite length.
    Let the endpoint of the geodesic ray $\phi(\gamma)$ be the vertex $v \in V$; then $\phi(Q)$ is a horocycle at $v$ in $(S-V, d^*)$. However, since $v$ is a cusp, we know that $\phi(Q)$ has finite length, which contradicts that $\phi$ is an embedding.
    

Now we prove the claim that $\partial C(W \cup \D^c)$ contains a half-plane.    
    By the assumption on $I$, for each point $p \in I$, there exists an open horoball in $\D$ that is tangent to $p$ and disjoint from $W$.  If $B$ and $B'$ are two horoballs tangent to $p$, then $B \subset B'$ or $B' \subset B$.  Hence,  the union $B_p$ of all these open horoballs that are tangent to $p$ and disjoint from $W$ is an open horoball such that $\partial B_p \cap W \neq \varnothing$.  Furthermore, $\partial B_p \cap W$ is a finite set since $W$ is a closed, discrete subset and $p \notin \overline{W}$.  Consider the function $f: I \to \mathbb N$ sending $p$ to the number of points in $\partial B_p \cap W$. The set $\{ p \in I : f(p) \geq 2\}$ is countable since $W$ is countable and, for every finite subset $Y \subset W$ containing at least two points, there are at most two horocycles in $\D$ that contain $Y$. In particular, since $I$ is uncountable, there exists $p \in I$ such that $f(p)=1$. See the left picture of Figure \ref{halfplane}.

    \begin{figure}[htbp]
  \centering
  \begin{overpic}[width=0.9\linewidth]{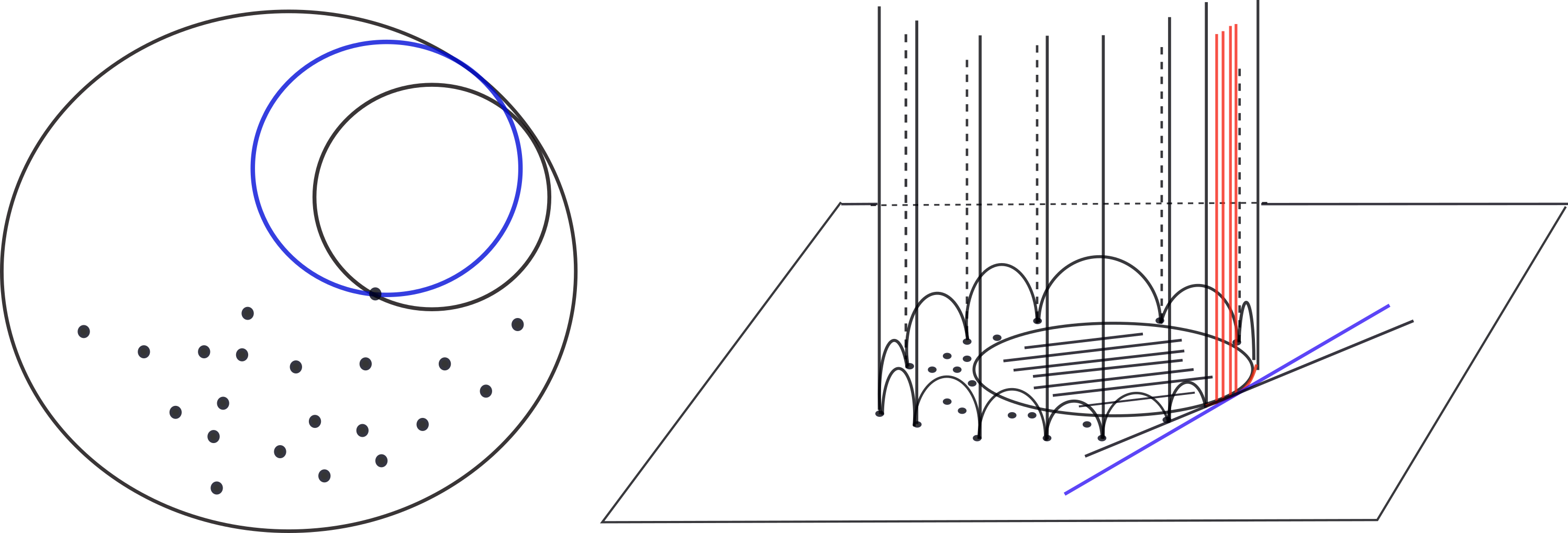}
    \put(35,26){$p$}
    \put(30,31){$q$}
    \put(18,26){$B_q$}
    \put(25,25){$B_p$}
    \put(22,13){$w$}
    \put(66,34){$w = \infty$}
    \put(81,25){Red lines in the}
    \put(81, 22.5){ruled surface $J\times \mathbb{R}_{>0}$}
    \end{overpic}
\caption{Left: Circles of supporting planes at $p$ and $q$. Right: Construction of the ruled surface. 
}
\label{halfplane}
\end{figure}
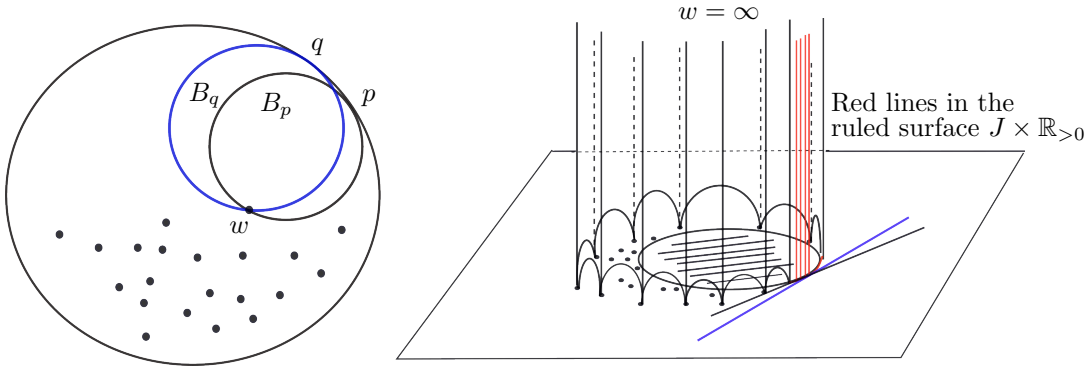
    
    We claim that $f^{-1}(1)$ is an open set in $I$.  Suppose otherwise. Then there exists a sequence of points $q_n \in I$ converging to $p$ such that $f(q_n) \geq 2$. Since $q_n \in I$ and $\overline{I} \cap \overline{W} =\varnothing$, the diameters of $B_{q_n}$ are bounded away from $0$. Let $u_n, v_n$ be two distinct points in $\partial B_{q_n} \cap W$. We may assume, by passing to subsequences, that $\lim_n u_n =x$, $\lim_n v_n =y$ in $\overline{\D}$, and that the sequence $\overline{B_{q_n}}$ converges to a closed horoball $\overline{B}$.
    Here, $B$ is an open horoball centered at $ p$ and disjoint from $ W$. Clearly $x$ and $y$ are not in $\partial \D$ since $\overline{I} \cap \overline{W}=\varnothing.$  Also, $x, y \in \partial B \cap W$. This shows $B=B_p$. But $f(p)=1$. Therefore $\{x\}=\{y\}= \partial B_p \cap W$. On the other hand, $W$ is discrete. This shows $u_n=v_n=x=y$ for $n$ large, which contradicts the choice of $u_n \neq v_n$.
    
    The above proof shows that there exists a small open interval neighborhood $I'\subset I$ containing $p$ such that for all $q \in I'$, $\partial B_q \cap W =\partial B_p \cap W$, which is a single point, denoted by $w$.  Consider the Möbius transformation of $\C$ sending $w$ to $\infty$,  $\mathbb S^1$ to $\mathbb S^1$, $W$ to $U$, and  $I'$ to an interval $J \subset \mathbb S^1$. Then $\partial C(W \cup \D^c)$ is isometric to $\partial C(U \cup \overline{\D})$. By construction, $\partial C(U \cup \overline{\D})$ contains $J \times \R_{>0}$ as a subsurface. See the right picture of Figure \ref{halfplane}. Let $l_{\E}(J)$ be the Euclidean length of $J$. It is well known\footnote{If $\alpha(t)$ is an arc-length parameterization of a planar simple arc $J$ for $t \in (a,b)$, then the map $F(t,s)=(\alpha(t),s): (a, b) \times (0, \infty) \to J \times (0, \infty)$ is an isometry from the metric $\frac{dt^2+ds^2}{s^2}$ to the induced hyperbolic metric on $J \times (0, \infty) \subset \H^3_U$.}     
    that the subsurface $J \times \mathbb{R}_{>0}$, equipped with the hyperbolic metric restricted from $\mathbb{H}_U^3$, is isometric to the subsurface $\{(t, 0, s) : t \in (0, l_{\E}(J)), s \in \R_{>0}\}\subset \H^3_U$ which contains a hyperbolic half-plane.  Thus, the claim follows. 
\end{proof}

\section{Proof of the discrete uniformization theorem}

Recall that Theorem \ref{dut} states:
\newtheorem*{maintheorem1}{Theorem \ref{dut}}
\begin{maintheorem1}
 (i) Every connected polyhedral surface with a hyperbolic background metric is discrete conformal to a standard model surface.

(ii) Every connected polyhedral surface with a Euclidean background metric whose empty-disks have uniformly bounded diameters is discrete conformal to a standard model surface. Moreover, the uniform-bound hypothesis cannot be omitted in general. 

(iii) Every connected polyhedral surface with a spherical background metric whose empty-disks have diameters uniformly bounded away from $\pi$ is discrete conformal to a standard model surface. Moreover, the uniform-bounded away from $\pi$ hypothesis cannot be omitted in general.  

 (iv) Two standard model surfaces are discrete conformal if and only if they are related by a Möbius transformation.
\end{maintheorem1}

\begin{proof}[Proof of Theorem \ref{dut}]
    Let $(S, V, d)$ be a polyhedral surface whose background metric $d$ satisfies one of the following conditions:
    \begin{enumerate}
        \item it is a hyperbolic cone metric;
        \item it is a Euclidean cone metric, and the diameters of empty-disks in $(S,V,d)$ are uniformly bounded above; or
        \item it is a spherical cone metric, and the diameters of empty-disks in $(S, V, d)$ are uniformly bounded away from $\pi$.
    \end{enumerate}
    We first consider the case where $S$ is simply connected. Thus $S$ has one topological end $\delta$ if $S$ is homeomorphic to $\R^2$, or has no topological ends if $S$ is homeomorphic to $\mathbb S^2$.
    By Theorem \ref{thm:completeness}, its associated hyperbolic metric $d^*$ is complete.
    Then by Theorem \ref{lw}, there exists a circle-type closed set $Y$ in $\hat\C =\partial \H^3_U$ and an isometry $f: \partial C(Y)  \to (S-V, d^*)$.
    Notice that the components of $Y$ are in one-to-one correspondence with the points in $V \cup \{\delta \}$ (if $\delta$ exists).
    Since $(S-V, d^*)$ has cusp ends at $V$,  $Y$ has at most one disk component with all other components isolated points.  It follows that $f$ extends to a homeomorphism on the cusps-filled surfaces. Here are the possible cases:
    \begin{enumerate}[label=(\arabic*)]
        \item If $S$ is homeomorphic to $\SS^2$, which is compact, then $V$ is a finite set since it is a closed discrete subset. Then $Y$ is also a finite subset of $\hat{\C}$ that has the same cardinality as $V$. By Theorem \ref{rivin}, $(S, V, d)$ is discrete conformal to $(\SS^2, Y, d_\SS)$ and $Y$ is unique up to M\"obius transformations of $\SS^2$.

        \item If $S$ is homeomorphic to $\R^2$ and $\delta$ corresponds to a point $y$ in $Y$, then $V$ is an infinite set since $(S,V,d)$ admits a  Delaunay triangulation. 
        Without loss of generality, we  assume $y=\infty \in \partial \H^3_U$, and then all other points in $Y$ are in $\E^2 = \partial\H^3_U-\{\infty\}$. Then $(S-V,d^*)$ is isometric to $\partial C(Y)$.
        Since the points in $V$ are all cusp points in $(S-V,d^*)$,   points in $Y-\{\infty\}$ are cusp points in $\partial C(Y)$, and hence isolated. 
        We note that $\partial C(Y) \cup (Y-\{\infty\}) = (\partial C(Y) \cup Y)-\{\infty\}$ is homeomorphic to $\E^2$; hence, $Y-\{\infty\}=Y\cap \E^2$ is a closed discrete subset of $\E^2$.
        By Theorem \ref{llr}, such a surface $(\E^2, Y-\{\infty\}, d_\E)$ is unique up to a similarity of $\E^2$.

        \item If $S$ is homeomorphic to $\R^2$ and $\delta$ corresponds to a disk component in $Y$, up to a M\"obius transformation, we may assume that the disk component is $\D^c$. Let $V_1:=Y - \D^c\subset\D$. Then $(S-V,d^*)$ is isometric to $\partial C(Y)=\partial C(V_1 \cup \D^c)$.
        Since the points in $V$ are all cusp points in $(S-V,d^*)$,  the points in $V_1$ are cusp points in $\partial C(Y)$, and hence isolated. 
        Thus $V_1$ is a closed discrete subset of $\D$, and $(S, V, d)$ is discrete conformal to $(\D, V_1, d_\H)$. 
        By Proposition \ref{polylimit}, $\overline{V_1} - V_1 = \partial \D$. Then by Theorem \ref{llz}, such a set $V_1$ is unique up to M\"obius transformations of $\D$.
    \end{enumerate}

    If $(S,V,d)$ is not simply connected, consider its universal cover $(\tilde S, \tilde V, \tilde d)$ with the associated complete hyperbolic metric $\tilde d^*$ on $\tilde S-\tilde V$. 
    The fundamental group $\Gamma =\pi_1(S)$ acts isometrically as a deck transformation group on   $(\tilde S, \tilde V, \tilde d)$ with quotient isometric to 
    the surface $(S, V, d)$.  By Corollary \ref{triagulationrefine} (2), we see that the same $\Gamma$ action on the subsurface $\tilde S -\tilde V$ is a free, discontinuous, and isometric action with respect to $\tilde d^*$.
    By the result for the simply connected case, we see that $(\tilde S,  \tilde V, \tilde d)$ is discrete conformal to a standard model surface $(\K^2, U, d_{\K})$, where $\K^2=\SS^2,\E^2$ or $\H^2_P=\D$ and $U$ is a closed discrete subset of $\K^2$ so that
    \begin{enumerate}[label=(\alph*)]
        \item the barycenter of $U$ is the origin if $\K^2=\SS^2$, after applying a M\"obius transformation,
        \item $U$ is infinite if $\K^2=\E^2$, and
        \item every point in the conformal infinity of $\H^2_P$ is a limit point of $U$ if $\K^2=\H^2_P$.
    \end{enumerate}
    This implies that $(\tilde S -\tilde V, \tilde d^*)$ is isometric to $\partial C(Y)$ in $\H^3_U$; here $Y=U$, $U\cup\{\infty\}$ or $U\cup\D^c$ when $\K^2=\SS^2$, $\E^2$ or $\H^2_P=\D$, respectively.   
    In particular,  $\Gamma$ acts isometrically, freely and discontinuously on $\partial C(Y)$.
    By Theorems \ref{rivin}, \ref{llr} and \ref{llz}, 
    the isometric action of $\Gamma$ on $\partial C(Y)$ extends to an isometric action of $\Gamma$ on $\H^3_U$. 
    
    It then follows that $\Gamma$ acts isometrically, freely and discontinuously on $\SS^2$, $\E^2$ or $\H^2_P =\D$ preserving $U$. By construction, $(S, V, d)$ is discrete conformal to the standard model surface $(\SS^2,U,d_\SS)/\Gamma$, $(\E^2, U, d_\E)/\Gamma$, or $(\H^2_P, U, d_\H)/\Gamma$.

We recall that the associated hyperbolic metrics of standard model surfaces are complete. Hence, the examples presented in Appendix \ref{appen:incomplete} show that the assumptions on the circumdisk diameters for Euclidean and spherical background metrics are sharp.

Finally, the uniqueness of the standard model surfaces within a discrete conformal class follows from the stronger result in Theorem \ref{thm:new5.2} together with Remark~\ref{rem:spherical-standard}.
\end{proof}

\begin{remark}  
By the above argument, Theorem \ref{dut} (ii) and (iii) hold for all polyhedral surfaces without the boundedness assumption as long as the associated hyperbolic metrics $d^*$  are complete.
\end{remark}

We recall the following theorems from \cite{luo-luo-rao} and \cite{Zhao2026rigidity}, which will be used to prove Theorem \ref{thm:new5.2}.

\begin{theorem} \label{5.2}

(i) \cite{luo-luo-rao}
Let $X_1$ and $X_2$ be two compact sets containing at least three points in
$\partial \H^3_P$ with vanishing
$1$-dimensional Hausdorff measure. 
Then every isometry between the boundaries of the convex hulls $\partial C(X_1)$ and $\partial C(X_2)$ in $\H^3_P$ extends to an ambient isometry of $\H^3_P$.   The extended isometry of the ambient space $\H^3_P$ is unique unless $X_1$ and $X_2$ are contained in circles, in which case the extended isometry is unique up to composition with inversions about the circles.

(ii) \cite{Zhao2026rigidity}
Let $X_1$ and $X_2$ be two circle-type closed sets with only one disk component in 
$\partial \H^3_P$ such that, for $i=1,2$, the 
$1$-dimensional Hausdorff measure of $X_i$ after removal of its
disk component is zero. 
Then every isometry between the boundaries of the convex hulls $\partial C(X_1)$ and $\partial C(X_2)$ in $\H^3_P$ extends to an ambient isometry of $\H^3_P$. 
The extended isometry of the ambient space $\H^3_P$ is unique unless both $X_1$ and $X_2$ are closed disks, in which case the extended isometry is unique up to composition with inversions about the boundary circles of $X_1$ and $X_2$.
\end{theorem}

\begin{theorem}\label{thm:new5.2}
Suppose that, for $i=1,2$, $(S_i, X_i, d_i)$ is a complete Riemannian surface of constant curvature $0$ or $-1$, equipped with nonempty closed subsets $X_i$ of vanishing $1$-dimensional Hausdorff measure, and that $(S_i, X_i, d_i)$ is not exceptional. 
If their associated convex-hull hyperbolic metrics $(d_1)^*_c$ and $(d_2)_c^*$ are related by an isometry $h$ that extends to a homeomorphism $\bar{h}$ from $(S_1, X_1)$ to $(S_2, X_2)$, then 

(i) $\bar{h}$ is a similarity transformation from $(S_1, X_1, d_1)$ to $(S_2, X_2, d_2)$  if $d_1$ has vanishing curvature; and

(ii) $\bar{h}$ is an isometry from $(S_1, X_1, d_1)$ to $(S_2, X_2, d_2)$ if $d_1$ has  curvature $-1$. 

Furthermore, if two standard model surfaces $(S_1, X_1, d_1)$ and $(S_2, X_2, d_2)$ are discrete conformal, 
then $(S_1, X_1, d_1)$ and $(S_2, X_2, d_2)$ are related by a M\"obius transformation.  
\end{theorem}

\begin{remark}\label{rem:spherical-standard}
    The last statement in Theorem \ref{thm:new5.2} also holds  if both $(S_i,d_i)$ have curvature $1$. Indeed, if the $S_i$ are $\SS^2$, it follows from Rivin's Theorem \ref{rivin}. If the $S_i$ are $\R P^2$, it follows by taking the universal covers and then applying the same argument.
\end{remark}

\begin{proof}
    Since $(S_1,X_1)$ is homeomorphic to $(S_2,X_2)$ by $\bar{h}$, we may assume that $S_1=S_2=S$.
    We first show that $(S,X_1,d_1)$ and $(S,X_2,d_2)$ must be of the same type, that is, the metrics $d_1$ and $d_2$ have the same curvature. Since $(S,X_i, d_i)$, $i\in\{1,2\}$, is spherical if and only if $S$ is a sphere or a projective plane, it suffices to distinguish the hyperbolic and Euclidean background metrics. By assumption, the surfaces $(S-X_1, (d_1)_c^*)$ and $(S-X_2, (d_2)_c^*)$ are isometric via an isometry $h$ which extends to a homeomorphism $\bar{h}$ of $(S,X_1)$ onto $(S,X_2)$. 
    For each $i\in\{1,2\}$, let $(\tilde{S},\tilde{X}_i,\tilde{d}_i)$ be the universal cover of $(S,X_i,d_i)$ and let $(\tilde{d}_i)_c^*$ be the convex-hull hyperbolic metric defined on  $\tilde{S}-\tilde{X}_i.$  By the definition of convex hull metrics, 
    the surfaces $(\tilde{S}-\tilde{X}_1, (\tilde{d}_1)_c^*)$ and $(\tilde{S}-\tilde{X}_2, (\tilde{d}_2)_c^*)$ are also isometric. By definition, each $(\tilde{S}-\tilde{X}_i, (\tilde{d}_i)_c^*)$ is isometric to a convex surface $\partial C(\tilde{X_i} \cup W_i)$, where $W_i$ is either a point or a round disk in $\partial \mathbb{H}^3_U$. Suppose otherwise, say,  that $d_1$ is Euclidean, and $d_2$ is hyperbolic. Then  $W_1$ is a point and $W_2$ is a round disk. 
    However, the strong rigidity result in \cite[Theorem 1.2]{Zhao2026rigidity} implies that the two sets $W_1$ and $W_2$ must be related by a M\"obius transformation. This is a contradiction since $W_1$ is a point and $W_2$ is a round disk. This contradiction shows that the two metrics $d_1$ and $d_2$ have the same curvature.

To prove the theorem, we first consider the case where $S$ is simply connected, i.e., $(S, d_1) = (S, d_2) = \E^2$ or $\H^2_P$. When $(S, d_i) = \H^2_P$ is considered as a subset of $\partial \H^3_U$, the definition of convex hull metrics and the assumption that $h:(S-X_1, (d_1)_c^*) \to (S-X_2, (d_2)_c^*)$ is an isometry imply that $Q_{S, X_2, d_2} \circ h \circ Q_{S, X_1, d_1}^{-1}$ is an isometry from $\partial C(Y_1)$ to $\partial C(Y_2)$, where $Y_i = X_i \cup \D^c$. Since $X_1$ and $X_2$ are nonempty, the surfaces $\partial C(Y_1)$ and $\partial C(Y_2)$ are not contained in any geodesic plane. By Theorem \ref{5.2}(ii), we conclude that the isometry $Q_{S, X_2, d_2} \circ h \circ Q_{S, X_1, d_1}^{-1}$ extends to an isometry $g$ of $\H^3_U$.
Therefore, 
\begin{equation}\label{hQ}
h = Q_{S, X_2, d_2}^{-1} \circ g \circ Q_{S, X_1, d_1}. 
\end{equation} 
We claim that this implies 
$$ h=g|_{\D-X_1}. $$ 
Indeed, by \eqref{hQ}, we see that $g(\D^c)=\D^c$, and hence $g(X_1)=X_2$. The equivariant property \eqref{equivq} of $Q_{S, X_i, d_i}$ then implies that $g \circ Q_{S, X_1, d_1}= Q_{S, X_2, d_2} \circ g$. Therefore, $g|_{\D-X_1}=h$. Consequently, $h$ extends to a hyperbolic isometry $g|_{\D}: \D \to \D$ in the Poincar\'e model, which yields conclusion (ii). Note that in this proof, we do not use the assumption that $h$ extends to a homeomorphism from $S$ to $S$.

For the case where $S =\E^2 \subset \partial \H^3_U$, we set $Y_i=X_i \cup \{\infty\}$. If both $C(Y_1)$ and $C(Y_2)$ are three-dimensional, we follow the same strategy as above and use Theorem \ref{5.2}(i) to obtain an isometry $g$ of $\H^3_U$ such that \eqref{hQ} holds. Since $h$ extends to a homeomorphism from $(S, X_1)$ to $(S, X_2)$, the equality \eqref{hQ} implies $g(\infty) =\infty$. By \eqref{equivq}, we again deduce that $h=g|_{\E^2-X_1}$, and that $\bar h=g|_{\E^2}$ is a similarity map.  
Now suppose that at least one of $C(Y_1)$ and $C(Y_2)$ is two-dimensional. The arguments in \cite[Sections 5.2 and 5.3]{luo-luo-rao} imply that both of them must be two-dimensional, and that $h$ maps the relative boundary of $C(Y_1)$ to that of $C(Y_2)$. It follows that $h$ induces an isometry $\hat{h}: C(Y_1) \to C(Y_2)$. To adapt the strategy used in the three-dimensional case, we let $g$ be the unique isometry of $\mathbb{H}^3_U$ such that $g|_{C(Y_1)}=\hat{h}$ and $g$ preserves the correspondence between the respective copies, i.e., for each $\sigma \in \{+, -\}$, the map $Q_{{S}, {X}_2, {d}_2} \circ g \circ Q_{{S}, {X}_1, {d}_1}^{-1}$ sends each copy $C(Y_1)_\sigma$ of $C(Y_1)$ in $\partial C(Y_1)$ isometrically to the corresponding copy $h(C(Y_1)_{\sigma})$ of $C(Y_2)$ in $\partial C(Y_2)$. By a similar argument, we conclude that $g(\infty)=\infty$ and that $\bar h=g|_{\E^2}$ is a similarity map.

Now consider the case where the surface $S$ is not simply connected. Let $F_i: (\tilde S, \tilde d_i) \to (S, d_i)$ be the Riemannian universal covering map, $\tilde X_i =F_i^{-1}(X_i)$, and $f: \tilde S \to \tilde S$ be the homeomorphism which is a lift of the homeomorphism $\bar h$.
Since $\bar h(X_1)=X_2$, we have $f(\tilde X_1) =\tilde X_2$ and $F_2 \circ f = \bar h \circ F_1.$  By Proposition \ref{chhm} (ii), i.e., \eqref{eq-chh},
$F_i|: (\tilde S-\tilde X_i, (\tilde d_i)^*_c) \to (S-X_i, (d_i)^*_c)$ are local isometries for $i=1,2$. Using the assumption that $h:(S-X_1, (d_1)^*_c) \to (S-X_2, (d_2)_c^*)$ is an isometry and $ F_2 \circ f  = h \circ F_1$ on $\tilde S -\tilde X_1$, we see that  $f$ is a local isometry and hence an isometry from $(\tilde S -\tilde X_1, (\tilde d_1)^*_c)$ to $(\tilde S -\tilde X_2, (\tilde d_2)^*_c)$. By the simply connected case just proved, we conclude that $f: (\tilde S, \tilde d_1) \to (\tilde S, \tilde d_2)$ is an isometry if $d_1$ is hyperbolic, or is a similarity map if $d_1$ is flat.  This implies that $\bar h: (S, d_1) \to (S, d_2)$ is, accordingly, either an isometry or a similarity map. Thus cases (i) and (ii) hold.

Finally, we consider the case where $(S_1, X_1, d_1)$ and $(S_2, X_2, d_2)$ are discrete conformal.
In this case, $X_i$ is a closed discrete subset of $S_i$. By the definition of discrete conformality, there exists an isometry $h:(S_1-X_1, (d_1)^*_c) \to (S_2-X_2, (d_2)^*_c)$ such that $h$ can be extended to a homeomorphism $\bar h: (S_1, X_1) \to (S_2, X_2)$. 
Therefore, by the two cases (i) and (ii) just proved, we see that $(S_1, X_1, d_1)$ and $(S_2, X_2, d_2)$ are related by a M\"obius transformation.
\end{proof}

\section{Proof of the Discrete Riemann Mapping Theorem}

Recall the Discrete Riemann Mapping Theorem \ref{disriemap} states:

\newtheorem*{maintheorem2}{Theorem \ref{disriemap}}
\begin{maintheorem2}
 Suppose $\Omega$ is a simply connected domain in the Riemann sphere $\hat \C$ such that its boundary $\partial \Omega$ contains at least three points and $V \subset \Omega$ is a discrete approximation of $\Omega$. Then there exists a discrete approximation $W$ of $\D$  such that the Thurston dome
$\partial C(V \cup \Omega^c)$ is isometric to $\partial C(W \cup \D^c)$. Furthermore, the discrete set $W$ is unique up to Möbius transformations of $\D$.
\end{maintheorem2}

By Thurston's theorem and the discussion in \S4, $\partial C(V \cup \Omega^c)$ is a hyperbolic surface with a cusp end at each $v \in V$. It follows that $\partial C(V \cup \Omega^c)\cup V$ is the Riemann surface obtained by filling in cusp ends. We begin with the following two propositions which will be used.

\begin{proposition}
\label{typehy}
If $V$ is a closed discrete subset of a simply connected domain $\Omega \subsetneq\C$, then the Riemann surface $\partial C(V\cup \Omega^c)\cup V$ is conformal to $\D$. 
\end{proposition}

\begin{proof}
Let $S=\partial C(V \cup \Omega^c)\cup V$ and $\hat S = S\cup \Omega^c$. We claim that $\hat{S}$ is a topological 2‑sphere.
Indeed, if $C(V\cup \Omega^c)$ is three-dimensional, we can choose an interior point $O$ of  $C(V\cup \Omega^c)$ and use the radial projection from $O$ to construct a homeomorphism $h$ from $\hat S$ to $\hat \C$ with $h|_{\Omega^c \cup V} =id$.  If $C(V\cup \Omega^c)$ is two‑dimensional, then $V \cup \Omega^c$ is contained in a circle $L$ in the Riemann sphere. By definition, $\partial C(V \cup \Omega^c)$ is the Alexandrov double $D(K) =K^+ \cup_{\partial K} K^-$ with $V\cup \Omega^c$ removed, where $K =C(V \cup \Omega^c)\cup V \cup \Omega^c$ is the compactified disk, and $K^{\pm}$ are two copies of $K$. The double $D(K)$ is a topological sphere with $V \cup \Omega^c \subset \partial K$ as the shared boundary points of the two copies.  There is a boundary homeomorphism $\partial K \to L$ fixing $V \cup \Omega^c$: use radial projection from a relative interior point of $K$ to $L$ in $C(L)$. Extend this boundary homeomorphism across each copy of $K$ to the two disks bounded by $L$ in the Riemann sphere. Gluing them yields a homeomorphism $h: D(K) \to \C \cup \{\infty\}$ with $h|_{V \cup \Omega^c} =id.$  Removing $\Omega^c$ proves that $\Omega \cong \partial C(V \cup \Omega^c) \cup V.$ In conclusion, 
we showed that 
$S$ is homeomorphic to $\Omega$, and hence is simply connected.

Choose a point $p \in V \subset S$ and a simple closed curve $\alpha \subset S-V$ that separates $p$ from the end of $S$. Let $R$ be the component of $S-\alpha$ which does not contain $p$. Since $S$ is topologically an open disk, $R$ is topologically an annulus.
We will prove that the modulus of $R$ is finite, i.e., $\operatorname{Mod}(R)<\infty$.
Since a punctured disk has infinite modulus, this shows that the end of $S$ is not a puncture.  Then the uniformization theorem implies that $S$ is conformally equivalent to $\D$.

We estimate the modulus of $R$ using the extremal lengths. Recall that for a family $\Gamma$ of rectifiable curves in a Riemann surface $X$ and a conformal metric $\rho$ on $X$ with finite area $A(\rho)$, we define 
    \begin{equation*}
        l_{\rho}(\Gamma) = \inf\left\{l_{\rho}(\gamma) : \gamma\in \Gamma\right\}.  
    \end{equation*}
    The \underline{extremal length} $EL(\Gamma)$ of $\Gamma$ is defined to be 
    \begin{equation*}
        EL(\Gamma) = \sup\left\{\frac{l_{\rho}^2(\Gamma)}{A(\rho)} : \rho \text{ is a finite area conformal metric on } X\right\}.
    \end{equation*}

For the annulus $R$, let $\Gamma$ be the family of all rectifiable simple closed curves in $R$ which separate the two ends of $R$. Then 
$\operatorname{Mod}(R)=(\operatorname{EL}(\Gamma))^{-1}.
$
Thus, it is enough to prove $\operatorname{EL}(\Gamma)>0$.

To achieve this, we take the Poincar\'e model $\H^3_P$ of the hyperbolic space, and consider the domain $\Omega$  as a subset of the conformal infinity $\mathbb S^2$ of $\H^3_P$. 
Let $d_{\E}$ be the ambient Euclidean metric on the Poincaré ball model $\H^3_P$ induced from $\R^3$, and let $d$ be the path metric on $\partial C(V \cup \Omega^c)$ induced by $d_{\E}$. Because the Poincaré hyperbolic metric is conformal to the Euclidean metric and the filled points in $V$ have zero area, $d$ is a conformal metric in the sense of Ahlfors-Bers with admissible Borel measurable conformal density on the Riemann surface $S$. 
By \cite[Theorem 4.1]{luow2024}, which states that the Euclidean area of the boundary of any convex set of dimension at least two in the Poincaré model $\H^3_P$ is at most $16\pi$,  we have $A(d)\leq 16\pi$.
Thus, it suffices to show that $l_d(\Gamma)>0$.

Let $$\Phi(x) = \frac{2x}{1+|x|^2} : \H^3_P \to \H^3_K$$
    be the isometry from the Poincar\'e model $\H^3_P$ to the Klein model $\H^3_K$. Then by \cite[Proposition 4.2]{luow2024}, $\Phi$ is $2$-Lipschitz with respect to the Euclidean metric, namely 
    $$|\Phi(x) - \Phi(y)| \leq 2|x - y|.$$
    Let $d'$ be the path metric on $\Phi(S)$ induced by the Euclidean metric $d_{\mathbb{E}}$. Then the $2$-Lipschitz property implies that for any path $\gamma$ on $S$, $$l_d(\gamma)\geq \frac{1}{2} l_{d'}(\Phi(\gamma)).$$ 
    Thus, it suffices to show that $\inf_{\gamma\in\Gamma}l_{d'}(\Phi(\gamma))>0$. For simplicity, from now on we identify the convex surface with its Klein realization and suppress $\Phi$ from the notation.

We claim that there exists $c>0$ such that $
l_{d'}(\gamma) \geq c$ for every  $\gamma \in \Gamma.$
Suppose otherwise. Then there is a sequence $\gamma_{n} \in \Gamma$ such that $l_{d'}\left(\gamma_{n}\right) \to 0.$ Since the diameter of a connected rectifiable curve is at most its length,
$$
\operatorname{diam}_{d'}(\gamma _{n})\leq l_{d'}(\gamma_{n})\to 0.
$$
Choose $x_{n} \in \gamma_{n}$. Since $\hat{S}$ is compact, after passing to a subsequence,
$x_{n} \to x \in \hat{S}.$ Let $h\colon \hat{S} \to \SS^{2}$ be the homeomorphism constructed above. Then
$\operatorname{diam}_{\SS^{2}}\bigl(h(\gamma_{n})\bigr)\to 0.$
On the other hand, note that the loop $h(\alpha)$ and the set $h(\Omega^c)$ are disjoint closed sets in $\SS^{2}$.
Moreover, $\Omega^c$ (and hence $h(\Omega^c)$) contains at least two points.
Therefore, for sufficiently large $n$, the small loop $h(\gamma_{n})$ cannot separate $h(\alpha)$ from $h(\Omega^c)$. This is a contradiction, because each $\gamma_{n}\in\Gamma$ separates $\alpha$ from $\Omega^c$ in $\hat S$.
Thus
$\inf _{\gamma \in \Gamma} l_{d'}(\gamma)>0$, and then
$\operatorname{Mod}(R)<\infty.$
Hence,  $S$ is conformally equivalent to $\mathbb{D}$.
\end{proof}

The next proposition shows that the converse of Proposition \ref{typehy} holds.

\begin{proposition}
\label{typeEu}
    If $V$ is an infinite closed discrete subset of $\C$,   then the Riemann surface $\partial C(V\cup \{\infty\})\cup V$ is conformal to $\C$. 
\end{proposition}
\begin{proof}
 We identify $\C$ with  $\SS^2-\{N\}$ where $N$, the north pole of $\SS^2$, is identified with $\infty$.  Let us normalize $V$ so that the south pole $(0, 0, -1)\in V$. Let $S = \partial C(V\cup\{N\})\cup V$ be the simply connected Riemann surface obtained by filling in the cusps corresponding to $V$. Take a simple closed loop $\alpha$ in $S-\{(0, 0, 1), (0, 0, -1)\}$ such that $\alpha$ separates $(0, 0, 1)$ and $(0, 0, -1)$, and $\alpha$ is contained in $B_{d_{\mathbb{E}}}(N, 1) - B_{d_{\mathbb{E}}}(N, 1/2)$. 

    We denote by $R$ the component of $S-\{\alpha\}$ that does not contain $(0, 0, -1)$. Then the topological surface $R$ is homeomorphic to an annulus. Our goal is to show that $R$, with the induced complex structure from $S$, is biholomorphic to the punctured disk $\D-\{0\}$. As a consequence, we see that the Riemann surface $S$ has a punctured end at $N$. Thus, by the uniformization theorem, $S$ is biholomorphic to $\C$.

    Since the ring $R$ is biholomorphic to $\D-\{0\}$ if and only if its modulus is infinite, we will show that the modulus $\Mod(R) = +\infty$ by producing a sequence of pairwise disjoint annuli $R_1, R_2, \cdots, R_n, \cdots$ in $R$ so that
    \begin{enumerate}[label=(\arabic*)]
        \item each $R_i$ separates the two ends of $R$,
        \item the modulus $\Mod(R_i)$ of the annulus $R_i$ is uniformly bounded away from zero.
    \end{enumerate}
    By the superadditivity property of modulus of rings \cite{Hubbard2006}, $\Mod(R)\geq \sum_{i=1}^\infty \Mod(R_i)$. Hence, we have $\Mod(R) = +\infty$. 

    The annuli $R_i$ are constructed as follows. Recall that by \cite[Proposition 4.7]{luow2024}, if $\pi:\partial \H^3_P \to S \cup \{N\}$ is the continuous extension of the shortest-distance projection in the Poincaré metric and $(0,0,0) \in C(V\cup \{N\})$, then for all $x\in\partial\H^3_P$, we have 
    \begin{equation}
    \label{1/8}
        \frac{1}{8}d_{\mathbb{S}}(x, N)\leq d_{\mathbb{E}}(\pi(x), N)\leq 8d_{\mathbb{S}}(x, N),
    \end{equation}
    where $d_{\mathbb{E}}$ is the Euclidean metric and $d_{\mathbb{S}}$ is the spherical metric on $\partial \H^3_P$.  
    In our case, the condition $(0,0,0) \in C(V \cup \{N\})$ is satisfied due to the choice that the south pole $(0,0,-1) \in V$.
    It follows that $\pi(B_{d_{\mathbb{S}}}(N, r)) \subset B_{d_{\mathbb{E}}}(N, 8r)$, and
    \begin{equation*}
        \pi(B_{d_{\mathbb{S}}}(N, r))\cap (\SS^2 - B_{d_{\mathbb{E}}}(N, 9r)) = \varnothing.
    \end{equation*}
    In particular, this implies that the sets 
    $$W_k:=\pi(B_{d_{\mathbb{S}}}(N, 9^{-4k}) - B_{d_{\mathbb{S}}}(N, 9^{-4k-2}) )$$
    are pairwise disjoint in $S$ for large $k$.

    The following lemma shows that each $W_k$ contains an annulus $R_k$ separating $N$ and $\alpha$ in $R$ and $\liminf \Mod(R_k)>0$, which will complete the proof of Proposition \ref{typeEu}. 
\end{proof}

  \begin{lemma}
        Let $r>0$ be small; for example, assume that $r<1/(2\times 9^3)$. 
        Then the projection $$\pi(B_{d_{\mathbb{S}}}(N, 81r) - B_{d_{\mathbb{S}}}(N, r))$$ in $R$ contains an annulus $X$ separating $N$ and $\alpha$ such that $\Mod(X)\geq c_1>0$ for some constant $c_1$ independent of $r$.
    \end{lemma}
    \begin{proof}
    As we have seen from \eqref{1/8}, the two compact connected sets $\pi(B_{d_{\mathbb{S}}}(N, r))$ and $\pi(\SS^2 - B_{d_{\mathbb{S}}}(N, 81r)^\circ)$ are disjoint in $S\cup N$. Since $\alpha$ and $(0,0,-1)$ are contained in $\pi(\SS^2 - B_{d_{\mathbb{S}}}(N, 81r)^\circ)$, and $N\in\pi(B_{d_{\mathbb{S}}}(N, r))$,
    we see that there exists exactly one component $X$ of $S-(\pi(B_{d_{\mathbb{S}}}(N, r))\cup\pi(\SS^2 - B_{d_{\mathbb{S}}}(N, 81r)^\circ))$ such that
    \begin{enumerate}[label=(\arabic*)]
        \item $X$ separates $(0,0,1)$ and $(0,0,-1)$,
        \item $X$ is topologically an annulus.
    \end{enumerate}
    Let $\partial_1$ and $\partial_2$ be the boundary components of $X$ that are contained in $\pi(B_{d_{\mathbb{S}}}(N, r))$ and $\pi(B_{d_{\mathbb{S}}}(N, 81r))$ respectively. Furthermore, by the choice of $\alpha$ and the fact that $9\times81r<1/2$, we see that $X\subset R$ separates $N$ and $\alpha$. 

    To estimate the modulus $\Mod(X)$ of $X$, consider the family $\Gamma^*$ of all paths in $X$ joining $\partial_1$ and $\partial_2$. It is well known that $\Mod(X) = EL(\Gamma^*)$. Since the Euclidean metric $d_{\mathbb{E}}$ is conformal to the Poincar\'e metric in $\H^3_P$, the path metric $d_1$ on $X$ induced by $d_{\mathbb{E}}$ is conformal with respect to the complex structure on $X$. In particular, by definition, 
    $$EL(\Gamma^*)\geq \frac{l_{d_1}^2(\Gamma^*)}{A(d_1)}$$
    where $A(d_1)$ is the area of $X$ with respect to $d_1$. We now use \cite[Theorem 4.5]{luow2024}, which states:  
    \begin{theorem}
        Let $\pi:\partial\H^3_P\to S$ be the shortest-path projection in $\H^3_P$ onto a hyperbolic convex surface, and let $d_1$ be the path metric on $S$ induced by the Euclidean metric $d_{\mathbb{E}}$ on $\R^3$. Then for all $x, y \in \partial\H^3_P$, 
        $$d_1(\pi(x), \pi(y)) \leq 2d_{\mathbb{S}}(x, y).$$
    \end{theorem}
        The theorem implies that if $c$ is a spherical rectifiable path, then the length of $\pi \circ c$ on the surface $S$ is at most twice the length of $c$ on $\mathbb S^2$ by partitioning $c$.         
        Therefore the area $A(d_1)$ of $X$ is at most four times the spherical area of the annulus $\{x\in \SS^2 : r\leq d_{\mathbb{S}}(x, N)\leq 81r\}.$ Thus, $A(d_1)\leq c_2r^2$ for some universal constant $c_2$. 

        On the other hand, take any $\gamma^*\in \Gamma^*$. Since $\pi(B_{d_{\mathbb{S}}}(N, r))\subset B_{d_{\mathbb{E}}}(N, 8r)$ and $$B_{d_{\mathbb{E}}}(N, 9r) \cap \pi(\SS^2 - B_{d_{\mathbb{S}}}(N, 81r)^\circ) = \varnothing,$$ there exists a subarc $\tilde\gamma^*$
        of $\gamma^*$ joining the Euclidean spheres $\partial B_{d_{\mathbb{E}}}(N, 8r)$ and  $\partial B_{d_{\mathbb{E}}}(N, 9r)$. Therefore, $l_{d_1}(\gamma^*)\geq l_{d_1}(\tilde \gamma^*)\geq r$, which means that $l_{d_1}(\Gamma^*)\geq r$. As a consequence, 
        \begin{equation*}
            \Mod(X) = EL(\Gamma^*) \geq \frac{l_{d_1}^2(\Gamma^*)}{A(d_1)}\geq \frac{1}{c_2}.\qedhere
        \end{equation*}
    \end{proof}

Now we prove Theorem \ref{disriemap}.


\begin{proof}[Proof of Theorem \ref{disriemap}]

The uniqueness of $W$ up to M\"obius transformation follows from Theorem \ref{llz}.

To prove the existence of $W$, we invoke Theorem \ref{lw} to obtain a compact set 
$Z \subset \hat{\mathbb{C}}$ whose connected components are points and round disks such that the convex hull $\partial C(Z)$ is isometric to $\partial C(V \cup \Omega^c)$ in $\H^3_U$. Let 
$f:\partial C(Z)\to\partial C(V \cup \Omega^c)$ be an isometry.
Since $V$ is closed and discrete in $\Omega$,
by considering the complex structure on $\partial C(V \cup \Omega^c)$, 
all but one topological ends of $\partial C(V \cup \Omega^c)$ are cusps and hence removable singularities; the same holds for $\partial C(Z)$. Let $Z_c$ be the set of points in $Z$ corresponding to isolated cusp points in $V$ under $f$ and $E$ be the remaining component in $Z$. Since $V$ is countable, the Riemann surface $\partial C(V \cup \Omega^c)$ has countably many cusp ends. 
After filling in all cusp ends of $\partial C(V \cup \Omega^c)$, by Proposition \ref{typehy}, the resulting surface has only one end that is not a conformal puncture. 
We then perform the same filling on the corresponding cusp ends $Z_c$ of $\partial C(Z)$ via the isometry $f$.  The extended $f$ on the filled-in surfaces is still holomorphic. It follows that the cusp-filled-in surface for $\partial C(Z)$ has an end that is not a conformal puncture and is corresponding to $E$. 
By Proposition \ref{typeEu}, the last component $E$ of $Z$ must be a disk instead of a point.
Using a Möbius transformation, we may assume the component  $E$ to be $\mathbb{D}^c\subset \hat{\C}$. Let $W =Z \cap \D$. It remains to show that $\overline{W} - W = \partial \mathbb{D}$.

Suppose, to the contrary, that $\overline{W} - W \neq \partial \mathbb{D}$.  
The proof of Proposition \ref{polylimit} shows that there exists a closed hyperbolic half-plane $Q \subset \H^2_P$ and a proper local isometric embedding $\rho: Q \to \partial C(W \cup  \mathbb{D}^c)$.  Let $P=\rho(Q)$ and $\Phi : \partial C(W \cup  \mathbb{D}^c)
\to \partial C(V \cup \Omega^c)$ be the isometry.  
We now show that due to the existence of the half-plane $\Phi(P) \subset \partial C(V \cup \Omega^c)$, 
some neighborhood of a prime end of $\Omega$ is disjoint from $V$, which contradicts the hypothesis that $V$ is a discrete approximation of $\Omega$.  
The proof of this statement is divided into the following four steps.

\vspace{1ex}
\noindent \textbf{Step 1.} We first show that there is no bi-infinite geodesic $\gamma\subset \mathbb{H}^3_U$ contained in  $\Phi(P)$.
Assume, to the contrary, that such a geodesic $\gamma$ exists. Notice that if either endpoint of $\gamma$ is a cusp point in $V$, then $\gamma$ cannot be contained in a half-plane $\Phi(P)$. Thus, the closure of
$\gamma$ is an embedded closed arc whose two endpoints are distinct points in $\Omega^c$. 

Since $\Phi|_P$ is a proper isometric embedding and $P$ is a closed hyperbolic half-plane,
$\gamma$ is a complete geodesic contained in $\Phi(P)$.
Either $\gamma$ is the boundary of $\Phi(P)$ or there are two components of
$\Phi(P)-\gamma$. If $\gamma =\partial \Phi(P)$, let $U =\Phi(P) -\gamma$. Otherwise, let $U$ be the unique component of $\Phi(P)-\gamma$, whose closure in $\Phi(P)$ is disjoint from the geodesic boundary $\partial \Phi(P)$.

Since $V$ is an approximation to  the simply connected domain $\Omega$, the convex hull $C(V \cup \Omega^c)$ is 3-dimensional. The radial projection from an interior point of $C(V \cup \Omega^c)$ defines a
homeomorphism
 $ h:\Omega^c \cup \partial C(V\cup \Omega^c)\cup V \to \mathbb S^2$ such that $h$ is the identity map on $V \cup \Omega^c$. 
It follows that $\partial C(V\cup \Omega^c)\cup V$ is homeomorphic via $h$ to the simply connected domain $\Omega$. By the Jordan curve theorem, $\gamma$ separates $\partial C(V\cup \Omega^c)\cup V$ into two components. Since $U$ is connected, $U$ is contained in one of the components. On the other hand, since $\Phi(P)$ is a closed subset of $\partial C(V\cup \Omega^c)\cup V$, we see that $U$ is also a closed subset of this component. Therefore, $U$ is a component of $\partial C(V\cup \Omega^c)\cup V -\gamma $. 

Hence, we have that
$\alpha:=h(\gamma)$
is a crosscut of $\Omega$, and that $h(U)$ is one of the two components of $\Omega-\alpha$.
Finally, since
$h(U)\cap V=\varnothing$, 
we have produced a crosscut $\alpha$ of $\Omega$ for which one component of
$\Omega-\alpha$ is disjoint from $V$. This contradicts the hypothesis that $V$ is a discrete
approximation of $\Omega$.

\vspace{1ex}
\noindent \textbf{Step 2.} We recall that the structure theorem for the convex surface $\partial C(X)$, where $X \subset \hat{\mathbb{C}}$ is a compact set containing at least three points (see Thurston \cite{thurston} or 
Epstein--Marden \cite{epsteinm2006}).  
A face $F \subset \partial C(X)$ is a surface with boundary such that 
$F = R \cap C(X)$, where $R$ is a supporting plane for $C(X)$ in $\mathbb{H}^3_U$.  
It is isometric to the convex hull $C(K) \subset \mathbb{H}^2$, where 
$K$ is a compact set in $\partial \mathbb{H}^2$.
A bending line is either a bi-infinite geodesic $\gamma$ such that $\gamma = R \cap C(X)$, where $R$ is a supporting plane for $C(X)$ in $\mathbb{H}^3_U$ or a boundary component of $\partial F$ for a face $F$. 
The structure theorem in \cite{thurston, epsteinm2006} for $\partial C(X)$ states that 
$\partial C(X)$  is the disjoint union of bending lines and interiors of faces. Since the number of faces is countable, there are countably many bending lines in the decomposition that come from the boundary components of faces. 

By Step 1, we see that if $F$ is a face of $\partial C(V \cup \Omega^{c})$ such that  $F \cap \Phi(P) \neq \varnothing$, no bi-infinite geodesic $\gamma$ in $F$ is contained in $\Phi(P)$.  In particular,  no boundary 
component of $\partial F$ is contained in $\Phi(P)$, and $\Phi(P)$ is not contained in the face $F$.  It follows that the nonempty intersection of the half-plane $Q$ with $\rho^{-1}(\Phi^{-1}({F}^\circ))$ is an open triangle as shown in the right-hand picture of Figure \ref{fig:F-P}.  This shows that there is at most one point in the conformal infinity $\partial Q \cap \partial \D$ of $Q$ that comes from the conformal infinity of $\rho^{-1}(\Phi^{-1}({F}^\circ))$ for each face $F$. Since the set of all faces is countable, there are countably many conformal infinity points of $\cup_{F} \rho^{-1}(\Phi^{-1}({F}^\circ))$ in $\partial Q \cap \partial \D$. 

\begin{figure}[htbp]
  \centering
  \begin{overpic}[width=0.85\linewidth]{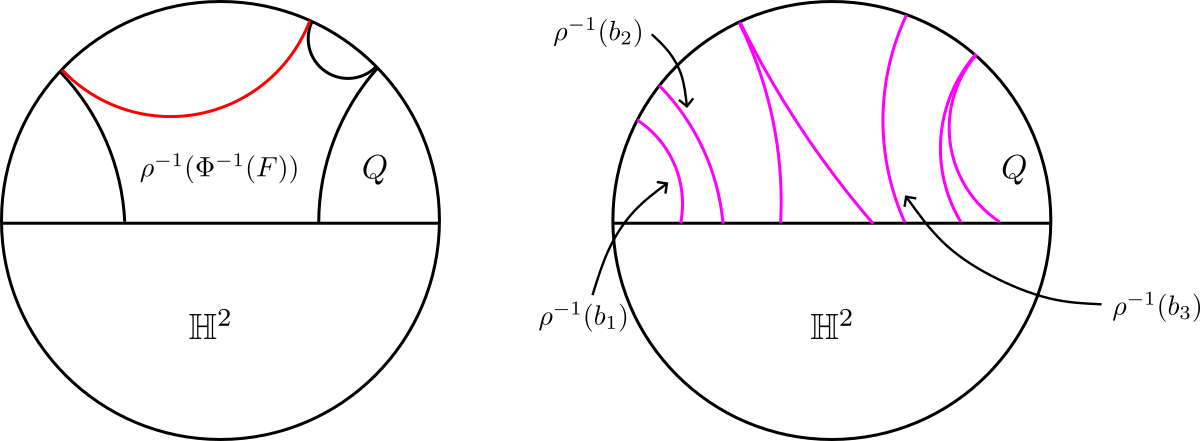}
\end{overpic}
\caption{In the left figure, there is a bi-infinite geodesic contained in both $\rho^{-1} (\Phi^{-1}(F))$ and the half-plane $Q$, which cannot happen. The right figure is the only possibility. }
\label{fig:F-P}
\end{figure}

Since $\partial C(V \cup \Omega^c)$ is a disjoint union of open faces and bending lines, restricting this decomposition to $\Phi(P)$ shows that $\Phi(P)$ is the disjoint union of the geodesic rays $\Phi(P) \cap L$ as $L$ ranges over the bending lines, and the sets
$\Phi(P)\cap {F}^\circ$, as $F$ ranges over the faces.

Pulling the decomposition back to $Q$ via the isometry $\Phi \circ \rho$, we see that $Q$ is a disjoint union of geodesic rays and countably many open triangles. There must be uncountably many geodesic rays in the decomposition since each point in the conformal infinity of $Q$ is the endpoint of a geodesic ray in the decomposition. 

It follows that there are three geodesic rays $\rho^{-1}(b_1), \rho^{-1}(b_2), \rho^{-1}(b_{3})$ in the decomposition of $Q$ such that

\begin{enumerate}
\item  there exists a supporting plane $B_{i}$ for $C(V\cup \Omega^c)$ such that the geodesic ray 
$\Phi(b_{i}) \subset B_{i} \cap \partial C(V \cup \Omega^{c})$.
\item the endpoints of geodesic rays $\rho^{-1}(b_{i})$ in 
the conformal infinity $\partial \D$ are pairwise distinct,  and are disjoint from the end points of the geodesic $\partial Q$ and all conformal infinity of $\rho^{-1}(\Phi^{-1}({F}^\circ))$ for all faces $F$;
\end{enumerate}
Note that condition (2) implies none of the geodesic rays $\rho^{-1}(b_i)$ is contained in $\rho^{-1}(\Phi^{-1}(F))$ for any face $F$ of $\partial C(V \cup \Omega^c)$. Furthermore, two distinct rays $\Phi(b_i)$ and $\Phi(b_j)$ cannot be contained in a single supporting plane $B$.  Otherwise, $B\cap \partial C(V \cup \Omega^c)$ would be a face of the convex hull
containing the two rays $\Phi(b_i)$ and $\Phi(b_j)$. This contradicts condition (2). 
Thus, the supporting planes $B_i$ of $C(V \cup \Omega^c)$ that contain $\Phi(b_i)$ are pairwise distinct. 

\vspace{1ex}
\noindent \textbf{Step 3.}
Next, we show that at least two geodesic rays $\Phi(b_{i})$ and $\Phi(b_{j})$ 
end at different points in $\partial \Omega$.  
Suppose otherwise; say all of them end at $r \in \partial \Omega$.  

We claim that at least two supporting planes $B_{i}$ and $B_{j}$ intersect in $\mathbb{H}^3_U$.  
Otherwise, $\partial B_{1}$, $\partial B_{2}$ and $\partial B_{3}$ are tangent at $r$ in 
$\partial \mathbb{H}^3$.  
Let $D_{i}$ be the closed disk in $\partial \mathbb{H}^3_U$ bounded by $\partial B_{i}$ such that $C(D_i)$ is a supporting half-space for $C(V\cup \Omega^c)$. By the definition of supporting plane and the convex hull of sets in $\partial \H^3_U$, we see that $D_{i}^\circ \cap (V\cup\Omega^c) = \varnothing$ and $|\partial D_i\cap (V\cup\Omega^c)|\geq 2 $. 
Now, due to tangency, we see that one of the $D_i$ is contained in $D_j$ for $j\neq i$. Since $\partial D_i \subset D_j^\circ \cup \{r\}$, it follows that $|\partial D_i\cap (V\cup\Omega^c)|=1$, which contradicts the fact that $B_i$ is a supporting plane. 
Therefore, the claim holds.
Without loss of generality, we may assume $B_{1} \cap B_{2} = L \neq \varnothing$.

\begin{figure}[htbp]
  \centering
  \begin{overpic}[width=0.4\linewidth]{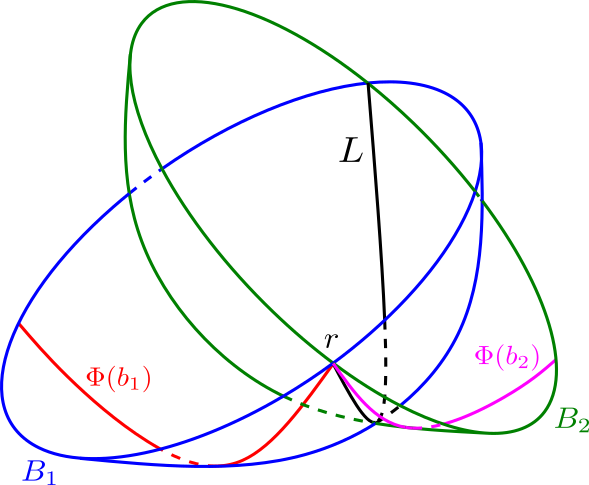}
\end{overpic}
\caption{Two supporting planes $B_1$ and $B_2$ intersecting at $L$, with geodesics $\Phi(b_1)$ and $\Phi(b_2)$.}
\label{fig:L-B}
\end{figure}

Since both $\Phi(b_{1})$ and 
$\Phi(b_{2})$ end at $r$, the geodesic $L$ also ends at the point $r$. See Figure \ref{fig:L-B}.
Let $b_{1}(t)$, $b_{2}(t)$ with $t \in [0, \infty)$, and $L(t)$ with $t \in \R$ be arc-length parameterizations of $b_1, b_2$ and the geodesic $L$, respectively, such that 
$$
r = \lim_{t \to \infty} \Phi(b_{i}(t)) = \lim_{t \to \infty} L(t).
$$
Since both geodesic rays $L|_{[t_0,\infty)}$ and $\Phi(b_{i})|_{[t_0,\infty)}$ in $\H^3_U$ end at the same point $r$,  
there exists a constant $C > 0$ such that the distance in $B_{i}$ satisfies
$$
d_{B_i}\bigl(\Phi(b_{i})(t), L(t)\bigr) =d_{\mathbb{H}^3}\bigl(\Phi(b_{i})(t), L(t)\bigr)  \le C, \qquad \text{for all } t \ge 0.
$$
Let 
$$
\Psi : B_{1} \cup B_{2} \to \partial C(V \cup \Omega^{c})
$$
be the restriction of the shortest distance projection from $\H^3_U-C(V \cup \Omega^c)^{\circ}$ to $\partial C(V \cup \Omega^c)$ which decreases length.
It follows that
$$
d_{\partial C(V \cup \Omega^{c})}\bigl(\Phi(b_{i}(t)), \Psi (L(t))\bigr) \le    d_{\H^3} \bigl(\Phi(b_{i}(t)), L(t))\bigr) \le  C
$$
for all $t\geq 0$. In particular, by the triangle inequality,
$$
d_{\partial C(V \cup \Omega^{c})}
\bigl(\Phi(b_{1})(t), \Phi(b_{2})(t)\bigr) \leq 
d_{\partial C(V \cup \Omega^{c})}
\bigl(\Phi(b_{1})(t), \Psi(L(t)) \bigr)+ d_{\partial C(V \cup \Omega^{c})}
\bigl(\Psi(L(t)), \Phi(b_{2})(t)\bigr)
\le 2C.
$$
We claim that for sufficiently large $t$, 
\begin{equation}\label{555}
d_{Q}(\rho^{-1}(b_{1}(t)), \rho^{-1}(b_{2}(t)))
\le 2C.
\end{equation}
Indeed, let $t_1$ be a number such that for $t>t_1$, the distance from $\rho^{-1}(b_i(t))$ to the geodesic $\partial Q$ is larger than $2C$ for $i=1,2.$  Let $c_t$ be the shortest path in $\partial C(V \cup \Omega^c)$ joining $\Phi(b_1(t))$ to $\Phi(b_2(t))$ of length at most $2C$.  Then $c_t$ must be contained in $\Phi\circ\rho(Q)$; otherwise the length of $c_t$ is at least the distance from $\rho^{-1}(b_i(t))$ to $\partial Q$, which is greater than $2C.$ Thus $(\Phi \circ \rho)^{-1}(c_t)$ is the geodesic segment in $Q$ joining $\rho^{-1}(b_1(t))$ to $\rho^{-1}(b_2(t))$.
It follows that $d_{Q}(\rho^{-1}(b_{1}(t)), \rho^{-1}(b_{2}(t)))$ equals the length of $(\Phi \circ \rho)^{-1}(c_t)$ which is at most $2C.$

But on the other hand, the two rays $\rho^{-1}(b_{1}) $ and $\rho^{-1}(b_{2})$ end at different points in the conformal infinity of $Q$ by condition (2), which implies $\lim_{t\to \infty}  d_{Q}(\rho^{-1}(b_{1}(t)), \rho^{-1}(b_{2}(t)))      = \infty$. This contradiction shows $\Phi(b_1)$ and $\Phi(b_2)$
end at different points in $\partial \Omega$.

\vspace{1ex}
\noindent \textbf{Step 4.}
Finally, take $t'>t_1$, we join $b_{1}(t')$ and $b_{2}(t')$ by the shortest geodesic in $P$ and form 
the concatenation 
$$
\alpha = b_{1}|_{[t',\infty)} \cup [b_{1}(t'), b_{2}(t')] \cup b_{2}|_{[t',\infty)}.
$$
Then $\alpha$ is a simple crosscut of $P$.

By construction, $h(\Phi(\alpha))$ is a crosscut of $\Omega$.  
Let $Z$ be the component of $P - \alpha$ disjoint from $\partial P$.  
Then $h(\Phi(Z))$ is a component of 
$$
\Omega - h(\Phi(\alpha))
$$
disjoint from $V$.  
This contradicts the assumption on $V$ and then completes the proof.\qedhere

\end{proof}

\appendix

\section{A Proof of Lemma \ref{lem:Q-S-X-d}}\label{appen:lemma2.7}

Recall that Lemma \ref{lem:Q-S-X-d} states: 
\newtheorem*{lemma2-7}{Lemma \ref{lem:Q-S-X-d}}
\begin{lemma2-7}
    Suppose $(S, V, d)$ is a simply connected complete Riemannian surface of constant curvature $0, -1, 1$, where $V \subset S$ is a non-empty closed subset.  Assume that $(S,V)$ is non-exceptional and, in the case of curvature $d$ is $1$, that $V$ is finite. Then
     there exists a 
    homeomorphism $Q_{S,V,d}: S - V \to \partial C(Y)$ such that if  $(S_1, V_1, d_1)$ and $(S_2, V_2, d_2)$ are simply connected and related by an isometry $A \in \operatorname{Isom}(\mathbb{H}^3)$ with $A(S_1) = S_2$ and $A(V_1) = V_2$,  then
    \begin{equation} \label{equivq1}
        A \circ Q_{S_1, V_1, d_1}(\xi) = Q_{S_2, V_2, d_2}\circ A (\xi),\quad\forall \xi\in S_1 - V_1.
    \end{equation}
\end{lemma2-7}

\begin{proof}
Let us consider the following three cases.

\medskip
 \textbf{Case 1.} $(S, d)$ is the hyperbolic plane $\mathbb{H}^2_P$.  The restriction of the canonical projection $\phi_{\H}$ to $\partial C(Y)$ is a homeomorphism onto $\H^2_P-V$ due to convexity.  
  We define the map $Q_{S,V,d}: \H_P^2 - V \longrightarrow \partial C(Y)$  to be the inverse map $(\phi_{\H}|_{\partial C(Y)})^{-1}$. By Lemma \ref{natural-h}(i), the equation \eqref{equivq1} holds.  

\medskip
    \textbf{Case 2.} $(S, d)$ is the unit sphere $\mathbb{S}^2$. Then the set $V$ is finite and $|V| \geq 3$.
    For each finite subset $Z \subset \mathbb{S}^2$, let
    $$\mu_Z = \frac{1}{|Z|}\sum_{\xi \in Z}\delta_\xi$$
    be the normalized counting measure on $Z$, and let $b(Z)$ be the conformal barycenter of $Z$ with respect to $\mu_Z$ (see \cite{Douady-Earle} for background on the conformal barycenter). Since $|V| \geq 3$, $b(V) \in C(Y)$ is well-defined by \cite{Douady-Earle}, and for any hyperbolic isometry $A \in \operatorname{Isom}(\mathbb{H}^3_P)$, we have $b(A(V)) = A(b(V))$.
    \begin{enumerate}
        \item[]\textbf{Case 2.1.} If $C(Y)$ is three-dimensional, then $b(V) \in  C(Y)^\circ$. We define $Q_{S,V,d}$ to be the inverse of the hyperbolic radial projection from $\partial C(Y)$ to $\mathbb{S}^2 - V$ centered at $b(V)$. 

        \item[]\textbf{Case 2.2.} 
        If $C(Y)$ is two-dimensional, i.e., $Y$ is contained in a circle $M \subset \mathbb S^2$, then $b(V)$ lies in the relative interior of $C(Y)$. 
        Let $\Omega_+$ and $\Omega_-$ be the two connected components of $\mathbb{S}^2 - M$.
        For each $w \in \mathbb{S}^2 - M$, let $P(w)$ be the unique point in the hyperbolic plane $C(M)$ such that the geodesic line passing through $P(w)$ and $w$ is orthogonal to $C(M)$. Extend $P$ to a continuous map from $\mathbb S^2$ to $C(M) \cup M$, still denoted by $P$, by setting $P(m) =m$ for $m \in M.$
        Then the restriction of $P$ to $\Omega_i \cup M$, $i \in \{+,-\}$,  is a homeomorphism onto $C(M) \cup M$.  We now construct a homeomorphism $Q_V$ from $C(M)\cup M$ to $C(Y)\cup Y$ as follows. For an arc-length parameterized geodesic ray $\gamma (t)$ emanating from $b(V)$, $t \in [0, \infty]$, we define
        $$Q_V(\gamma(t)) := \gamma\left(\frac{t}{1+t/R_V}\right),$$
        where $R_V\in(0, \infty]$ is the parameter such that $\gamma(R_V) \in \partial C(Y) \cup Y$. 

For each $i\in\{+,-\}$,
define the homeomorphism $F_i(x)=(Q_V(P(x)), i):  \Omega_i \cup M \to (C(Y) \cup Y) \times \{i\}$. Then $F_+|_M =F_-|_M$. The restriction to $S-V$ of the gluing $F_+\cup F_-$ is a homeomorphism from $S-V$ to the Alexandrov double $\partial C(Y)$.

 \end{enumerate}


    \textbf{Case 3.} $(S, d)$ is the Euclidean plane $\E^2$. Let $C_{\E}(V)$ be the Euclidean convex hull of $V$. For each $\xi\in C_{\E}(V)$, we define a map $R: C_{\E}(V) \longrightarrow \partial C(Y)\cup V$ by setting $R(\xi)=\xi$ when $\xi\in V$, and letting $R(\xi)$ be the first point of intersection between $\partial C(Y)$ and the vertical ray emanating from $\xi$ when $\xi\in C_{\E}(V) - V$.  In terms of the canonical projection $\phi_{\E}$, we have that $\phi_{\E} \circ R$ is the identity map. 
    
    \begin{enumerate}
        \item[] \textbf{Case 3.1.} $C_{\mathbb{E}}(V)=\mathbb{E}^2$. By convexity, $\phi_{\E}|: \partial C(Y) \to \E^2-V$ is a homeomorphism. Thus we define $Q_{S,V,d}$ to be the inverse map $(\phi_{\E}|_{\partial C(Y)})^{-1}$.  Furthermore, by Lemma \ref{natural-h}(i), the equation \eqref{equivq1} holds.

        \item[] \textbf{Case 3.2.} $C_{\mathbb{E}}(V)$ is bounded. Then $C_{\mathbb{E}}(V)$ is a closed set. We let $\xi_0$ be the barycenter of $C_{\mathbb{E}}(V)$.
    
        (1) The interior $C_{\mathbb{E}}(V)^\circ$ is nonempty. For each $\xi \in \mathbb{E}^2 - C_{\mathbb{E}}(V)$, let $P(\xi)$ be the first intersection point of the ray $\xi_0 + \mathbb{R}_{\geq 0}(\xi - \xi_0)$ with $\partial C_{\mathbb{E}}(V)$.  We then define the map $Q_{S,V,d}: \mathbb{E}^2 - V \longrightarrow \partial C(Y)$ as follows:
        \begin{equation}\label{eq:Q-1}
            Q_{S,V,d}(\xi) = 
            \begin{cases}         
                R(\xi), & \xi \in C_{\mathbb{E}}(V) - V, \\         
                R(P(\xi)) + \left(0, 0,\Vert\xi - P(\xi)\|\right), & \xi \in \mathbb{E}^2 - C_{\mathbb{E}}(V).   
            \end{cases}
        \end{equation}
    
        (2) The interior $C_{\mathbb{E}}(V)^\circ$ is empty. Then $C_{\mathbb{E}}(V)$ is a line segment of finite length. Let $p_1$ and $p_2$ be the two endpoints of $C_{\mathbb{E}}(V)$, and let $l$ be the line passing through $C_{\mathbb{E}}(V)$. For each $\xi \in l - C_{\mathbb{E}}(V)$, let $P(\xi)$ denote the endpoint of $C_{\mathbb{E}}(V)$ nearest to $\xi$. We define a homeomorphism $Q$ from $l - V$ to the relative boundary of $C(Y)$ by 
        \begin{equation*}
            Q(\xi) = 
            \begin{cases}         
                R(\xi), & \xi \in C_{\mathbb{E}}(V) - V, \\         
                P(\xi) + \left(0, 0,\Vert\xi - P(\xi)\|\right), & \xi \in l - C_{\mathbb{E}}(V).      
            \end{cases}
        \end{equation*}
        Let $H_+$ and $H_-$ be the two connected components of $\mathbb{E}^2 - l$. We now extend $Q$ to a homeomorphism $Q_{S,V,d}$ from $\mathbb{E}^2 - V$ to $\partial C(Y)$ by defining a homeomorphism $Q_i$ from $H_i$ to the relative interior of the copy $C(Y)_i$ for each $i \in \{+,-\}$.  For each $\xi\in(p_1,p_2)$, let $L_i(\xi)$ be the ray in $H_i$ emanating from $\xi$ such that the angle between $L_i(\xi)$ and the ray from $p_2$ to $\infty$ through $p_1$ is $\frac{\pi\|\xi - p_1\|}{\|p_2 - p_1\|}$. Then $\cup_{\xi\in(p_1,p_2)} L_i(\xi) = H_i$. We then let $Q_i$ be the map sending $L_i(\xi)$ isometrically to the vertical ray emanating from $R(\xi)$ for each $\xi\in(p_1,p_2)$. 

\input{tikz-fig1}

        \item[]\textbf{Case 3.3.} $C_{\mathbb{E}}(V)\neq\mathbb{E}^2$  is a proper unbounded subset of $\E^2$, $C_{\mathbb{E}}(V)^\circ\ne\varnothing$, and $C_{\mathbb{E}}(V)$ is closed. Let $\Lambda$ be the set of unit vectors $v \in \mathbb{S}^1$ such that $p_0 + \mathbb{R}_+ v \subset C_{\mathbb{E}}(V)$ for some $p_0 \in \mathbb{E}^2$. It is clear that either $\Lambda=\{v,-v\}$ for some $v\in\mathbb{S}^1$, or $\Lambda$ is an arc of length $l(\Lambda)\in[0,\pi]$ (which reduces to a single point when $l(\Lambda)=0$).
        
        (1) $\Lambda=\{v,-v\}$ for some $v\in\mathbb{S}^1$. 
        Let $w$ be a unit vector orthogonal to $v$. For each $\xi \in \mathbb{E}^2 - C_{\mathbb{E}}(V)$, let $P(\xi)$ denote the intersection point of the line $\xi + \mathbb{R}w$ and $\partial C_{\mathbb{E}}(V)$ that is closest to $\xi$. 
        We define the map $Q_{S,V,d}: \mathbb{E}^2 - V \longrightarrow \partial C(Y)$ as in \eqref{eq:Q-1}.

        (2) $\Lambda$ is an arc or a point of length $l(\Lambda)\in[0,\pi]$. Let $w \in \Lambda$ denote the midpoint of $\Lambda$. If $\Lambda\neq\{w\}$, then every ray $z+\R_{\geq 0} w$ intersects $C_{\E}(V)$. We define the map $Q_{S,V,d}: \mathbb{E}^2 - V \longrightarrow \partial C(Y)$ as in \eqref{eq:Q-1}, where $P(\xi)$ is the first intersection point of the ray $\xi + \mathbb{R}_{\geq 0}w$ with $\partial C_{\mathbb{E}}(V)$ for each $\xi \in \mathbb{E}^2 - C_{\mathbb{E}}(V)$. Now suppose $\Lambda=\{w\}$. Let 
        $$
        m=\min\{x\cdot w:x\in C_{\mathbb{E}}(V)\},\quad I=\{x\in C_{\mathbb{E}}(V):x\cdot w=m\},
        $$
         where $I$ is a compact interval or a single point since $\Lambda$ is a point.  Let $\xi_0$ be the midpoint of $I$ and  $\Lambda'$ be the set of unit vectors $v \in \mathbb{S}^1$ such that $(\xi_0 + \mathbb{R}_+ v )\cap  C_{\mathbb{E}}(V)\neq\varnothing$. 
        Since $C_{\mathbb{E}}(V)^\circ \neq \varnothing$, $\Lambda'$ is an arc with positive length. Let $w'$ denote the midpoint of $\Lambda'$. We seek to determine a point $\xi_0'\in C_{\mathbb{E}}(V)^{\circ}$ that is equivariant under Euclidean similarity transformations.
        If $w'\neq w$, then $\partial C_{\mathbb{E}}(V)$ has a unique intersection point $\xi'$ with the ray $\xi_0 + \mathbb{R}_{+}w'$. We can simply define $\xi_0'$ as the midpoint of the line segment connecting $\xi_0$ and $\xi'$. Now consider the case $w'=w$.
        Let $w_1$ and $w_2$ be the midpoints of the two connected components of $\Lambda' -\{w\}$. For each $i\in\{1,2\}$, let $\xi_i$ be the unique intersection point of the ray $\xi_0 + \mathbb{R}_{+}w_i$ with $\partial C_{\mathbb{E}}(V)$. Let $\xi_0'$ be the midpoint of the line segment connecting $\xi_1$ and $\xi_2$. 
        By the convexity of $C_{\mathbb{E}}(V)$, we still have $\xi_0'\in C_{\mathbb{E}}(V)^\circ$.
        For each $\xi \in \mathbb{E}^2 - C_{\mathbb{E}}(V)$, let $P(\xi)$ denote the unique intersection point of the ray $\xi_0' + \mathbb{R}_{\geq 0}(\xi - \xi_0')$ with $\partial C_{\mathbb{E}}(V)$. We then define the map $Q_{S,V,d}\colon \mathbb{E}^2 - V \longrightarrow \partial C(Y)$ as in \eqref{eq:Q-1}.
        \item[] \textbf{Case 3.4.} $C_{\mathbb{E}}(V)$ is unbounded and $C_{\mathbb{E}}(V)^\circ = \varnothing$. It follows that $C_{\mathbb{E}}(V)$ is either a line or a half-line. 
        
        (1) $C_{\mathbb{E}}(V)$ is a line. Let $P$ denote the nearest-point projection of $\mathbb{E}^2$ onto $C_{\mathbb{E}}(V)$. We then define the map $Q_{S,V,d}: \mathbb{E}^2 - V \longrightarrow \partial C(Y)$ using the same formula as in \eqref{eq:Q-1}. 
      By the definition of Alexandrov double $\partial C(Y)$, we see that  $Q_{S,V,d}$ is a homeomorphism.

\input{tikz-fig2}

    (2)      $C_{\mathbb{E}}(V)$ is a half-line. Identify $\E^2$ with the $xy$-plane $\R^2$. We may assume, after applying an isometry of $\E^2$, that $C_{\E}(V)$ is the closed ray $[0, \infty)\times \{0\} \subset \R^2$.
    Define a homeomorphism $Q: C_{\E}(V) \times \R_{\geq 0} \to C(Y)\cup V$ by the formula:
    $$Q(\xi, t) =R(\xi) + (0,0, t).$$
    Note that by definition, $Q(\xi, 0) =(\xi, 0)$ for $\xi \in V$.
    Consider the degree-one homogeneous map
    $F: \E^2 \to C_{\E}(V) \times \R_{\geq 0} =\{(x, 0, z) \in \R^3 : x,z \geq 0\}$ given by
    $$ F(r e^{i \theta}) =(r\cos(\theta/2), 0, r|\sin(\theta/2)|), \quad -\pi \leq \theta  \leq \pi. $$
    One easily checks that the restrictions of $F$ to the closed upper half-plane and to the closed lower half-plane are both homeomorphisms onto $C_{\E}(V) \times \R_{\geq 0}$, $F|_V =id$,  and $Q\circ F|_{\R \times \{0\} - V}$ is a homeomorphism onto the $1$-dimensional boundary of $C(Y)$. 
     It follows that the composition $Q_{S,V,d}:=Q \circ F|: \E^2-V \to \partial C(Y)$ is a homeomorphism.

\item[]\textbf{Case 3.5.} The Euclidean convex hull $C_{\mathbb{E}}(V)$ is not closed. Let $\operatorname{Sim}(\mathbb{E}^2)$ denote the group of similarity transformations on $\mathbb{E}^2$, and let
$$
G_V = \{ A \in \operatorname{Sim}(\mathbb{E}^2) : A(V) = V \}.
$$
In what follows, we identify a Euclidean similarity $A$ with its unique extension to an isometry of the upper half-space model $\mathbb{H}^3_U$ fixing $\infty$. We have the following lemma:

\begin{lemma}
    Let $V$ be a closed subset of $\mathbb{E}^2$. If $C_{\mathbb{E}}(V)$ is not closed, then $G_V$ is a compact subgroup of $\operatorname{Sim}(\mathbb{E}^2)$. 
\end{lemma}
\begin{proof}
    Since $V$ is closed, $G_V$ is a closed subgroup of the Lie group $\operatorname{Sim}(\mathbb{E}^2)$. We show that $G_V$ cannot contain translations or dilations, from which its compactness will follow.

    We first show that $G_V$ contains no nonzero translation.
    Suppose, for contradiction, that $G_V$ contains a translation $A(x) = x + a$ with $a \neq 0$. Since $C_{\mathbb{E}}(V)$ is convex and invariant under $x \mapsto x + a$, it must contain the affine line $x + \mathbb{R}a$ for every $x \in C_{\mathbb{E}}(V)$. Consequently, $C_{\mathbb{E}}(V)$ is of the form $\mathbb{R}a + I$, where $I$ is a convex interval in the orthogonal line $L = (\mathbb{R}a)^\perp$.
    Let $\pi : \mathbb{E}^2 \to L$ be the orthogonal projection. We claim that $\pi(V)$ is closed. Indeed, let $\{v_n\} \subset V$ such that $\pi(v_n) \to y \in L$. Since $V$ is invariant under $A$, by shifting each $v_n$ by a suitable integer multiple of $A$, we may assume that the resulting sequence $\{v_n'\} \subset V$ is bounded. Passing to a subsequence, $v_n'$ converges to some $v \in V$, which satisfies $\pi(v) = y$. Thus $\pi(V)$ is closed in $L$. 
    Since the convex hull of any closed subset of $\mathbb{R}$ is a closed interval, $I = C_{\mathbb{E}}(\pi(V))$ is closed in $L$. It follows that $C_{\mathbb{E}}(V) = \mathbb{R}a + I$ is closed in $\mathbb{E}^2$, which is a contradiction.

    We next demonstrate that every element of $G_V$ is an isometry of $\E^2$.
    Suppose otherwise that there exists $A \in G_V$ which is not an isometry of $\E^2$. Then $A$ has a unique fixed point in $\mathbb{E}^2$. By translating $V$ and replacing $A$ by $A^{-1}$ if necessary, we may assume that the origin $0$ is the fixed point of $A$ and write $A(x) = \lambda R x$, where $\lambda > 1$ and $R \in \mathrm{O}(2)$. 
    For any $x \in V$, we have $A^{-n}(x) \in V$ and $A^{-n}(x) \to 0$ as $n \to \infty$. As $V$ is closed, $0 \in V$.
    If $R$ has finite order, then some positive power of $A$ is a homothety $x \mapsto \mu x$ with $\mu > 1$. Since $0 \in V$, the convexity of $C_{\mathbb{E}}(V)$ implies that $C_{\mathbb{E}}(V)$ is a convex cone with vertex at the origin. Let 
    $$
    K = V \cap \{ x \in \mathbb{E}^2 : 1 \leq \|x\| \leq \mu \}.
    $$
    Then $K$ is a compact set bounded away from the origin, and $C_{\mathbb{E}}(V)$ is the smallest convex cone at the origin containing $K$. It follows that $C_{\mathbb{E}}(V)$ is closed since it is generated by a compact set of directions, which is a contradiction.
    If $R$ has infinite order, $R$ is an irrational rotation. For any $v \in V \setminus \{0\}$, the directions of the points $\{A^n(v)\}_{n \in \mathbb{Z}}$ are dense in $\mathbb{S}^1$. Moreover, $\|A^n(v)\| = \lambda^n \|v\| \to \infty$ as $n \to +\infty$. This implies that $C_{\mathbb{E}}(V) = \mathbb{E}^2$, which is closed, again a contradiction.
    Therefore, every element of $G_V$ is in $\operatorname{Isom}(\mathbb{E}^2)$. 
    
    Since $G_V$ contains no nontrivial translations, all elements of $G_V$ must share a common fixed point. To see this, note that the commutator of two nontrivial rotations with distinct centers is a nontrivial translation, and nontrivial translations also arise as squares of glide reflections or products of reflections across distinct parallel lines. Hence, all elements of $G_V$ share a common fixed point $p \in \mathbb{E}^2$, and every element of $G_V$ is either a rotation about $p$ or a reflection across a line passing through $p$. It follows that $G_V$ is conjugate to a closed subgroup of $\mathrm{O}(2)$, and is therefore compact. This completes the proof.
\end{proof}

Now, for each closed subset $V \subset \mathbb{E}^2$, let $Y = V \cup \{\infty\} \subset \partial_\infty \mathbb{H}^3_U$. We aim to choose a basepoint in the interior of the hyperbolic convex hull $C(Y)$ that is equivariant under Euclidean similarities. 
Since $V$ is closed and $C_{\mathbb{E}}(V)$ is not closed, $V$ cannot be contained in a line. Therefore, $C(Y)$ is three-dimensional with $C(Y)^\circ \neq \varnothing$. 
The compact group $G_V$ acts by isometries on $\mathbb{H}^3$ and preserves $C(Y)^\circ$. For any $z \in C(Y)^\circ$, let $b(G_V \cdot z)$ denote the Cartan barycenter of the orbit $G_V \cdot z$, which is the unique minimizer in $\mathbb{H}^3$ of the displacement function
$$
x \mapsto \int_{G_V} d(x, g \cdot z)^2 \, d\mu(g),
$$
where $\mu$ is the normalized Haar measure on $G_V$. By the convexity of $C(Y)$ and the invariance of the metric, $b(G_V \cdot z)$ is a $G_V$-fixed point lying in $C(Y)^\circ$.

Let $\mathcal{S}$ denote the set of closed subsets of $\mathbb{E}^2$ whose Euclidean convex hulls are not closed, and consider the orbit space $\mathcal{S}/\!\operatorname{Sim}(\mathbb{E}^2)$. For each equivalence class $[V]$, we choose a representative $V_0 \in [V]$ and fix a $G_{V_0}$-fixed point $o_{V_0} \in C(V_0 \cup \{\infty\})^\circ$. For any $V_1 \in [V]$, there exists $A \in \operatorname{Sim}(\mathbb{E}^2)$ such that $V_1 = A(V_0)$. We then define 
$$
o_{V_1} = A(o_{V_0}).
$$
This assignment is well-defined: if $V_1 = B(V_0)$ for another similarity $B \in \operatorname{Sim}(\mathbb{E}^2)$, then $B^{-1}A \in G_{V_0}$. Since $o_{V_0}$ is fixed by $G_{V_0}$, we have
$$
A(o_{V_0}) = B\bigl(B^{-1}A(o_{V_0})\bigr) = B(o_{V_0}).
$$
By construction, the assignment satisfies $o_{A(V)} = A(o_V)$ for every Euclidean similarity $A$. We then define $Q_{S,V,d}$ to be the inverse of the hyperbolic radial projection from $\partial C(Y)$ to $\mathbb{E}^2 \setminus V$ centered at $o_V$.
\end{enumerate}

    The equivariance condition \eqref{equivq1} for $Q_{S,V,d}$ follows directly, as all constructions in the spherical setting are equivariant under hyperbolic isometries, while those in the Euclidean setting are equivariant under hyperbolic isometries fixing $\infty$.
\end{proof}

\section{Vertex scaling and Wu's example}
\subsection{Discrete conformal equivalence and vertex scaling }\label{appen:vs}

There are several equivalent formulations of the notion of discrete conformal equivalence
for polyhedral surfaces given in Definition \ref{dcdef}. See, for instance, \cite{luo2004, glsw, bps} and the references therein. In this appendix, we  recall the \underline{vertex scaling} definition, which will be used to construct Tianqi Wu's example.

Fix a triangulation $\T$ of a surface $S$ with vertex set $V =V(\T)$. A triangulated surface $(S, V, \T, d)$ with a Euclidean (resp. hyperbolic or spherical) cone metric can be represented by the edge length function $l_d: E(\T) \to \R_{>0}$. We consider $l_d$ as a coordinate for $d$.
One defines two Euclidean (resp. hyperbolic or spherical) edge length functions $l_d$ and $l_{d'}$ on $(S, V, \T)$ to be related by a \underline{vertex scaling} if there exists a function $\lambda: V \to \R_{>0}$ such that
\begin{equation}\label{vs}
    \begin{aligned}
& l_{d'}(e)= \lambda(v)\lambda(v')l_d(e) \\
\Bigl(\text{resp.}\quad
& \sinh\left(\frac{l_{d'}(e)}{2}\right)= \lambda(v)\lambda(v')
\sinh\left(\frac{l_d(e)}{2}\right) \\
\text{or}\quad & \sin\left(\frac{l_{d'}(e)}{2}\right)
= \lambda(v)\lambda(v') \sin\left(\frac{l_d(e)}{2}\right)\Bigr),
\end{aligned}
\end{equation}
where $v,v'$ are the endpoints of the edge $e$. For simplicity, we will write $\lambda*l_d$ for $l_{d'}$.

The vertex scaling operation $\lambda*l_d$ is closely related to the length cross-ratio. 
Recall that for an oriented Euclidean (resp. hyperbolic or spherical) quadrilateral $Q=\square v_iv_kv_jv_m$ with diagonal $e=v_kv_m$, its \underline{length cross-ratio}  $lcr(Q, e)$ is defined to be
\begin{equation}\label{eq:lcr}\tag{2}
   \frac{l_{ik}l_{jm}}{l_{im}l_{jk}} \quad \Bigl(\text{resp.} \quad \frac{\sinh(l_{ik}/2)\sinh(l_{jm}/2)}{\sinh(l_{im}/2)\sinh(l_{jk}/2)} \quad \text{or} \quad \frac{\sin(l_{ik}/2)\sin(l_{jm}/2)}{\sin(l_{im}/2)\sin(l_{jk}/2)} \Bigr).
   \end{equation}
   It is equal to the absolute value of the cross-ratio of its four vertices in $\C$.  For each edge $e$ in  a triangulated cone metric surface $(S, V, \T, l)$, its length cross-ratio $lcr_l(e)$ is defined to be the length cross-ratio $lcr(\tau\cup\sigma, e)$ where $\tau$ and $\sigma$ are the two triangles adjacent to $e$.
It can be shown easily that two triangulated cone metric surfaces $(S, V, \T, l)$ and $(S, V, \T, l')$ are related by a vertex scaling if and only if their corresponding length cross-ratios are the same, i.e., $lcr_l(e) =lcr_{l'}(e)$, for all edges $e$.

 The length cross-ratio $lcr_l(e)$ is equal to the exponential of the shear coordinate of the two
 ideal triangles $C(v_i, v_k, v_m)$ and $C(v_j, v_k, v_m)$ along the geodesic $C(v_k, v_m)$ (see, for instance, \cite{glsw}). Specifically, the shear coordinate $Z(e)$ on the edge $e=v_kv_m$ satisfies $Z(e) = \log lcr_l(e)$. See Proposition 5.3.2 in \cite{bps}. Since the vertex scaling operation preserves the length cross-ratio, it does not change the shear coordinate and is closely related to the convex hull construction in the hyperbolic 3-space $\H^3$.  In particular, if $(S, V, \T, l)$ and $(S, V, \T', l')$ are two polyhedral surfaces with Delaunay triangulations $\T=\T'$ and edge length functions $l, l'$ such that $l$ and $l'$ are related by a vertex scaling, then these two polyhedral surfaces are discrete conformal since their associated hyperbolic metrics are isometric.
The relationship between vertex scaling and discrete conformal equivalence is the following theorem proved in \cite{glsw,gglsw,Izmestiev-Prosanov-Wu}.   Suppose $(S, V, d_1)$ and $(S, V, d_2)$ are two closed polyhedral surfaces. Then they are marked discrete conformal with the same background geometry if and only if there exists a sequence of Delaunay triangulated cone metric surfaces $(S, V, \T_i, l_i)$ for $i=1,2,..., n$ such that (1) $(S, V, \T_1, l_1) =(S, V, d_1)$ and $(S, V, \T_n, l_n) =(S, V, d_2)$, (2) if $\T_i =T_{i+1}$, then $l_{i+1} = w_i* l_i$ for some $w_i:V\to \R_{>0}$, and (3) if $\T_i \neq \T_{i+1}$, then $(S, V, \T_{i+1}, l_{i+1})$ is obtained from $(S, V, \T_i, l_i)$ by an edge flip.

In Appendix \ref{appendix:example}, Tianqi Wu constructed an example of a non-polyhedral standard model surface that is discrete conformal to a polyhedral surface. This prompts the following question.  It is known that there exists a closed discrete subset $V \subset \D$ with $\overline{V}-V =\partial \D$ such that $\partial C(V \cup \D^c)$ is not the union of  2-dimensional faces. More precisely, there exists a vertex $v \in V$ and a sequence of infinitely many distinct geodesics $L_i$ in $\H^3$ ending at $v$, each contained in $\partial C(V \cup \D^c)$.
These are the geodesic laminations in Thurston's sense (\cite{thurston, epsteinm2006}). 
We conjecture these examples will not occur in the polyhedral category.

\begin{conjecture}
    If $(\D, V, d_{\H})$ is a standard model surface discrete conformal to a polyhedral surface, then $\partial C(V \cup \D^c)$ is a union of 2-dimensional faces of $C(V \cup \D^c)$.
\end{conjecture}

\begin{figure}[ht!]
\centering
\begin{overpic}[scale=0.88]{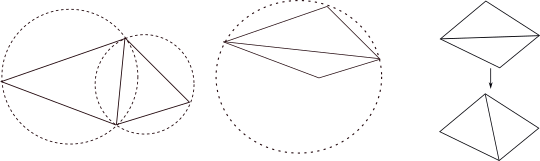}
      \put(22,25){$v_i$}
        \put(22,3){$v_j$}
        \put(-4,12){$v_k$}
        \put(36,10){$v_m$}
        \put(10,6){$l_{jk}$}
         \put(10,20){$l_{ik}$}
        \put(30,18){$l_{im}$}
        \put(28,6){$l_{jm}$}
        \put(6,13.5){$\theta_{k}^{ij}$}
         \put(28,11){$\theta_{m}^{ij}$}
        
        \put(60,30){$v_i$}
        \put(40,24){$v_k$}
        \put(58,12){$v_j$}
        \put(72,20){$v_m$}         
         \put(48,26){$l_{ik}$}
        \put(66,26){$l_{im}$}
         \put(48,14){$l_{jk}$}
        \put(66,14){$l_{jm}$}

\end{overpic} \hspace{.7cm}

\caption{The Delaunay condition $\theta_k^{ij} + \theta_m^{ij}\leq \pi$, a disk that is not an empty-disk, and an edge flip.  
}
\label{edp}
\end{figure}

\subsection{Wu's  example}\label{appendix:example}
The following example, constructed by Tianqi Wu \cite{wupc}, 
shows that there exists a standard model surface  $(\E^2, V, d_\E)$ that is not polyhedral but is discrete conformal to a polyhedral surface $(S, V, d)$. Similar examples can be constructed for standard model surfaces with hyperbolic background metrics.

\begin{figure}[htbp]
\begin{center}
  \begin{overpic}[width=\linewidth]{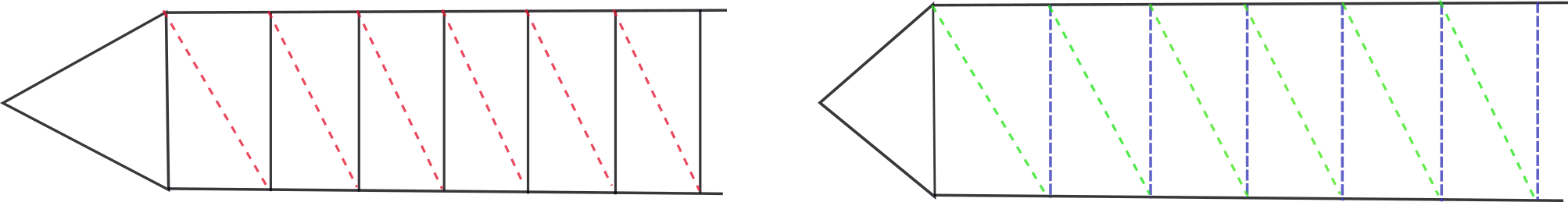}
    \put(3,0){$\sqrt{2}$}
    \put(3,10){$\sqrt{2}$}
    \put(12,9.5){$\frac{1}{\sqrt{2}}$}
    \put(12,3){$\frac{1}{\sqrt{2}}$}
    \put(7,5){$\sqrt2$}
    \put(53,0){$\sqrt{2}$}
    \put(53,10){$\sqrt{2}$}
    \put(62,9.5){$\frac{1}{\sqrt{2}}$}
    \put(62,3){$\frac{1}{\sqrt{2}}$}
    \put(57,5){$2$}
\end{overpic}
\caption{Left: two front parts of $(S, V, d)$ corresponding to $A$ and $B$; Right: the back part of $(S, V, d)$ corresponding to $C$.}
\label{poly3piece}
\end{center}
\end{figure}

We construct $(S, V, d)$ by gluing pieces of Euclidean triangles and rectangles with prescribed sizes shown in Figure \ref{poly3piece}. The left piece is constructed by gluing an equilateral triangle of side length $\sqrt{2}$ with an infinite sequence of triangulated rectangles of side lengths $\sqrt{2}$ and $1/\sqrt{2}$. The right piece is constructed by gluing an isosceles right triangle of side lengths $2$ and $\sqrt{2}$ with an infinite sequence of triangulated rectangles of side lengths $2$ and $1/\sqrt{2}$. We mark the longer sides of these rectangles, which have length $2$, with dashed lines.

Take two identical copies $A$ and $B$ of the left piece and one copy $C$ of the right piece and glue them along the boundary. The left-hand picture of Figure \ref{poly3piece2} shows $(S, V, d)$ with a triangulation $\mathcal{T}$. It is straightforward to check that $\mathcal{T}$ is a Delaunay triangulation since the sum of the two angles opposite each edge is at most $\pi$. Let $\tilde l:E(\mathcal{T})\to \R$ denote the edge length function of $(S, V, \mathcal{T}, l)$.  One can also readily verify that $(S, V, d)$ is polyhedral.

We claim that $(S, V, d)$ is discrete conformal to the standard model $(\E^2, V, d_\E)$ with
$V =\{\pm(2^n, 0), (0, 2^n) : n \in \mathbb Z_{\geq 0 }\} \cup \{(0,0)\}$. This is clearly not a polyhedral surface, since it cannot be triangulated geometrically with $V$ as its vertex set. However, we can construct a topological tessellation of $(\E^2, V, d_\E)$ that is combinatorially equivalent to $\mathcal{T}$, as shown in Figure \ref{standardnopoly}. The edge length function $l: E(\mathcal{T})\to \R$ can be computed directly from the coordinates of vertices in $V$.

\begin{figure}[htbp]
\begin{center}
  \begin{overpic}[width=0.65\linewidth]{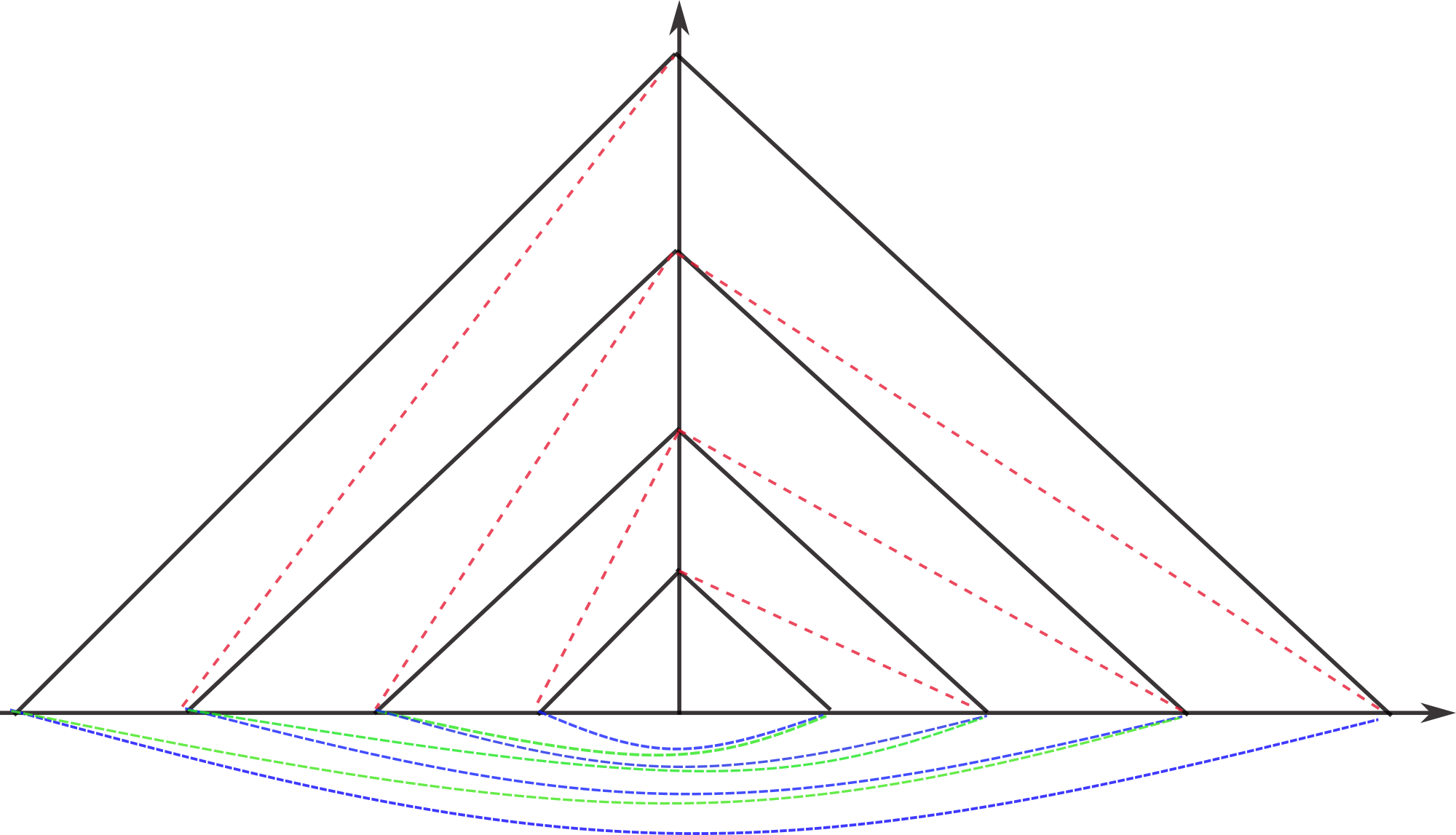}
  \put(44,7){$0$}
  \put(44,17){$1$}
  \put(44,27){$2$}
  \put(44,39){$4$}
  \put(44,52){$8$}
  \put(57,7){$1$}
  \put(68,7){$2$}
  \put(81,7){$4$}
  \put(91,7){$8$}
  \put(33,7){$-1$}
  \put(21,7){$-2$}
  \put(9,7){$-4$}
  \put(-3,7){$-8$}
\end{overpic}
\caption{A non-polyhedral standard model surface discrete conformal to a polyhedral surface.}  
\label{standardnopoly}
\end{center}
\end{figure}
 
A conformal factor $\lambda: V \to \R_{>0}$ can be defined as follows.    Set $\lambda(0) = \sqrt{2}$, and for $n \in \Z_{\geq 0}$,  $$\lambda(\pm (2^n, 0)) = \lambda((0, 2^n))= 2^{-n/2}.$$ 
The three pieces $A$, $B$, and $C$ of $(S, V, d)$ correspond to the first quadrant, the second quadrant and the (closed) lower half-plane in $\E^2$, respectively. A simple calculation using the vertex scaling formula $\tilde l_{ij} = \lambda_i\lambda_j l_{ij}$ shows that the resulting surface is isometric to the surface obtained by gluing $A$, $B$, and $C$ isometrically along their boundaries, namely, $(S, V, d)$.
The diagonals of rectangles and isosceles trapezoids are also related by $\lambda$ via the vertex scaling formula. This follows from Ptolemy’s identity for cyclic quadrilaterals and the fact that the diagonals of   isosceles trapezoids have the same length.  Since $(S,V,d)$ and $(\E^2, V,d_\E)$ are related by a vertex scaling, their associated complete hyperbolic metrics are isometric. 

Indeed, the shear coordinates of the associated hyperbolic metric of $(S, V, d)$ can be computed directly from the formula $Z(e) = \log lcr(e)$.  The convex-hull hyperbolic metric associated with $(\E^2, V, d_\E)$ is induced from $\partial C(V\cup\{\infty\})$. Its shear coordinates can also be computed from $\mathcal{T}$. In Figure \ref{poly3piece2}, $\partial C(V\cup\{\infty\})$ contains a face with infinitely many vertices; this face is triangulated according to the triangulation of $C$ in Figure \ref{poly3piece}. The shear coordinates are computed by the same formula $Z(e) = \log lcr(e)$. Since $\tilde l_{ij} = \lambda_i\lambda_j l_{ij}$, the two surfaces have the same length cross-ratio for every edge in $\mathcal{T}$, and hence their associated hyperbolic metrics are isometric. Therefore, $(S, V, d)$ is discrete conformal to the standard model surface $(\E^2, V, d_\E)$.

\begin{figure}[htbp]
\begin{center}
    \includegraphics[width=0.2\textwidth]{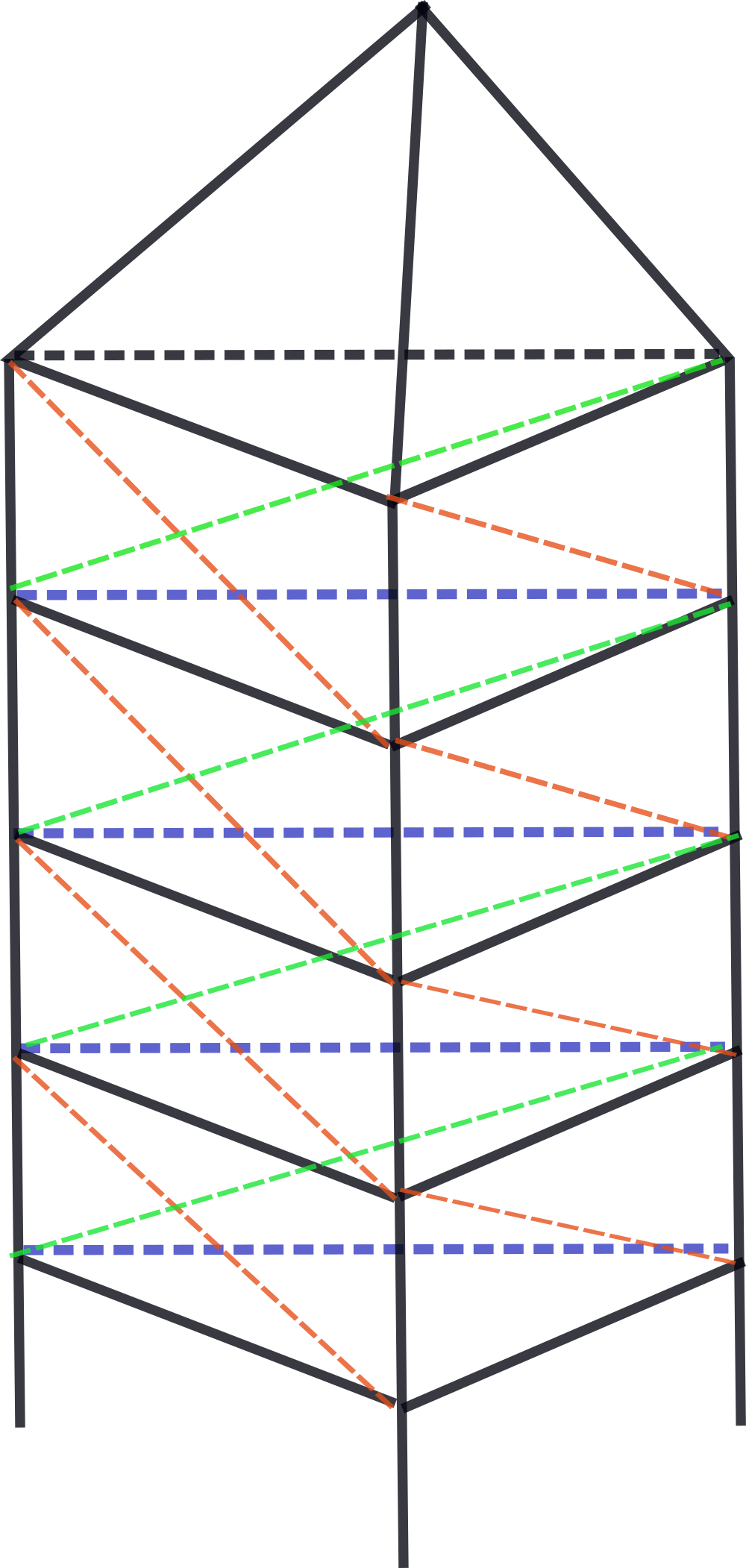}
    \qquad
    \includegraphics[width=0.6\textwidth]{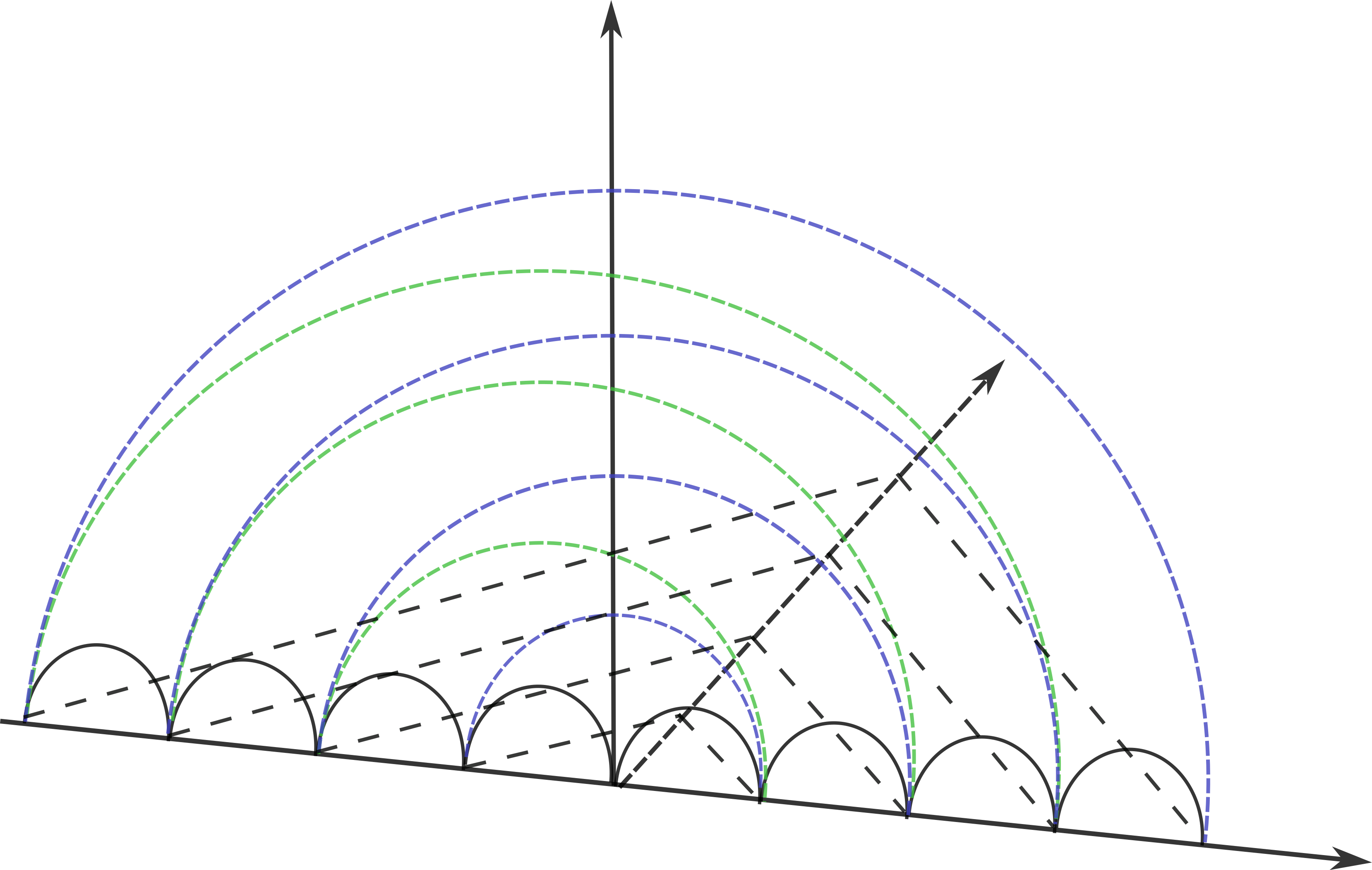}
\caption{The polyhedral surface $(S, V, d)$ and the infinite vertical face of $\partial C(V\cup\{\infty\})$.}
\label{poly3piece2}
\end{center}
\end{figure}

By the rigidity theorem, $(\E^2, V, d_\E)$ is the unique standard model discrete conformal to $(S, V,  d)$. 
This example shows that it is necessary to include non-polyhedral standard model surfaces in Theorem \ref{dut}.

\section{Polyhedral surfaces with incomplete $d^*$}\label{appen:incomplete}

In this section, we provide examples of polyhedral surfaces with Euclidean and spherical background metrics whose associated hyperbolic metrics $d^*$ are incomplete. Since standard model surfaces always have complete associated hyperbolic metrics, these examples show the sharpness of the assumptions on uniform bounds for the diameters of circumdisks in Theorem \ref{dut} (ii) and (iii).

\subsection{An example with a Euclidean background metric}

We now construct an example of a polyhedral surface $(S, V, d)$ with Euclidean background metric whose associated hyperbolic metric $d^*$ is incomplete. 
The surface $S$ is chosen to be the universal cover of the punctured plane $\E^2-\{(0,0)\}$ with the flat metric. We construct a closed discrete subset $V$ such that the circumdisks in a Delaunay triangulation $\T$ of $(S,V,d)$ have unbounded diameters, while the associated hyperbolic surface $(S-V,d^*)$ contains a curve of finite length that escapes every compact set. 

Let $S=(0,\infty)\times\mathbb{R}$
with coordinates $(\rho,\theta)$ be the universal cover of $\E^2-\{(0,0)\}$ with the covering map $\Dev(\rho,\theta)=(\rho\cos\theta,\rho\sin\theta)$.
Then $(S,d)$ is a flat surface with the induced metric $d=\Dev^*(d_\E)$.
We also use the map $\Dev$ as the developing map of $S$ into the Euclidean plane.
Let
\begin{equation}
    V=\left\{ a_{n,k}=\left(2^{k},\frac{n\pi}{2}\right):n,k\in\Z, |k|\geq |n|+2\right\}
\end{equation}
be the set of vertices.
We show that $(S,V,d)$ is a polyhedral surface with incomplete associated hyperbolic metric $d^*$. 

We note that each $a_{n,k}$ lies on the line $L_n=\{(\rho,{n\pi}/{2}):\rho>0\}$, and we let 
\begin{equation}
    W_n=\{(\rho,\theta):\rho>0,\ \frac{n\pi}{2}\le\theta\le\ \frac{(n+1)\pi}{2}\}
\end{equation}
be the wedge with boundary $L_n$ and $L_{n+1}$.
Then the restriction of $\Dev$ to each $W_n$ is an embedding, and the image $U_n:=\Dev(W_n)$ is some quadrant in $\E^2$. 

\begin{figure}[ht]
\centering
\begin{tikzpicture}[scale=0.95, line cap=round, line join=round]
    \draw (0,0) -- (7.5,0) node[right] {$\Dev(L_n)$};
    \draw (0,0) -- (0,7.5) node[above] {$\Dev(L_{n+1})$};

    \coordinate (s_x)  at (1.4, 0);   
    \coordinate (s_y)  at (0, 1.4);   
    \coordinate (r_x)  at (2.8, 0);   
    \coordinate (R_x)  at (4.2, 0);   
    \coordinate (2R_x) at (5.6, 0);   
    \coordinate (S_y)  at (0, 5.6);   

    \draw (0.35,0) -- (0,0.35);
    \draw (0.7,0)  -- (0,0.7);
    \draw (1.4,0)  -- (0,1.4);
    \draw (0.7,0)  -- (0.35,0);
    \draw (0,0.7)  -- (0,0.35);
    \draw (1.4,0)  -- (0.7,0);
    \draw (0,1.4)  -- (0,0.7);
    \draw[densely dashed] (0.7,0) -- (0,0.35);
    \draw[densely dashed] (1.4,0) -- (0,0.7);

    \draw (r_x) -- (s_x) -- (s_y) -- cycle;

    \fill[gray!10] (r_x) -- (s_y) -- (S_y) -- (R_x) -- cycle;
    \draw[thick] (r_x) -- (s_y) -- (S_y) -- (R_x) -- cycle;
    \draw[thick, densely dashed] (r_x) -- (S_y);

    \draw (R_x) -- (2R_x) -- (S_y) -- cycle;

    \coordinate (M1_x) at (6.8, 0); 
    \coordinate (M1_y) at (0, 6.8);
    \draw (2R_x) -- (M1_x) -- (M1_y) -- (S_y);
    \draw[densely dashed] (2R_x) -- (M1_y);

    \draw (M1_x) -- (7.2,0);
    \draw (M1_y) -- (0,7.2);
    \draw[dotted] (6.8,0) -- (0,6.8);

    \fill (s_x)  circle (1.5pt) node[below=2pt] {$b_{n,-n-3}$};
    \fill (s_y)  circle (1.5pt) node[left=2pt]  {$b_{n+1,-n-3}$};
    \fill (r_x)  circle (1.5pt) node[below=2pt] {$b_{n,-n-2}$};
    \fill (R_x)  circle (1.5pt) node[below=2pt] {$b_{n,n+2}$};
    \fill (2R_x) circle (1.5pt) node[below=2pt] {$b_{n,n+3}$};
    \fill (S_y)  circle (1.5pt) node[left=2pt]  {$b_{n+1,n+3}$};

    \node at (1.2, 2.4) {$Q_n$};
    \node at (3.6, 1.5) {$T_n^{\mathrm{out}}$};
    \node at (1.3, 0.45) {$T_n^{\mathrm{in}}$};
    \node[align=center] at (0.6, 0.45) {\tiny inner\\ \tiny layers};
    \node[align=center] at (5.8, 5.8) {\footnotesize outer\\ \footnotesize layers};
\end{tikzpicture}
\caption{A Delaunay triangulation of one developed wedge $W_n$ for $n\ge0$. }
\label{fig:wedge}
\end{figure}

Since the surface $(S,d)$ is symmetric with respect to the line $\theta=0$, we will only consider the triangulation of the half-plane $(0,\infty)\times\mathbb{R}_{\geq0}$. Fix $n\in\Z_+$, and let $b_{n,k}=\Dev(a_{n,k})=(2^k\cos(n\pi/2),2^k\sin(n\pi/2))$. We note that for each $k\geq n+3$ or $k\leq -n-4$, the four points $b_{n,k}, b_{n,k+1},b_{n+1,k}$, and $b_{n+1,k+1}$ are the vertices of an isosceles trapezoid, so they are concyclic. In addition, we have a triangle $T_{n}^{in}$ with vertices $\{b_{n,-n-3}, b_{n,-n-2},b_{n+1,-n-3}\}$, a triangle $T_{n}^{out}$ with vertices $\{b_{n,n+2}, b_{n,n+3},b_{n+1,n+3}\}$, and a quadrilateral $Q_n$ with vertices $\{b_{n,-n-2},\allowbreak b_{n,n+2},\allowbreak b_{n+1,n+3},\allowbreak b_{n+1,-n-3}\}$, where the vertices of $Q_n$ are concyclic.
It is not hard to verify that the interior of the circumdisk of each of the cells above contains no vertex in $V$; hence, these cells form the Delaunay tessellation (and then a Delaunay triangulation) of $U_n$. See Figure \ref{fig:wedge}. 
Thus we have a triangulation of $(S,V,d)$.

On the other hand, for each cell $\tau\subset U_n$, let $D$ be its open circumdisk. Then the closure $\overline{D}$ is contained in $U_{n-1}\cup U_n\cup U_{n+1}$. Since the restriction of $\Dev$ to $W_{n-1}\cup W_n\cup W_{n+1}$ is an isometric embedding onto $U_{n-1}\cup U_n\cup U_{n+1}$, $\Dev$ has a well-defined inverse $\iota: \overline{D} \to S$, which is an isometric embedding. Since $D$ lies in the interior of $U_{n-1}\cup U_n\cup U_{n+1}$, we know that $\iota(D)$ lies in the interior of $W_{n-1}\cup W_n\cup W_{n+1}$. Moreover, since $D$ does not contain any vertices in $U_n$, we know that $\iota(D)$ does not contain any vertices in $W_n$. Therefore, $\iota(D)$ is an empty-disk.
Hence, $(S,V,d)$ is a polyhedral surface. In particular, the diameters of the circumdisks of $Q_n$ are unbounded.
 
We now consider the associated hyperbolic metric $d^*$ of $(S,V,d)$.
For every Euclidean triangle of $\mathcal{T}$, take the ideal hyperbolic triangle with the same three boundary points in $\partial\mathbb{H}^3_U=\mathbb{E}^2\cup\{\infty\}$.  The Euclidean gluing maps extend to hyperbolic isometries, producing the metric $d^*$ on $S-V$.
Since $\Dev$ is a global developing map, adjacent triangles will always be developed simultaneously in their actual relative positions in $\mathbb{E}^2$.  Thus the corresponding ideal triangles are glued along the common ideal geodesic in that simultaneous development.

For each $n$, the vertices of $Q_n$ are concyclic, and hence they are on the conformal boundary of some hyperbolic plane in $\H^3_U$. Then the four vertices form a convex ideal quadrilateral in $\H^3_U$, which is denoted by $Q_n^*$. 
Since $Q_{n-1}$ and $Q_{n}$ share the common edge connecting $b_{n,-|n|-2}$ and $b_{n,|n|+2}$, we know that $Q_{n-1}^*$ and $Q_{n}^*$ share the common side $E_n^*:=[b_{n,-|n|-2},b_{n,|n|+2}]$.
Let $t_n=2^{-|n|-2}$ and
\begin{equation}
    x_n=\left( \frac{2t_n}{1+t_n^2} \cos\frac{n\pi}{2}, \frac{2t_n}{1+t_n^2} \sin\frac{n\pi}{2}, \frac{1-t_n^2}{1+t_n^2}\right) \in \E^2\times\R_{>0}\cong\H^3_U. 
\end{equation}
One can verify that $x_n\in E_n^*$. 

\begin{figure}[htbp]
\centering
\begin{tikzpicture}[x=1.55cm,y=1.55cm,line cap=round,line join=round]
    \draw (-1.25,0) -- (5.25,0);
    \node[below right] at (4.85,-0.02) {$\Dev(L_n)$};
    \draw[-{Latex[length=2mm]}] (0,0) -- (0,2.25) node[left] {$z$};
    \draw[thick] (4,0) arc[start angle=0,end angle=180,radius=1.875];
    \fill (0.25,0) circle (1.4pt) node[below] {$b_{n,-|n|-2}$};
    \fill (4,0) circle (1.4pt) node[below] {$b_{n,|n|+2}$};
    \fill (0,1) circle (1.5pt) node[left] {$p$};
    \fill (0.48,0.8824) circle (1.5pt) node[right] {$x_n$};
    \node[align=center] at (2.45,1.5) {$E_n^*$};
\end{tikzpicture}
\caption{A 2-dimensional vertical cross-section. As $t_n\to0$, the point $x_n$ approaches the point $p$.}
\label{fig:hyperbolic-cross-section}
\end{figure}

Let $p=(0,0,1)\in\H^3_U$. Then
\begin{equation}
    d_{\H}(x_n,p)=2\arctanh t_n \leq \frac{2t_n}{1-t_n^2}<3t_n,
\end{equation}
since $t_n\leq 1/4$. See Figure \ref{fig:hyperbolic-cross-section}. Let $\gamma_n$ be the hyperbolic geodesic segment joining $x_n\in E_n^*$ to $x_{n+1}\in E_{n+1}^*$. Since $Q_n^*$ is convex, we have $\gamma_n\subset Q_n^*$. By the triangle inequality,
\begin{equation}\label{eq:gamma-bound-raw}
    \ell(\gamma_n)
    =d_{\mathbb{H}}(x_n,x_{n+1})
    \leq d_{\mathbb{H}}(x_n,p)+d_{\mathbb{H}}(p,x_{n+1})< 3t_n + 3t_{n+1}.
\end{equation}
Let 
\begin{equation}
    \gamma=\gamma_0*\gamma_1*\gamma_2*\cdots
\end{equation}
be the concatenation of the segments.
Since
\begin{equation}
    \sum_{n=0}^{\infty}\ell(\gamma_n)
    < 3\sum_{n=0}^{\infty}(t_n+t_{n+1}) = 3\sum_{n=0}^{\infty}(2^{-n-2}+2^{-n-3}) < \infty,
\end{equation}
the path $\gamma$ has finite total length in $(S-V,d^*)$. Note that $\gamma$ escapes every compact subset of $(S-V,d^*)$, and thus the metric $d^*$ is incomplete.

\subsection{An example with a spherical background metric}

We next construct an example of a polyhedral surface $(S, V, d)$ with a spherical background metric whose associated hyperbolic metric $d^*$ is incomplete. 
We use a family of spherical quadrilaterals as building blocks for $(S, V, d)$, which is topologically an infinite cylinder. We then show that the circumdisks of triangles in a Delaunay triangulation of $(S, V, d)$ are compactly contained in $(S, V, d)$, while the associated metric $d^*$ is not complete.

    Let $N=(0,0,1)$, and, for each $n\in\Z$, let $u_n=(\cos (n\pi/2),\sin (n\pi/2),0)$. These are points in $\SS^2$. Let $\epsilon_n=2^{-|n|-2}$, and we have the following sequences in $\SS^2$:
	\begin{equation}\label{eq:ABdef}
		A_n=(\sin\epsilon_n)\cdot u_n+(\cos\epsilon_n)\cdot N,
		\qquad
		B_n=(\sin\epsilon_n)\cdot u_n-(\cos\epsilon_n)\cdot N.
	\end{equation}
	As $|n|\to+\infty$, we have $A_n\to N$ and $B_n\to -N$.
	
	For any $n\in\mathbb Z$, let $Q_n=A_nA_{n+1}B_{n+1}B_n$ be the spherical quadrilateral with cyclically ordered vertices $A_n,A_{n+1},B_{n+1},$ and $B_n$,
	where consecutive vertices are joined by the unique shortest geodesic segment in $\SS^2$. 
    We triangulate $Q_n$ by the diagonal $A_nB_{n+1}$ to produce two triangles $\triangle A_nA_{n+1}B_{n+1}$ and $\triangle A_nB_{n+1}B_n$.
	
    We let $\mathbf n_n=(\sin\epsilon_{n+1})\cdot u_n+(\sin\epsilon_{n})\cdot u_{n+1}$ and $h_n=\sin\epsilon_{n}\sin\epsilon_{n+1}$.
	Then
	\begin{equation}\label{eq:four-on-circle}
		\langle \mathbf n_n,A_n\rangle
		=\langle \mathbf n_n,A_{n+1}\rangle
		=\langle \mathbf n_n,B_{n+1}\rangle
		=\langle \mathbf n_n,B_n\rangle
		=h_n.
	\end{equation}
	Therefore, the four points $A_n, A_{n+1}, B_n,$ and $B_{n+1}$ lie on the boundary circle of the open spherical disk $D_n:=\{x\in\SS^2:\langle \mathbf n_n,x\rangle>h_n\}$, which is the open circumdisk of $Q_n$. Note that the diameter of $D_n$ is strictly less than $\pi$, but converges to $\pi$ as $n\to\infty$.

	\begin{figure}[htbp]
    	\centering
    	\begin{tikzpicture}[scale=0.92, line cap=round, line join=round]
    		\foreach \y/\lab in {0/0,1/1,2/2,3/3,4/0} {
    			\draw (-0.3,\y)--(6.3,\y);
    			\node[left] at (-0.35,\y) {$j=\lab$};
    		}
    		\foreach \x in {0,1,2,3,4,5,6} {
    			\draw (\x,0)--(\x,4);
    		}
    		\foreach \x in {0,1,2,3,4,5} {
    			\foreach \y in {0,1,2,3} {\draw (\x,\y+1)--(\x+1,\y);}
    		}
    		\node[below] at (0,0) {$n-1$};
    		\node[below] at (1,0) {$n$};
    		\node[below] at (2,0) {$n+1$};
    		\node[below] at (3,0) {$n+2$};
    		\draw[very thick,-{Latex[length=3mm]}] (6.55,4) arc[start angle=90,end angle=-90,radius=2];
    		\node[right,align=left] at (8.6,2) {top and bottom\\rows are identified};
    		\node[above] at (3,4.25) {four cyclically glued strips};
    	\end{tikzpicture}
    	\caption{The construction of the quotient surface, where each square represents a spherical quadrilateral $Q_n^{(j)}$.}
    	\label{fig:strips}
    \end{figure}    

	Next, we glue these quadrilaterals together to construct an infinite cylinder as follows. 
    For each $n\in\Z$, we take four copies $Q_n^{(j)}=A_n^{(j)}A_{n+1}^{(j)}B_{n+1}^{(j)}B_n^{(j)}$, $j\in\Z/4\Z$, of $Q_n$. Then for each $j\in\Z/4\Z$, we glue $Q_n^{(j)}$ to $Q_{n+1}^{(j)}$ along the edge $A_{n+1}^{(j)}B_{n+1}^{(j)}$, for all $n$, and obtain a strip $P^{(j)}$ with vertices $\{A_n^{(j)},B_n^{(j)}:n\in\mathbb Z\}$. The two boundary polygonal lines of $P^{(j)}$ are
	$$
	\partial_+P^{(j)}=\bigcup_n A_n^{(j)}A_{n+1}^{(j)},
	\qquad
	\partial_-P^{(j)}=\bigcup_n B_n^{(j)}B_{n+1}^{(j)}.
	$$
	Since $d_{\mathbb S^2}(A_n,A_{n+1})=d_{\mathbb S^2}(B_n,B_{n+1}),$ we glue $\partial_+P^{(j)}$ and $\partial_-P^{(j+1)}$ via the identification $A_n^{(j)}A_{n+1}^{(j)}\sim B_n^{(j+1)}B_{n+1}^{(j+1)}$ for all $n\in\mathbb Z$, and then obtain a quotient surface $S$ with the induced spherical cone metric $d$ and the vertex set $\{v_{n,j}=[A_n^{(j)}]=[B_n^{(j+1)}]\}$. See Figure \ref{fig:strips} for an illustration. Let $\T$ be the resulting triangulation induced by the previously specified triangulation of each $Q_n$. It is not hard to see that $S\cong \R\times S^1$ and $\T$ is locally finite.

	It follows that $(S, V, d)$ is a polyhedral surface with a Delaunay triangulation $\mathcal{T}$, once we show that each circumdisk is compactly contained in $(S, V, d)$.	
	Fix $n\in\mathbb Z$ and $j\in\mathbb Z/4\mathbb Z$.
	Let $K_{n,j}$ be the finite triangulated patch consisting of $Q_{n-1}^{(j)}\cup Q_n^{(j)}\cup Q_{n+1}^{(j)}$
	together with the one triangle of $P^{(j+1)}$ adjacent to the upper edge of $Q_n^{(j)}$ and the one triangle of $P^{(j-1)}$ adjacent to the lower edge of $Q_n^{(j)}$.
    Let $\triangle A_n^{(j)}A_{n+1}^{(j)}C_n^{(j+1)}$ be the triangle in $Q_n^{(j+1)}$ that is adjacent to $A_n^{(j)} A_{n+1}^{(j)}$, and $\triangle B_n^{(j)}B_{n+1}^{(j)}C_n^{(j-1)}$ the triangle in $Q_n^{(j-1)}$ adjacent to $B_n^{(j)} B_{n+1}^{(j)}$. See Figure \ref{fig:finitepatch} for a combinatorial figure.
    We note that $Q_{n-1}^{(j)}\cup Q_n^{(j)}\cup Q_{n+1}^{(j)}$ is naturally embedded in $\SS^2$ with the vertices $A_{n-1}, A_n, A_{n+1},B_{n-1}, B_n$ and $B_{n+1}$ in \eqref{eq:ABdef}. Since $K_{n,j}$ is simply connected, this embedding can be extended to a developing map $\Dev_{n,j}:K_{n,j}\to \SS^2$ by gluing $\triangle A_n^{(j)}A_{n+1}^{(j)}C_n^{(j+1)}$ and $\triangle B_n^{(j)}B_{n+1}^{(j)}C_n^{(j-1)}$.
    One can then directly verify that $D_n$, viewed as the circumdisk of $Q_n^{(j)}$, is contained in the interior of $\Dev_{n,j}(K_{n,j})$. Hence the circumdisk of $Q_n^{(j)}$ is compactly contained in $(S,V,d)$. We omit the details of the computation.

	\begin{figure}[ht]
		\centering
		\begin{tikzpicture}[scale=1.05,line cap=round,line join=round]
			\coordinate (A0) at (0,1.2);
			\coordinate (A1) at (2,1.2);
			\coordinate (A2) at (4,1.2);
			\coordinate (A3) at (6,1.2);
			\coordinate (B0) at (0,0);
			\coordinate (B1) at (2,0);
			\coordinate (B2) at (4,0);
			\coordinate (B3) at (6,0);
			\coordinate (Cp) at (3,2.45);
			\coordinate (Cm) at (3,-1.25);
			\draw (A0)--(A1)--(A2)--(A3);
			\draw (B0)--(B1)--(B2)--(B3);
			\draw (A0)--(B0) (A1)--(B1) (A2)--(B2) (A3)--(B3);
			\draw (A0)--(B1) (A1)--(B2) (A2)--(B3);
			\draw (A1)--(Cp)--(A2);
			\draw (B1)--(Cm)--(B2);
			\foreach \P/\lab/\pos in {A0/$A_{n-1}^{(j)}$/above,A1/$A_n^{(j)}$/above left,A2/$A_{n+1}^{(j)}$/above right,A3/$A_{n+2}^{(j)}$/above,
				B0/$B_{n-1}^{(j)}$/below,B1/$B_n^{(j)}$/below left,B2/$B_{n+1}^{(j)}$/below right,B3/$B_{n+2}^{(j)}$/below,Cp/$C_n^{(j+1)}$/above,Cm/$C_n^{(j-1)}$/below} {
				\fill (\P) circle (1.4pt); \node[\pos] at (\P) {\lab};
			}
			\node at (1,0.62) {$Q_{n-1}^{(j)}$};
			\node at (3,0.62) {$Q_n^{(j)}$};
			\node at (5,0.62) {$Q_{n+1}^{(j)}$};
			\draw (A3)--(B3)
			node[midway,right=8pt] {$P^{(j)}$};
		\end{tikzpicture}
		\caption{The finite developed patch $K_{n,j}$.}
		\label{fig:finitepatch}
	\end{figure}

   Finally, we construct a path of finite total length that escapes every compact subset.
	For each spherical triangle of $\mathcal T$ in $\SS^2$, take the ideal hyperbolic triangle with the same three ideal vertices in the Poincar\'e ball model
	$\mathbb H^3_P=\{x\in\mathbb R^3:\|x\|<1\}$.  The spherical transition isometries extend to isometries of $\mathbb H^3_P$, so these ideal triangles glue together, inducing the metric $d^*_{\mathcal T}$ on $S-V$. We now construct a finite-length path escaping every compact set in $(S-V,d^*_{\mathcal T} )$.
	
	Fix one strip, say $P$ (for simplicity, we temporarily omit the superscript).  For each $n$, the four ideal points $A_n,A_{n+1},B_{n+1},B_n$
	lie on the circle $\partial D_n$.  Hence, the two ideal triangles associated with the two triangles in $Q_n$ are coplanar and form an ideal hyperbolic quadrilateral, denoted $Q_n^*$.
    Denote by $P^*$ the strip obtained by gluing all $Q_n^*$.
	Let $E_n^*=[A_n,B_n]$ be a side of $Q_n^*$, which is the common edge of $Q_{n-1}^*$ and $Q_n^*$, and the developments of these two quadrilaterals agree on that edge. The ideal geodesic $E_n^*$ meets the equatorial plane $\{x_3=0\}\subset\R^3$ at a unique point $x_n$. Denote by $\gamma_n$ the hyperbolic segment in $\H^3_P$ joining $x_n$ to $x_{n+1}$. See Figure \ref{spherical-quad}.

    \begin{figure}[!htbp]
    \begin{center}
      \begin{overpic}[width=0.4\linewidth]{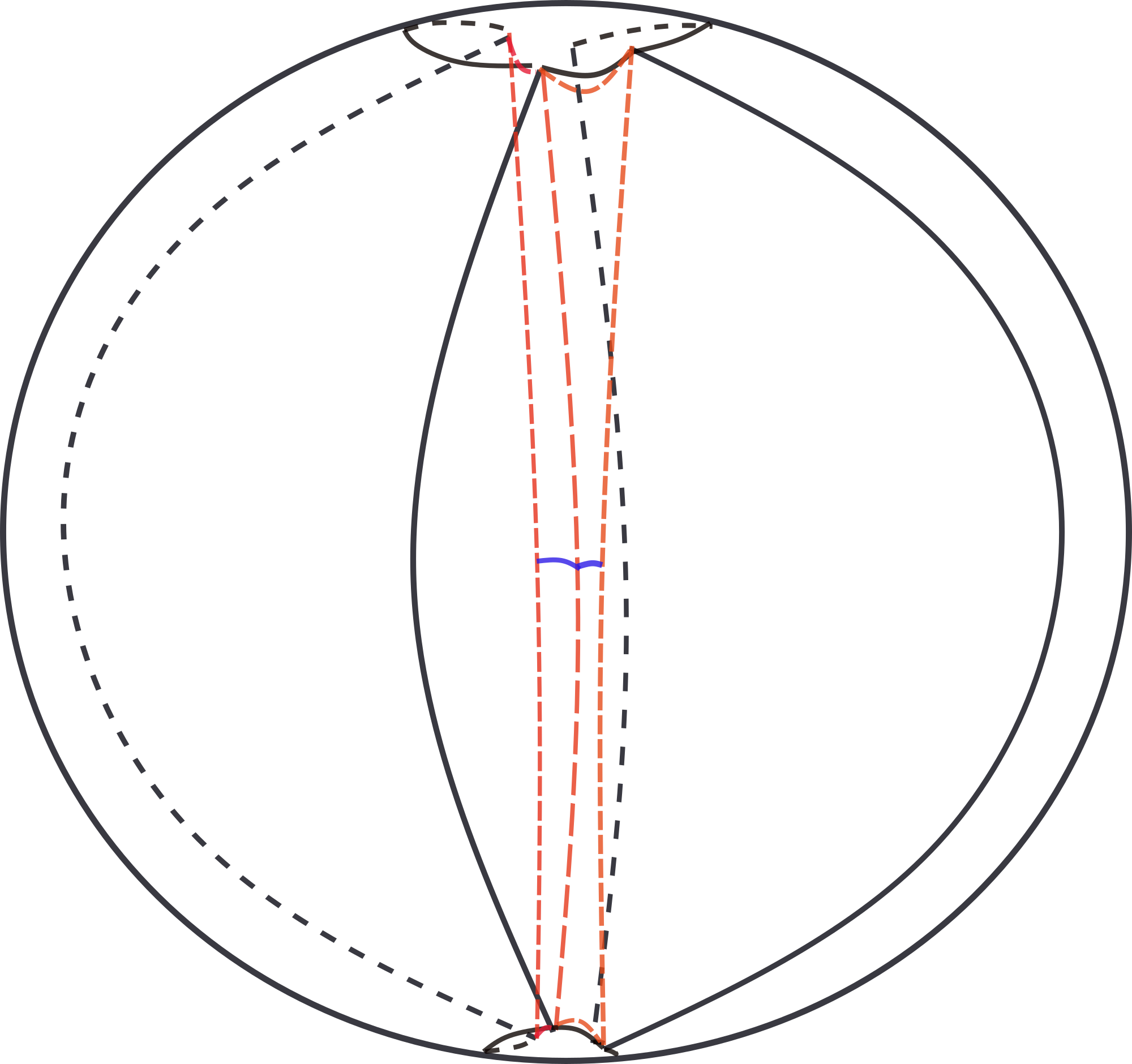}
        \put(60,50){$Q_n$}
        \put(94.2,50){$Q_{n+1}$}
        \put(22,50){$Q_{n-1}$}
        \put(40,30){$E^*_{n-1}$}
        \put(51,47){$\gamma_{n}$}
    
        \end{overpic}
    \caption{The development of $Q_{n-1}\cup Q_n\cup Q_{n+1}$.}
    \label{spherical-quad}
    \end{center}
    \end{figure}

    Let $O\in\H^3_P$ be the origin of $\R^3$. By basic hyperbolic geometry, we have
    \begin{equation}
        x_n=\left(\tan\frac{\epsilon_n}{2}\right)\cdot u_n
    \end{equation}
    and
    \begin{equation}
        d_{\H^3_P}(O,x_n)\leq2{\epsilon_n}.
    \end{equation}
    We omit the verification of the above identity and inequality.

    Since the ideal quadrilateral $Q_n^*$ is convex in the hyperbolic plane containing it, we have $\gamma_n\subset Q_n^*$. By the triangle inequality in $\mathbb H^3_P$,
    \begin{equation}
        \ell(\gamma_n)=d_{\mathbb H^3_P}(x_n,x_{n+1})
		\le d_{\mathbb H^3_P}(x_n,O)+d_{\mathbb H^3_P}(O,x_{n+1})\leq 2\epsilon_n+2\epsilon_{n+1}.
    \end{equation}
	Let 
\begin{equation}
    \gamma=\gamma_0*\gamma_1*\gamma_2*\cdots
\end{equation}
be the concatenation of the segments.
Since
\begin{equation}
    \sum_{n=0}^{\infty}\ell(\gamma_n)
    < 2\sum_{n=0}^{\infty}(\epsilon_n+\epsilon_{n+1}) = 2\sum_{n=0}^{\infty}(2^{-n-2}+2^{-n-3}) < \infty,
\end{equation}
the path $\gamma$ has finite total length in $(S-V,d^*)$. Note that $\gamma$ escapes every compact subset of $(P^*,d^*)$ and $P^*$ is not compact in $(S-V,d^*)$. Thus $(S-V,d^*)$ is incomplete.

\bibliographystyle{plain}
\bibliography{reference}
\end{document}

%% file: tikz-fig1.tex
\begin{figure}[!htbp]\label{fig:case-3-i}
\centering
\begin{subfigure}[b]{.497\linewidth}
\centering
\resizebox{\CasePanelWidth\linewidth}{!}{%
\begin{tikzpicture}
\PanelBoxA
\begin{scope}[x={(1.05cm,-.12cm)},y={(.62cm,.33cm)},z={(0cm,1.05cm)}]
  \BasePlane
  \foreach \j in {2,1,0,-1,-2}{\foreach \i in {-2,-1,0,1,2}{\LatticeTile{\i}{\j}}}
  \foreach \j in {-2,-1,0,1,2,3}{%
    \foreach \i in {-2,-1,0,1,2}{%
      \draw[mesh] plot[domain=0:180,samples=18]
        ({\i+.5+.5*cos(\x)},{\j},{.5*sin(\x)});
    }%
  }%
  \foreach \i in {-2,-1,0,1,2,3}{%
    \foreach \j in {-2,-1,0,1,2}{%
      \draw[mesh] plot[domain=0:180,samples=18]
        ({\i},{\j+.5+.5*cos(\x)},{.5*sin(\x)});
    }%
  }%
  \foreach \i in {-2,-1,0,1,2}{%
    \draw[surface rim] plot[domain=0:180,samples=20]
      ({\i+.5+.5*cos(\x)},{-2},{.5*sin(\x)});
  }
  \foreach \j in {-2,-1,0,1,2}{%
    \draw[surface rim] plot[domain=0:180,samples=20]
      ({3},{\j+.5+.5*cos(\x)},{.5*sin(\x)});
  }
  \foreach \i in {-2,-1,0,1,2,3}{\foreach \j in {-2,-1,0,1,2,3}{\VPoint{\i}{\j}}}
  \node[lab,text=surfaceedge,anchor=south] at (-1.35,1.35,1.08) {$\partial C(Y)$};
  \draw[red ray] (-.5,.5,.03)--(-.5,.5,.70);
  \RPoint{-.5}{.5}{0}\RPoint{-.5}{.5}{.707}
  \draw[red ray] (1.5,.5,.03)--(1.5,.5,.70);
  \RPoint{1.5}{.5}{0}\RPoint{1.5}{.5}{.707}
  \RPoint{.5}{-.5}{0}
  \draw[red ray] (.5,-.5,.03)--(.5,-.5,.70);
  \RPoint{.5}{-.5}{.707}
  \node[slab,anchor=north east] at (.40,-.62,0) {$\xi$};
\end{scope}
\end{tikzpicture}%
}
\caption{Case 3.1}
\end{subfigure}\hfill%
\begin{subfigure}[b]{.497\linewidth}
\centering
\resizebox{\CasePanelWidth\linewidth}{!}{%
\begin{tikzpicture}
\PanelBoxB
\begin{scope}[x={(1.05cm,-.12cm)},y={(.62cm,.33cm)},z={(0cm,1.05cm)}]
  \BasePlane
  \def\a{1.45}
  \def\H{3.45}
  \path[footprint] (-\a,-\a,0)--(\a,-\a,0)--(\a,\a,0)--(-\a,\a,0)--cycle;
  \path[face] (-\a,\a,\H)--(\a,\a,\H)
    --plot[domain=0:180,samples=32] ({\a*cos(\x)},{\a},{\a*sin(\x)})--cycle;
  \path[face] (-\a,-\a,\H)--(-\a,\a,\H)
    --plot[domain=0:180,samples=32] ({-\a},{\a*cos(\x)},{\a*sin(\x)})--cycle;
  \foreach \ii in {0,1,2,3,4,5,6,7}{%
    \foreach \jj in {0,1,2,3,4,5,6,7}{%
      \pgfmathsetmacro{\xa}{-\a+2*\a*\ii/8}
      \pgfmathsetmacro{\xb}{-\a+2*\a*(\ii+1)/8}
      \pgfmathsetmacro{\ya}{-\a+2*\a*\jj/8}
      \pgfmathsetmacro{\yb}{-\a+2*\a*(\jj+1)/8}
      \pgfmathsetmacro{\za}{sqrt(max(0,2*\a*\a-\xa*\xa-\ya*\ya))}
      \pgfmathsetmacro{\zb}{sqrt(max(0,2*\a*\a-\xb*\xb-\ya*\ya))}
      \pgfmathsetmacro{\zc}{sqrt(max(0,2*\a*\a-\xb*\xb-\yb*\yb))}
      \pgfmathsetmacro{\zd}{sqrt(max(0,2*\a*\a-\xa*\xa-\yb*\yb))}
      \path[surface cell] (\xa,\ya,\za)--(\xb,\ya,\zb)--(\xb,\yb,\zc)--(\xa,\yb,\zd)--cycle;
    }%
  }%
  \path[face] (-\a,-\a,\H)--(\a,-\a,\H)
    --plot[domain=0:180,samples=32] ({\a*cos(\x)},{-\a},{\a*sin(\x)})--cycle;
  \path[face] (\a,-\a,\H)--(\a,\a,\H)
    --plot[domain=0:180,samples=32] ({\a},{\a*cos(\x)},{\a*sin(\x)})--cycle;
  \draw[surface rim] plot[domain=-1.45:1.45,samples=36]
    (\x,{\a},{sqrt(max(0,\a*\a-\x*\x))});
  \draw[surface rim] plot[domain=-1.45:1.45,samples=36]
    (\x,{-\a},{sqrt(max(0,\a*\a-\x*\x))});
  \draw[surface rim] plot[domain=-1.45:1.45,samples=36]
    ({\a},\x,{sqrt(max(0,\a*\a-\x*\x))});
  \draw[surface rim] plot[domain=-1.45:1.45,samples=36]
    ({-\a},\x,{sqrt(max(0,\a*\a-\x*\x))});
  \foreach \x/\y in {-1.45/-1.45,1.45/-1.45,1.45/1.45,-1.45/1.45}{\VPoint{\x}{\y}}
  \CPoint{0}{0}\node[slab,anchor=north east] at (-.06,-.05,0) {$\xi_0$};
  \draw[guide] (0,0,0)--(0,-\a,0);
  \draw[red ray] (0,-\a,0)--(0,-3.00,0);
  \draw[red connector] (0,-\a,0)--(0,-\a,\a);
  \RPoint{0}{-\a}{0}\RPoint{0}{-\a}{\a}
  \draw[red ray] (0,-\a,\a)--(0,-\a,\RedRayTop);
  \draw[guide] (0,0,0)--(-\a,0,0);
  \draw[red ray] (-\a,0,0)--(-3.00,0,0);
  \draw[red connector] (-\a,0,0)--(-\a,0,\a);
  \RPoint{-\a}{0}{0}\RPoint{-\a}{0}{\a}
  \draw[red ray] (-\a,0,\a)--(-\a,0,\RedRayTop);
  \draw[guide] (0,0,0)--(\a,0,0);
  \draw[red ray] (\a,0,0)--(3.05,0,0);
  \draw[red connector] (\a,0,0)--(\a,0,\a);
  \RPoint{\a}{0}{0}\node[redlab,anchor=north west] at ({\a+.06},-.08,0) {$P(\xi)$};
  \RPoint{2.55}{0}{0}\node[redlab,anchor=north] at (2.55,-.08,0) {$\xi$};
  \draw[red ray] (\a,0,\a)--(\a,0,\RedRayTop);
  \RPoint{\a}{0}{\a}
  \node[redlab,anchor=south west] at ({\a+.10},0,{\a+.08}) {$R(P(\xi))$};
  \foreach \x/\y in {-1.45/-1.45,1.45/-1.45,1.45/1.45,-1.45/1.45}{%
    \draw[continuation] (\x,\y,2.85)--(\x,\y,3.62);
  }
\end{scope}
\end{tikzpicture}%
}
\caption{Case 3.2(1)}
\end{subfigure}

\par\vspace{.5mm}

\begin{subfigure}[b]{.497\linewidth}
\centering
\resizebox{\CasePanelWidth\linewidth}{!}{%
\begin{tikzpicture}
\PanelBoxC
\begin{scope}[x={(1.05cm,-.12cm)},y={(.62cm,.33cm)},z={(0cm,1.05cm)}]
  \BasePlane
  \draw[line width=2.72pt,draw=foot] (-2,0,0)--(2,0,0);
  \draw[hull edge] (-3.25,0,0)--(3.25,0,0);
  \VPoint{-2}{0}\VPoint{2}{0}
  \path[face] (-2,0,3.45)--(2,0,3.45)
    --plot[domain=0:180,samples=42] ({2*cos(\x)},0,{2*sin(\x)})--cycle;
  \draw[surface rim] plot[domain=0:180,samples=42]
    ({2*cos(\x)},0,{2*sin(\x)});
  \draw[red ray] (-2,0,0)--(-3.28,0,0);
  \draw[red ray] (-2,0,.02)--(-2,0,\RedRayTop);
  \draw[red ray] (2,0,0)--(3.28,0,0);
  \draw[red ray] (2,0,.02)--(2,0,\RedRayTop);
  \pgfmathsetmacro{\zs}{sqrt(3)}
  \draw[red ray] (-1,0,0)--(-2.28,1.28,0);
  \draw[red connector] (-1,0,0)--(-1,0,\zs);
  \RPoint{-1}{0}{0}\RPoint{-1}{0}{\zs}
  \draw[red ray] (-1,0,\zs)--(-1,0,\RedRayTop);
  \draw[red ray] (1,0,0)--(2.28,1.28,0);
  \draw[red connector] (1,0,0)--(1,0,\zs);
  \RPoint{1}{0}{0}\RPoint{1}{0}{\zs}
  \draw[red ray] (1,0,\zs)--(1,0,\RedRayTop);
  \draw[red ray] (0,0,0)--(0,2.12,0);
  \draw[red connector] (0,0,0)--(0,0,2);
  \RPoint{0}{0}{0}
  \draw[red ray] (0,0,2)--(0,0,\RedRayTop);
  \RPoint{0}{0}{2}
  \node[slab,anchor=west] at (2.48,1.48,0) {$H_+$};
  \node[slab,anchor=north east] at (-2.04,-.08,0) {$p_1$};
  \node[slab,anchor=north west] at (2.04,-.08,0) {$p_2$};
\end{scope}
\end{tikzpicture}%
}
\caption{Case 3.2(2)}
\end{subfigure}\hfill%
\begin{subfigure}[b]{.497\linewidth}
\centering
\resizebox{\CasePanelWidth\linewidth}{!}{%
\begin{tikzpicture}
\PanelBoxD
\begin{scope}[x={(1.05cm,-.12cm)},y={(.62cm,.33cm)},z={(0cm,1.05cm)}]
  \BasePlane
  \path[footprint] (-3.35,-1,0)--(3.35,-1,0)--(3.35,1,0)--(-3.35,1,0)--cycle;
  \foreach \n in {-3,-2,-1,0,1,2}{\StripWall{\n}{1}}
  \foreach \n in {-3,-2,-1,0,1,2}{\StripTile{\n}}
  \foreach \n in {-3,-2,-1,0,1,2}{\StripWall{\n}{-1}}
  \foreach \n in {-3,-2,-1,0,1,2,3}{%
    \draw[mesh] plot[domain=-1:1,samples=24] ({\n},\x,{sqrt(max(0,1-\x*\x))});
    \VPoint{\n}{-1}\VPoint{\n}{1}
  }
  \foreach \px in {-1.5,1.5}{%
    \draw[red ray] (\px,-1,0)--(\px,-2.04,0);
    \draw[red connector] (\px,-1,0)--(\px,-1,.5);
    \RPoint{\px}{-1}{0}\RPoint{\px}{-1}{.5}
    \draw[red ray] (\px,-1,.5)--(\px,-1,\RedRayTop);
  }
  \draw[red ray] (.5,-1,0)--(.5,-2.12,0);
  \draw[red connector] (.5,-1,0)--(.5,-1,.5);
  \RPoint{.5}{-1}{0}
  \draw[red ray] (.5,-1,.5)--(.5,-1,\RedRayTop);
  \RPoint{.5}{-1}{.5}
\end{scope}
\end{tikzpicture}%
}
\caption{Case 3.3(1)}
\end{subfigure}
\caption{The map $Q$ in Cases 3.1, 3.2(1), 3.2(2), and 3.3(1). Horizontal red rays are mapped isometrically to vertical red rays.}
\end{figure}

%% file: tikz-fig2.tex
\begin{figure}[!htbp]
\centering
\begin{subfigure}[b]{.497\linewidth}
\centering
\resizebox{\CasePanelWidth\linewidth}{!}{%
\begin{tikzpicture}
\PanelBoxE
\begin{scope}[x={(1.05cm,-.12cm)},y={(.62cm,.33cm)},z={(0cm,1.05cm)}]
  \BasePlane
  \def\aa{.62}
  \path[footprint] (0,0,0)--(3.45,{-\aa*3.45},0)--(3.45,{\aa*3.45},0)--cycle;
  \foreach \n in {0,1,2}{\WedgeWall{\n}{1}}
  \foreach \n in {0,1,2}{\WedgeWall{\n}{-1}}
  \foreach \n in {0,1,2,3}{\VPoint{\n}{\aa*\n}\VPoint{\n}{-\aa*\n}}
  \draw[construction] (2.05,0,.02)--(3.20,0,.02);
  \node[slab,text=construct!80!black,anchor=north] at (2.66,-.14,.03) {$w$};
  \pgfmathsetmacro{\zp}{.5*sqrt(1+\aa*\aa)}
  \draw[red ray] (.5,{.5*\aa},0)--(-1.18,{.5*\aa},0);
  \draw[red connector] (.5,{.5*\aa},0)--(.5,{.5*\aa},\zp);
  \RPoint{.5}{.5*\aa}{0}\RPoint{.5}{.5*\aa}{\zp}
  \draw[red ray] (.5,{.5*\aa},\zp)--(.5,{.5*\aa},3.52);
  \draw[red ray] (2.5,{-2.5*\aa},0)--(-.95,{-2.5*\aa},0);
  \draw[red connector] (2.5,{-2.5*\aa},0)--(2.5,{-2.5*\aa},\zp);
  \RPoint{2.5}{-2.5*\aa}{0}\RPoint{2.5}{-2.5*\aa}{\zp}
  \draw[red ray] (2.5,{-2.5*\aa},\zp)--(2.5,{-2.5*\aa},3.52);
  \draw[red ray] (1.5,{1.5*\aa},0)--(-1.25,{1.5*\aa},0);
  \draw[red connector] (1.5,{1.5*\aa},0)--(1.5,{1.5*\aa},\zp);
  \RPoint{1.5}{1.5*\aa}{0}
  \draw[red ray] (1.5,{1.5*\aa},\zp)--(1.5,{1.5*\aa},3.62);
  \RPoint{1.5}{1.5*\aa}{\zp}
\end{scope}
\end{tikzpicture}%
}
\caption{Case 3.3(2), $\Lambda\neq\{w\}$}
\end{subfigure}\hfill%
\begin{subfigure}[b]{.497\linewidth}
\centering
\resizebox{\CasePanelWidth\linewidth}{!}{%
\begin{tikzpicture}
\PanelBoxF
\begin{scope}[x={(1.05cm,-.12cm)},y={(.62cm,.33cm)},z={(0cm,1.05cm)}]
  \BasePlane
  \path[footprint] (0,-1,0)--(3.45,-1,0)--(3.45,1,0)--(0,1,0)--cycle;
  \foreach \n in {0,1,2}{\StripWall{\n}{1}}
  \path[face] (0,-1,3.45)--(0,1,3.45)
    --plot[domain=0:180,samples=30] (0,{cos(\x)},{sin(\x)})--cycle;
  \draw[surface rim] plot[domain=0:180,samples=30]
    (0,{cos(\x)},{sin(\x)});
  \foreach \n in {0,1,2}{\StripWall{\n}{-1}}
  \foreach \n in {0,1,2,3}{\VPoint{\n}{-1}\VPoint{\n}{1}}
  \draw[construction] (2.55,0,.02)--(3.35,0,.02);
  \node[slab,text=construct!80!black,anchor=north] at (2.96,-.13,.03) {$w$};
  \draw[line width=1.92pt,draw=construct!65] (0,-1,0)--(0,1,0);
  \CPoint{0}{0}
  \draw[construction] (0,0,0)--(1,-1,0);
  \draw[construction] (0,0,0)--(1,1,0);
  \draw[draw=groundedge!90,line width=.53pt] (1,-1,0)--(1,1,0);
  \CPoint{1}{0}
  \draw[guide] (1,0,0)--(1.5,-1,0);
  \draw[red ray] (1.5,-1,0)--(2.08,-2.16,0);
  \draw[red connector] (1.5,-1,0)--(1.5,-1,.5);
  \RPoint{1.5}{-1}{0}\RPoint{1.5}{-1}{.5}
  \draw[red ray] (1.5,-1,.5)--(1.5,-1,3.52);
  \draw[guide] (1,0,0)--(1.5,1,0);
  \draw[red ray] (1.5,1,0)--(2.08,2.16,0);
  \draw[red connector] (1.5,1,0)--(1.5,1,.5);
  \RPoint{1.5}{1}{0}\RPoint{1.5}{1}{.5}
  \draw[red ray] (1.5,1,.5)--(1.5,1,3.62);
  \draw[guide] (1,0,0)--(2.5,-1,0);
  \draw[red ray] (2.5,-1,0)--(3.45,-1.63,0);
  \draw[red connector] (2.5,-1,0)--(2.5,-1,.5);
  \RPoint{2.5}{-1}{0}\RPoint{2.5}{-1}{.5}
  \draw[red ray] (2.5,-1,.5)--(2.5,-1,3.52);
  \node[slab,text=construct!80!black,anchor=east] at (-.035,-.015,.055) {$\xi_0$};
  \node[slab,text=construct!80!black,anchor=north east] at (.99,-1.055,.01) {$\xi_1$};
  \node[slab,text=construct!80!black,anchor=south west] at (1.01,1.055,.055) {$\xi_2$};
  \node[slab,text=construct!80!black,anchor=south west] at (1.045,-.02,.16) {$\xi_0'$};
\end{scope}
\end{tikzpicture}%
}
\caption{Case 3.3(2), $\Lambda=\{w\}$ and $w'=w$}
\end{subfigure}

\par\vspace{.5mm}

\begin{subfigure}[b]{.497\linewidth}
\centering
\resizebox{\CasePanelWidth\linewidth}{!}{%
\begin{tikzpicture}
\PanelBoxG
\begin{scope}[x={(1.05cm,-.12cm)},y={(.62cm,.33cm)},z={(0cm,1.05cm)}]
  \BasePlane
  \draw[line width=2.72pt,draw=foot] (-3.35,0,0)--(3.35,0,0);
  \draw[hull edge] (-3.35,0,0)--(3.35,0,0);
  \foreach \n in {-3,-2,-1,0,1,2}{\CurtainCell{\n}}
  \foreach \n in {-3,-2,-1,0,1,2,3}{\VPoint{\n}{0}}
  \draw[continuation] (-2.80,0,3.45)--(-3.48,0,3.45);
  \draw[continuation] (2.80,0,3.45)--(3.48,0,3.45);
  \node[slab,anchor=west] at (2.48,1.48,0) {$H_+$};
  \foreach \px in {-1.5,1.5}{%
    \draw[red ray] (\px,0,0)--(\px,2.02,0);
    \draw[red connector] (\px,0,0)--(\px,0,.5);
    \RPoint{\px}{0}{0}\RPoint{\px}{0}{.5}
    \draw[red ray] (\px,0,.5)--(\px,0,3.52);
  }
  \draw[red ray] (.5,0,0)--(.5,2.12,0);
  \draw[red connector] (.5,0,0)--(.5,0,.5);
  \RPoint{.5}{0}{0}
  \draw[red ray] (.5,0,.5)--(.5,0,3.62);
  \RPoint{.5}{0}{.5}
\end{scope}
\end{tikzpicture}%
}
\caption{Case 3.4(1)}
\end{subfigure}\hfill%
\begin{subfigure}[b]{.497\linewidth}
\centering
\resizebox{\CasePanelWidth\linewidth}{!}{%
\begin{tikzpicture}
\ProcessPanelBox
  \begin{scope}[shift={(-5.05cm,-.55cm)},scale=1.02]
    \path[ground] (-1.75,0)--(1.75,0)--(1.75,2.25)--(-1.75,2.25)--cycle;
    \draw[hull edge,-{Latex[length=1.9mm]}] (0,0)--(-1.82,0);
    \draw[axis] (0,0)--(1.90,0) node[below,slab] {$x$};
    \draw[axis] (0,0)--(0,2.38) node[left,slab] {$y$};
    \draw[red ray] (0,0)--(-1.42,1.42);
    \draw[red ray] (0,0)--(0,2.02);
    \draw[red ray] (0,0)--(1.42,1.42);
    \draw[draw=rayred,line width=.58pt]
      plot[domain=0:45,samples=12] ({.54*cos(\x)},{.54*sin(\x)});
    \node[redlab] at (.60,.25) {$\theta$};
    \node[lab,anchor=north] at (0,-.31) {$H_+$};
  \end{scope}

  \draw[axis,line width=.72pt] (-2.85,.72)--(-2.22,.72);

  \begin{scope}[shift={(-1.90cm,-.55cm)},scale=1.02]
    \path[footprint] (0,0)--(2.25,0)--(2.25,2.25)--(0,2.25)--cycle;
    \draw[axis] (0,0)--(2.42,0);
    \draw[axis] (0,0)--(0,2.42) node[left,slab] {$t$};
    \draw[red ray] (0,0)--(.77,1.86);
    \draw[red ray] (0,0)--(1.43,1.43);
    \draw[red ray] (0,0)--(1.86,.77);
    \draw[draw=rayred,line width=.58pt]
      plot[domain=0:22.5,samples=12] ({.55*cos(\x)},{.55*sin(\x)});
    \node[redlab] at (.86,.14) {$\theta/2$};
    \node[lab,anchor=north] at (1.12,-.30)
      {$C_{\mathbb E}(V)\times\mathbb R_+$};
  \end{scope}

  \draw[axis,line width=.72pt] (.83,.72)--(1.46,.72);

  \begin{scope}[shift={(1.78cm,-.55cm)},x=.98cm,y=.98cm]
    \foreach \n in {0,1,2}{%
      \path[face] (\n,3.40)--({\n+1},3.40)
        --plot[domain=0:180,samples=28]
          ({\n+.5+.5*cos(\x)},{.5*sin(\x)})--cycle;
      \draw[surface rim] plot[domain=0:180,samples=28]
        ({\n+.5+.5*cos(\x)},{.5*sin(\x)});
    }
    \draw[axis] (0,0)--(3.30,0) node[below,slab] {$x$};
    \draw[axis] (0,0)--(0,3.65) node[left,slab] {$z$};
    \node[lab,text=surfaceedge,anchor=north] at (1.72,3.24) {$C(Y)$};
    \draw[red curve] plot[domain=0:1,samples=28]
      (\x,{sqrt(max(0,.25-(\x-.5)^2))+2.4142*\x});
    \draw[red curve end] plot[domain=1:1.16,samples=12]
      (\x,{sqrt(max(0,.25-(\x-1.5)^2))+2.4142*\x});
    \draw[red curve] plot[domain=0:1,samples=24]
      (\x,{sqrt(max(0,.25-(\x-.5)^2))+\x});
    \draw[red curve] plot[domain=1:2,samples=24]
      (\x,{sqrt(max(0,.25-(\x-1.5)^2))+\x});
    \draw[red curve end] plot[domain=2:2.62,samples=18]
      (\x,{sqrt(max(0,.25-(\x-2.5)^2))+\x});
    \draw[red curve] plot[domain=0:1,samples=24]
      (\x,{sqrt(max(0,.25-(\x-.5)^2))+.4142*\x});
    \draw[red curve] plot[domain=1:2,samples=24]
      (\x,{sqrt(max(0,.25-(\x-1.5)^2))+.4142*\x});
    \draw[red curve end] plot[domain=2:3,samples=24]
      (\x,{sqrt(max(0,.25-(\x-2.5)^2))+.4142*\x});
  \end{scope}
\end{tikzpicture}%
}
\caption{Case 3.4(2)}
\end{subfigure}
\caption{The map $Q$ in Cases 3.3(2), 3.4(1), and 3.4(2).}
\label{fig:case-3-ii}
\end{figure}